\documentclass[10pt]{article}

by-\headheight \advance \topmargin by-\headsep 
\usepackage{graphicx}
\usepackage{amsmath,amssymb,amsthm,amsfonts,mathrsfs}
\usepackage{amssymb}
\usepackage{color}
\usepackage{colortbl}
\newtheorem{thm}{Theorem}[section]
\newtheorem{cor}[thm]{Corollary}
\newtheorem{lem}[thm]{Lemma}

\theoremstyle{definition}

\theoremstyle{remark}
\newtheorem{rem}{Remark}[section]

\numberwithin{equation}{section}

\DeclareMathSymbol{\C}{\mathalpha}{AMSb}{"43}

\newcommand{\N}{{\mathbb N}}
\newcommand{\R}{{\mathbb{R}}}

\newcommand{\bsub}{\begin{subequations}}
\newcommand{\esub}{\end{subequations}$\!$}
\begin{document}

\title{  A coupled Hartree system with Hardy-Littlewood-Sobolev critical exponent: existence and multiplicity of high energy positive solutions}

%\author{ Lun Guo
%\thanks  {School of Mathematical Sciences, Fudan University, Shanghai 200433, China.
%Email: \texttt{lguo@mails.ccnu.edu.cn}.}
%}
\author{Mengyao Chen \thanks{College of Science, Wuhan University of Science and Technology,  430065,  Wuhan, China.
Email: \texttt{cmy@mails.ccnu.edu.cn}.}
\qquad
Lun Guo\thanks {
School of Mathematics and Statistics, South-Central University for Nationalities, Wuhan 430074,  China.
Email: \texttt{lguo@mails.ccnu.edu.cn}.}
\qquad
Qi Li\thanks{College of Science, Wuhan University of Science and Technology,  430065,  Wuhan, China.
Email: \texttt{qili@mails.ccnu.edu.cn}.}
}

%\date{\date}

\smallbreak \maketitle

\begin{abstract}
This paper deals with  a coupled Hartree system with Hardy-Littlewood-Sobolev critical exponent
\begin{equation*}
  \begin{cases}
    -\Delta u+(V_1(x)+\lambda_1)u=\mu_1(|x|^{-4}*u^{2})u+\beta (|x|^{-4}*v^{2})u, \ \  &x\in \R^N, \\
    -\Delta v+(V_2(x)+\lambda_2)v=\mu_2(|x|^{-4}*v^{2})v+\beta (|x|^{-4}*u^{2})v, \ \  &x\in \R^N, \\
  \end{cases}
\end{equation*}
where $N\geq 5$, $\lambda_1$, $\lambda_2\geq 0$ with $\lambda_1+\lambda_2\neq 0$, $V_1(x),  V_{2}(x)\in L^{\frac{N}{2}}(\R^N)$ are nonnegative functions and $\mu_1$, $\mu_2$, $\beta$ are positive constants. Such system arises from mathematical models in Bose-Einstein condensates theory and nonlinear optics. By variational methods combined with degree theory, we prove some results about the existence and multiplicity of high energy positive  solutions under the hypothesis $\beta>\max\{\mu_1,\mu_2\}$.\\
%{\textbf{2010 AMS Subject Classification:}} 35J60; 58J20.
%Primary ; secondary .
%\\

\end{abstract}

\vskip 0.05truein

\vskip 0.1truein

\section{Introduction and main results}

\qquad In this paper, we look for standing waves of the two-component nonlinear Hartree system
\begin{equation}\label{S-2}
\begin{cases}
  i\frac{\partial\Phi_1}{\partial t}=-\Delta\Phi_{1}+V_1(x)\Phi_1-\mu_1(K(x)*|\Phi_1|^{2})\Phi_1-\beta(K(x)*|\Phi_2|^{2})\Phi_1, \\
  % (x,t)\in \R^N \times  \R^{+}, \\
  i\frac{\partial\Phi_2}{\partial t}=-\Delta\Phi_{2}+V_2(x)\Phi_{2}-\mu_2(K(x)*|\Phi_2|^{2})\Phi_2-\beta(K(x)*|\Phi_1|^{2})\Phi_2,  \\
   %(x,t)\in \R^N \times  \R^{+} , \\
  (x,t)\in \R^N \times  \R^{+}, \,\, \Phi_{j}=\Phi_{j}(x,t) \in \mathbb{C}, \ \ \ \Phi_{j}(x,t)\to 0, \, \,   \text{as} \,\, |x|\to +\infty,  \,\, j=1,2,
\end{cases}
\end{equation}
where $i$ is the imaginary unit, $\Phi_j$ $(j=1,2)$ is the corresponding condensate amplitudes,  $K(x)$ is a nonnegative function which possesses information about the self-interaction between the particles, $\mu_j$ is the interspecies scattering length and coupling constant $\beta$ is the  intraspecies scattering length: $\mu_{j}>0$ corresponds to the attractive and $\mu_{j}<0$ to the repulsive self-interactions; similarly,  the coupling constant $\beta>0$ corresponds to the attraction and $\beta<0$ to the repulsion between the two components in  the system,   for more details we refer to  \cite{EGBB,T}. The system \eqref{S-2}  can also  be found in the studies of nonlinear optics  \cite{AA}. Physically, the solution $\Phi_j$ in system \eqref{S-2}  denotes the $j$-th component of the beam in Kerr-like photorefractive media. The positive constant $\mu_j$ indicates the self-focusing strength in the component of the beam, and  the coupling constant $\beta$ measures the interaction between the first and the second component of the beam.

%The system \eqref{S-2} also arises from nonlinear optics. Optical pulses propagating in a linear medium have a natural tendency to broaden in time (dispersion) and space (diffraction).  Such broadening can be eliminated  in a nonlinear medium that modifies its refractive index in the presence of light.  Self-trapping of incoherence beam occurs  in a nonlinear medium,  we refer the reader to \cite{AA}.

In order to find standing waves  of the form
\begin{equation}\label{Guo-add}
(\Phi_1(x,t),\Phi_2(x,t))=(e^{i\lambda_1t}u(x), e^{i\lambda_2t}v(x)),
\end{equation}
then we are led to study the  solutions of the following problem
\begin{equation}\label{eqs1.3}
  \begin{cases}
    -\Delta u+(V_1(x)+\lambda_1)u=\mu_1(K(x)*u^{2})u+\beta (K(x)*v^{2})u, \ \  &x\in \R^N, \\
    -\Delta v+(V_2(x)+\lambda_2)v=\mu_2(K(x)*v^{2})v+\beta (K(x)*u^{2})v, \ \  &x\in \R^N. \\
  \end{cases}
\end{equation}

For any $\beta\neq 0$, the system \eqref{eqs1.3} possesses a trivial solution $(0,0)$ and a pair of semi-trivial solutions with one component being zero, which have the form $(u,0)$ or $(0,v)$. Usually, we look for solutions of system \eqref{eqs1.3} which are different from the preceding ones.  A solution $(u,v)$ such that $u\neq 0, v\neq 0$, is non-trivial and  $u>0$, $v>0$, respectively positive solution. A solution is called a ground state solution if  its energy is minimal among the energy of all the non-trivial solutions of system \eqref{eqs1.3}.

%A solution  nontrivial if both $u\neq 0$ and  $v\neq 0$. We call a nontrivial solution  positive if both $u>0$  and $v>0$.

Let $K(x)$ be a Dirac-delta function, that is $K(x)=\delta(x)$, then  system \eqref{eqs1.3} transforms  into the following Schr\"{o}dinger system
\begin{equation}\label{S-1.3}
   \begin{cases}
    -\Delta u+(V_1(x)+\lambda_1)u=\mu_1u^3+\beta v^{2}u, \ \  x\in \R^N, \\
    -\Delta v+(V_2(x)+\lambda_2)v=\mu_2v^{3}+\beta u^{2}v, \ \  x\in \R^N. \\
  \end{cases}
  \end{equation}
Due to the important application in physics, the system \eqref{S-1.3} in low dimensions $(N=1,2,3)$ has been widely investigated. Not only the existence but also the qualitative properties of solutions for system \eqref{S-1.3} have been established by researchers via  variational methods, Lyapunov-Schmidt reduction method or bifurcation method,  see for example \cite{AC1,AC2,BD,BWW,CZ1,DL, DWW, LW1,LW3,LW4,MPi,NT,PW,SW,S1,WW2} and the references therein. Note that, if $N \leq 3$, the nonlinearity and  coupling terms in system \eqref{S-1.3} are of subcritical growth with respect to Sobolev critical exponent, the cubic nonlinearities and coupled terms are all of critical growth for $N=4$ and  super-critical growth  for $N\geq 5$. Here we only introduce some results  closely related to our paper. Chen and Zou \cite{CZ2,CZ3} considered the following Br\'{e}zis-Nirenberg type problem
\begin{equation}\label{CZ}
  \begin{cases}
    -\Delta u+ \lambda_1u=\mu_1 u^{2^*-1}+\beta u^{\frac{2^*}{2}-1}v^{\frac{2^*}{2}}, \,\, x\in \Omega \\
     -\Delta v+ \lambda_2v=\mu_2 v^{2^*-1}+\beta u^{\frac{2^*}{2}} v^{\frac{2^*}{2}-1},   \,\, x\in \Omega  \\
    u, v\geq 0  \,\, x\in \Omega, \ \  u=v=0 \,\, x\in \partial\Omega
  \end{cases}
\end{equation}
where $\Omega \subset \R^N$ is a smooth bounded domain, $2^*=\frac{2N}{N-2}$ is the critical Sobolev exponent, $-\lambda_1(\Omega)<\lambda_1,\lambda_2<0$, $\mu_{1},\mu_2 >0$ and $\beta\neq 0$.  They  established the existence, uniqueness and limit behaviour of positive ground state solutions in some ranges of $\lambda_1, \lambda_2, \beta$.  It turned out that results in the higher dimensions are quite different from those in $N=4$.  Recently, some existence and multiplicity results for a Coron type system \eqref{CZ} in a bounded domain with one or multiple small holes were  obtained in \cite{PPW,PS}. In contrast, there are  very few results for system \eqref{S-1.3} on the whole space $\R^4$.
%Peng et al. \cite{PPW} obtained  uniqueness of the least energy solution for $\beta>0$, and the non-degeneracy of the manifold of the synchronized positive solutions  to system  \eqref{S} with $\lambda_j=V_j(x)=0$ $(j=1,2)$ for the critical case in $\R^N(N\geq 3)$.
 Wu and Zou \cite{WZ}  proved the existence of positive ground state solutions for system \eqref{S-1.3} with steep potential wells. Moreover, they studied the phenomenon of phase separation of ground state solutions  as $\beta\to -\infty$.   The positive solutions of system \eqref{S-1.3} are  also considered at higher energy levels than the ground state energy level, see  \cite{CP1,LL} for example.  By considering the functional constrained on a subset of the Nehari manifold consisting of functions invariant with respect to a subgroup of  $O(N+1)$,   Clapp and Pistoia  \cite{CP1} proved that system \eqref{S-1.3} has infinitely many fully nontrivial solutions, which are not conformally equivalent. Let $\lambda_1=\lambda_2=0$ and $N=4$ in system \eqref{S-1.3},  Liu and Liu \cite{LL} proved  the existence of  positive solutions once that $\|V_1\|_{L^2}+\|V_2\|_{L^2}$ is small enough and $\beta>\max\{\mu_1, \mu_2\}$. It extends the celebrated work for semilinear Schr\"{o}dinger equation by Benci and Cerami \cite{BC} to Schr\"{o}dinger system \eqref{S-1.3}. Suppose that  $\lambda_1=\lambda_2>0$ in system \eqref{S-1.3}, the existence result in \cite{LL} was extended by  Guo et al. \cite{GLLM}. More novelly, by establishing a new type of global compactness result,  they overcame the loss of compactness of Palais-Smale sequence, and succeeded in proving the multiplicity  of positive solutions via variational methods.

In the following,  we are concerned about system \eqref{eqs1.3} with a Riesz potential response function $K(x)$,  more precisely, $K(x)=|x|^{-\alpha}$, $\alpha\in (0,N)$ and
 \begin{equation}\label{S-1.5}
   \begin{cases}
    -\Delta u+(V_1(x)+\lambda_1)u=\mu_1(|x|^{-\alpha}*u^{2})u+\beta (|x|^{-\alpha}*v^{2})u, \ \  x\in \R^N, \\
    -\Delta v+(V_2(x)+\lambda_2)v=\mu_2(|x|^{-\alpha}*v^{2})v+\beta (|x|^{-\alpha}*u^{2})v, \ \  x\in \R^N. \\
  \end{cases}
  \end{equation}
Clearly, semi-trivial solutions of system \eqref{S-1.5} correspond to solutions of Choquard-Pekar type equation
\begin{equation}\label{S-1.6}
 -\Delta u+(V_j(x)+\lambda_j)u=\mu_j(|x|^{-\alpha}*u^{2})u, \ \  x\in \R^N, \ \ j=1,2.
\end{equation}
Such a equation appeared as early as 1954, in a work by Pekar describing the quantum mechanics of apolaron at rest, see \cite{P}. We also refer to \cite{L,MPT} for more physical backgrounds. In the last few decades, the  equation \eqref{S-1.6} and its general form have been extensively investigated, not only  on subcritical case \cite{ANY,L,L1,MZ,M,MS1,MS2}, but also critical case \cite{AFM,DY,GSYZ,GY,GHPS}.

Compared with  a single equation, there are very few results available on the coupled Hartree system  \eqref{S-1.5}. The first attempt is due to Yang, Wei and Ding \cite{YWD},  they studied a singular perturbed problem related to system \eqref{S-1.5} and proved the existence of ground state solutions for $\beta$ large enough. Recently, Wang and Shi \cite{WS} proved  the existence and non-existence of ground state solutions to system  \eqref{S-1.5} with $\alpha=1$, $N=3$. The  existence of normalized solutions to  system \eqref{S-1.5} was proved by Wang and Yang \cite{WY} under suitable ranges for the parameters $\mu_1,\mu_2,\beta$. All the papers mentioned above deal with the subcritical case.  For the critical case $\alpha=4$ (the nonlocal terms in system \eqref{S-1.5} are critical  with respect to Hardy-Littlewood-Sobolev inequality, see Lemma 2.1 below), Gao et al. \cite{GLMY} proved that system \eqref{S-1.5} exists a positive solution with its functional energy lying in $(c_{\infty}, \min\{\frac{\mathcal{S}_{HL}^{2}}{4\mu_1}, \frac{\mathcal{S}_{HL}^{2}}{4\mu_2}, 2c_{\infty} \})$,  if  $\lambda_1=\lambda_2=0$ and  $V_{1}, V_2$ satisfy\\
$(C_1)$\,$V_{1}(x),V_{2}(x)\geq 0$, $\forall x\in \R^N$; \\
$(C_2)$\,$V_1(x), V_{2}(x)\in L^{\frac{N}{2}}(\R^N)\cap L_{loc}^{\infty}(\R^N)$;\\
$(C_3)$\,$V_1(x)$, $V_{2}(x)$ satisfy
\begin{align*}\label{Guo-New}
 0&< \frac{\beta-\mu_2}{2\beta-\mu_1-\mu_2}C(N,4)^{-\frac{1}{2}}\|V_1\|_{L^{\frac{N}{2}}}+\frac{\beta-\mu_1}{2\beta-\mu_1-\mu_2}C(N,4)^{-\frac{1}{2}}\|V_2\|_{L^{\frac{N}{2}}} \nonumber \\
  & < \min \Big\{\sqrt{\frac{\beta^2-\mu_1\mu_2}{\mu_1(2\beta-\mu_1-\mu_2)}}, \sqrt{\frac{\beta^2-\mu_1\mu_2}{\mu_2(2\beta-\mu_1-\mu_2)}}, \sqrt{2} \Big\} \mathcal{S}_{HL}-\mathcal{S}_{HL},
\end{align*}
where $c_{\infty}$ denotes by the ground state energy level and
$$
\mathcal{S}_{HL}:=\underset{u\in D^{1,2}(\R^N)\backslash\{0\}}{\inf}\frac{\int_{\R^N}|\nabla u|^2dx}{(\int_{\R^N}(|x|^{-4}*u^2)u^2dx)^{\frac{1}{2}}}.
$$
Note that $V_j\in L^\frac{N}{2}(\R^N)$ and $\lambda_j=0$, $j=1,2$,  imply that the only limit problem for \eqref{S-1.5} is
$$
  \begin{cases}
     -\Delta u=\mu_1(|x|^{-\alpha}*u^{2})u+\beta (|x|^{-\alpha}*v^{2})u, \ \  x\in \R^N, \\
    -\Delta v=\mu_2(|x|^{-\alpha}*v^{2})v+\beta (|x|^{-\alpha}*u^{2})v, \ \  x\in \R^N, \\
    u,v\in D^{1,2}(\R^N),
  \end{cases}
$$
 and this fact was a key ingredient in the approach of \cite{GLMY}.

On the other hand, $\lambda_j>0$ is very significant by its physical meaning, as one can readily seen in the ansatz \eqref{Guo-add}. In this paper,
we are ready  to investigate the existence and multiplicity  of positive  solutions for system \eqref{S-1.5} in such cases. Let us consider  the following coupled Hartree system with Hardy-Littlewood-Sobolev critical exponent 
\begin{equation}\label{S}
   \begin{cases}
    -\Delta u+(V_1(x)+\lambda_1)u=\mu_1(|x|^{-4}*u^{2})u+\beta (|x|^{-4}*v^{2})u, \ \  &x\in \R^N, \\
    -\Delta v+(V_2(x)+\lambda_2)v=\mu_2(|x|^{-4}*v^{2})v+\beta (|x|^{-4}*u^{2})v, \ \  &x\in \R^N. \\
  \end{cases}
  \end{equation}
Here and always in the sequel we assume that  $N\geq5$,  constants $\lambda_j\geq 0$, $\mu_{j}>0$,  $\beta \in \R  \setminus \{0\}$, $j=1,2$ and potentials $V_j(x)$ satisfy
 %$(A_1)$-$(A_3)$. We remark  that, if  $V_{j}(x)+\lambda_{j}\geq 0$, but $V_{j}(x)+\lambda_{j}\not\equiv 0$ for $j=1,2$, then it is easy to prove that system \eqref{S} does not have any ground state solutions (see Lemma \ref{lm2.6} below for computation details). Therefore we must look for solutions at higher energy levels than ground state energy level. {\color{blue}{In this paper,  we assume that  potential functions $V_1(x)$ and $V_2(x)$ satisfy the following assumptions:
\\
$(A_1)$ \, $V_{1}(x),V_{2}(x)\geq 0$, $\forall x\in \R^N$, \ $V_1(x)+V_2(x)\not\equiv 0$;
\\
$(A_2)$ \, $V_1(x), V_{2}(x)\in L^{\frac{N}{2}}(\R^N)\cap L_{loc}^{q}(\R^N)$, $q>\frac{N}{2}$;\\
$(A_3)$ \, $V_1(x)$, $V_{2}(x)$ verify
\begin{align*} \label{V}
  &\frac{\beta-\mu_2}{2\beta-\mu_1-\mu_2}\|V_1\|_{L^\frac{N}{2}}+\frac{\beta-\mu_1}{2\beta-\mu_1-\mu_2}\|V_2\|_{L^{\frac{N}{2}}} \nonumber \\
  & < \min \left\{\sqrt{\frac{\beta^2-\mu_1\mu_2}{\mu_1(2\beta-\mu_1-\mu_2)}},\,  \sqrt{\frac{\beta^2-\mu_1\mu_2}{\mu_2(2\beta-\mu_1-\mu_2)}},\,  \sqrt{2} \right\} \mathcal{S}-\mathcal{S},
\end{align*}
where $\mathcal{S}$ is the best constant in the  Sobolev embedding $D^{1,2}(\R^N)\hookrightarrow L^{2^*}(\R^N)$.

We remark  that, if  $V_{j}(x)+\lambda_{j}\geq 0$, but $V_{j}(x)+\lambda_{j}\not\equiv 0$ for $j=1,2$, then it is easy to prove that system \eqref{S} does not have any ground state solutions (see Lemma \ref{lm2.6} below for computation details). Therefore we must look for solutions at higher energy levels than ground state energy level. The main results of  our paper are stated as follows.

%Throughout this paper, we suppose that $V_1(x)$ and $V_2(x)$ satisfy the following assumptions:\\
%$(V_1)$ \, $V_{1}(x),V_{2}(x)\geq 0$, $\forall x\in \R^4$; \\
%$(V_2)$ \, $V_1(x), V_{2}(x)\in L^{2}(\R^4)\cap L_{loc}^{\infty}(\R^4)$. \\

%\begin{thm}\label{Th1.1}
%Let $\lambda=0$, $\beta>\max\{\mu_1,\mu_2\}$.  Assume that $V_1,V_2$ satisfy $(A_1)$-$(A_2)$ and
%\begin{align} \label{V}
%  0&< \frac{\beta-\mu_2}{2\beta-\mu_1-\mu_2}\|V_1\|_{L^2}+\frac{\beta-\mu_1}{2\beta-\mu_1-\mu_2}\|V_2\|_{L^2} \nonumber \\
%  & < \min \Big\{\sqrt{\frac{\beta^2-\mu_1\mu_2}{\mu_1(2\beta-\mu_1-\mu_2)}}, \sqrt{\frac{\beta^2-\mu_1\mu_2}{\mu_2(2\beta-\mu_1-\mu_2)}}, \sqrt{2} \Big\} \mathcal{S}-\mathcal{S}.
%\end{align}
%Then system \eqref{S} exists at least a positive  solution.
%\end{thm}

%Next, we study the multiplicity of bound state solutions for system \eqref{S} in the case $\beta>\max\{\mu_1,\mu_2\}$, which is  motivated by the work \cite{CP2} of Cerami and Passaseo.

\begin{thm}\label{Th1.2}
Let $\beta>\max\{\mu_1,\mu_2\}$, $\lambda_1, \lambda_2\geq 0$ and $\lambda:=\max\{\lambda_1, \lambda_2\}>0$.  \\
$(i)$\,Suppose that $V_1(x),V_2(x)$ satisfy $(A_1)$-$(A_2)$, then there exists $\lambda^*>0$, such that if $\lambda\in (0,\lambda^*)$,    system \eqref{S} exists at least a positive  solution. \\
$(ii)$\,Suppose that $V_1(x),V_{2}(x)$ satisfy $(A_1)$-$(A_3)$, then  there also exists $\lambda^{**}>0$ with $\lambda^{**}<\lambda^* $ such that if $\lambda \in (0, \lambda^{**})$,  system \eqref{S} has at least two distinct positive  solutions.
\end{thm}

\begin{rem}
In the proof of Theorem \ref{Th1.2}, the requirement $V_1(x), V_{2}(x)\in  L_{loc}^{q}(\R^N)$, $q>\frac{N}{2}$ is used only in
showing that a non-trivial non-negative solution of system \eqref{S} is indeed a positive solution of system \eqref{S} by the Harnack iquality in \cite{PSbook}.
\end{rem}

%\begin{rem}
%We point out that the existence result in Theorem \ref{Th1.2} still hold if $\lambda_1=\lambda_2=0$, which  is exactly  the main result of Theorem 1.3 in \cite{GLMY}. Hence, the results we obtained in Theorem \ref{Th1.2} can be seen as an improvement of the existence results by Gao et. al \cite{GLMY}.

%\end{rem}

 In the case $\lambda_1=\lambda_2=0$,  Gao et al. \cite{GLMY} obtained a high energy positive  solution of system \eqref{S} by using Linking theorem combined with a global compactness result. Compared with this, the problem we considered is much more difficult due to the existence of potential terms $\lambda_1u$ and $\lambda_2v$.
The major difficulties and challenges we faced  are:  first,  the main aim in our paper is to investigate the multiplicity of positive solutions of system \eqref{S}.   However, the classical  Linking theorem can only prove the existence of Palais-Smale sequences, but fail to obtain the multiplicity of Palais-Smale sequences. In the paper \cite{CP2}, Cerami and Passaseo developed a new variational trick and proved the multiplicity of solutions for scalar equation with Neumann boundary in half space $\R^N_{+}$ under similar hypotheses as we stated. Also,  the similar arguments are used  to study the other equations, see \cite{AFM,CM,GL} for example. Obviously, the methods developed for the scalar equation are not directly applicable to system \eqref{S}, several difficulties arise because of  the nonlocal  interaction terms  with critical exponent. Second,  dealing with such kind of critical elliptic system, we have to    overcome the loss of compactness and  distinguish the solutions from the semi-trivial ones.  To the best of our current knowledge, the only available compactness result associated with system \eqref{S} is given by Gao et al. \cite{GLMY}.  Unfortunately, the global compactness result in \cite{GLMY} can not work well on our problem, since the work space we chosen is totally different from the setting $D^{1,2}(\R^N)\times D^{1,2}(\R^N)$ in \cite{GLMY}. The loss of compactness makes the study of system \eqref{S} very complicated and substantially different, some new ideas must be introduced. In order to overcome the lack of compactness, we adapt some ideas from \cite{CP2} together with new techniques of analysis for Hartree system,  and give a complete description for the Palais-Smale sequences of the corresponding energy functional in the product spaces $H^{1}(\R^N)\times H^{1}(\R^N)$ and $H^{1}(\R^N)\times D^{1,2}(\R^N)$. This is the first attempt in this direction, and our result in this aspect is new and original.  To the end, adopting this description, we  succeed in proving the existence and multiplicity of  positive solutions via degree theory, which can be seen as an improvement of the existence results by Gao et. al \cite{GLMY}.

The paper is organized as follows. In section 2,  we introduce some preliminary  results and obtain a nonexistence result. In section 3, we establish a global compactness result and investigate the behavior of Palais-Smale sequences of $I$.  In section 4,  we make the proof of some technical lemmas that will be used in Section 5.  At end of this paper, we will prove our main results.   Our notations are standard. We will use $C$, $C_{i}$, $i\in\mathbb{N}$, to denote different positive constants from line to line.

%Our paper is organized as follows. In Sect. 2, we will  carry out the reduction procedure. Then, we will study the reduced finite dimensional problem and prove Theorem \ref{Th1.2}  in Sect. 3. Our notations are standard. We will use $C$ to denote different positive  constant from line to line.
%

\medskip
\section{Preliminary  results  and nonexistence result}
\setcounter{equation}{0}

%In  this section,  firstly we introduce  some notations which will be used frequently, then we will give some preliminary  results related to system \eqref{S}.

\qquad Let $X$ and $Y$ be two
 Banach spaces with norms $\|\cdot\|_{X}$ and $\|\cdot\|_{Y}$,   we define the standard norm on the product space $X\times Y$ as following
 $$
 \|(x,y)\|_{X\times Y}^{2}=\|x\|_{X}^{2}+\|y\|_{Y}^{2}.
 $$
 Particularly, if $X$ and $Y$ are two Hilbert spaces with inner products $\langle \cdot,\cdot\rangle_{X}$ and
$\langle\cdot,\cdot\rangle_{Y}$, then $ X\times Y$ is also a Hilbert space with the inner product
 $$
 \langle(x,y),(\varphi,\phi)\rangle_{X\times Y}=\langle x,\varphi\rangle_{X}+\langle y,\phi\rangle_{Y}.
 $$

Let the  Sobolev space $D^{1,2}(\R^N)$ be endowed with the norm and inner product respectively
$$\|u\|_{D^{1,2}}^{2}=\int_{\R^N}|\nabla u|^{2}dx,  \ \ \   \langle u, v\rangle_{D^{1,2}}= \int_{\R^N}\nabla u \cdot \nabla v dx.$$
Then we define the product space $H_{0}:=D^{1,2}(\R^N)\times D^{1,2}(\R^N)$ with norm
\begin{equation*}
  \|(u,v)\|_{H_0}^2:= \|(u,v)\|_{D^{1,2}\times D^{1,2}}^{2}=\int_{\R^N}\big(|\nabla u|^{2}+|\nabla v|^{2}\big)dx.
\end{equation*}
Recalling that the standard norm of $H^{1}(\R^N)$ be endowed with the norm
\begin{equation*}
  \|u\|_{H^1}^2:=\int_{\R^N}(|\nabla u|^{2}+u^{2})dx.
\end{equation*}
 Obviously, if $\lambda_1,\lambda_2>0$,  then $\|\cdot\|_{\lambda_j}$, $j=1,2$,  are equivalent to $\|\cdot\|_{H^1}$, where
  $$
 \|u\|_{\lambda_j}^2:=\int_{\R^N}(|\nabla u|^{2}+\lambda_j u^{2})dx.
 $$
In this paper, the work space we chosen depends on the sign of $\lambda_{j}$ ($j=1,2$).
Here and in the sequel we always require  $\lambda_1\geq0$, $\lambda_2 \geq 0$ and $\lambda:=\max\{\lambda_1,\lambda_2\}>0$.  Set
\begin{equation*}
  \|(u,v)\|_{H}^{2}: = \|u\|_{\lambda_1}^2+ \|v\|_{\lambda_2}^2,
  %=\int_{\R^N}\big(|\nabla u|^{2}+\lambda_1 u^{2}+|\nabla v|^{2}+\lambda_2 v^{2}\big)dx,
\end{equation*}
where  $H$ is defined by
\begin{equation*}
H=
\begin{cases}
H^1(\R^N)\times H^{1}(\R^N),\qquad &\text{if}\ \lambda_1>0,\lambda_2>0,\\
H^1(\R^N)\times D^{1,2}(\R^N),\qquad &\text{if}\ \lambda_1>0,\lambda_2=0,\\
D^{1,2}(\R^N)\times H^{1}(\R^N),\qquad &\text{if}\ \lambda_1=0,\lambda_2>0.
\end{cases}
\end{equation*}

%If $\lambda_1>0$ and $\lambda_2>0$, then we  will work in space $H:=H^{1}(\R^N)\times H^{1}(\R^N)$ endowed with norm
%\begin{equation*}
%  \|(u,v)\|_{H}^{2}:=\|(u,v)\|_{H^{1}\times H^{1}}^{2} =\int_{\R^N}\big(|\nabla u|^{2}+\lambda_1 u^{2}+|\nabla v|^{2}+\lambda_2 v^{2}\big)dx.
%\end{equation*}
%If $\lambda_j=0$ for some $j\in \{1,2\}$,  without loss of generality, we may assume $\lambda_1=0$ and $\lambda_2>0$. Then we  consider in $H:=D^{1,2}(\R^N)\times H^{1}(\R^N)$ endowed with
%\begin{equation*}
%  \|(u,v)\|_{H}^{2}:= \|(u,v)\|_{D^{1,2}\times H^{1}}^{2}=\int_{\R^N}\big(|\nabla u|^{2}+|\nabla v|^{2}+\lambda_2 %v^2\big)dx.
%\end{equation*}

Let $u^{+}:=\max\{0,u\}$ and $u^{-}:=\max\{0,-u\}$, $u=u^{+}-u^{-}$ and similarly for $v$, $v=v^{+}-v^{-}$. In order to obtain the positive solutions of system \eqref{S},  we are ready to  consider the modified   system
\begin{equation}\label{S-4}
  \begin{cases}
    -\Delta u+(V_1(x)+\lambda_1)u=\mu_1(|x|^{-4}*|u^+|^2)u^++\beta (|x|^{-4}*|v^+|^2)u^+, \ \ x\in \R^N, \\
    -\Delta v+(V_2(x)+\lambda_2)v=\mu_2(|x|^{-4}*|v^+|^2)v^++\beta (|x|^{-4}*|u^+|^2)v^+, \ \ x\in \R^N.
  \end{cases}
\end{equation}
If $(u,v)$ is a solution of system \eqref{S-4},  multiplying the first equation by $u^-$ and the second equation by $v^-$ in system \eqref{S-4} and integrating on $\R^N$, then we get
\begin{equation*}
  \int_{\R^N} \big(|\nabla (u^-)|^{2}+|\nabla (v^-)|^2\big)dx+\int_{\R^N}(V_1(x)+\lambda_1)(u^{-})^{2}dx
  +\int_{\R^N}(V_2(x)+\lambda_2)(v^{-})^{2}dx=0,
\end{equation*}
which implies that  $u\geq 0$, $v\geq 0$. By assumptions $(A_1)$-$(A_2)$ and the strong maximum principle,   we can prove $u>0$, $v>0$. Thus, $(u,v)$ is a positive solution of  system \eqref{S}. In the sequel, we focus our attention on the modified system \eqref{S-4} and shall to prove the  existence and multiplicity of non-trivial  solutions to system \eqref{S-4} by variational methods.
Next, we first introduce the well-known Hardy-Littlewood-Sobolev inequality, which will be frequently used,  see \cite{LL1}.

\begin{lem}
(Hardy-Littlewood-Sobolev inequality) Suppose  $\alpha \in (0,N)$, and $p$, $r>1$  with $ \frac{1}{p}+\frac{1}{r}+\frac{\alpha}{N}=2$. Let $f\in L^{p}(\R^N)$, $g\in L^{r}(\R^N)$, then there exists a sharp constant $C(p,r,\alpha,N)$, independent of $f$ and $g$, such that
\begin{equation}\label{add2.1}
\Big|\int_{\R^N}\int_{\R^N}\frac{f(x)g(y)}{|x-y|^{\alpha}}dxdy \Big| \leq C(p,\alpha,r,N)\|f\|_{L^p}\cdot \|g\| _{L^r},
\end{equation}
where $\|\cdot\|_{L^p}=\left(\int_{\R^N}|u|^{p}dx\right)^{\frac{1}{p}}$. If $p=r=\frac{2N}{2N-\alpha}$, then
\begin{equation*}%%%%\label{shuai-9-21-1}
C(p,r,\alpha,N)=C(N,\alpha)=\pi^{\frac{\alpha}{2}}\frac{\Gamma(\frac{N-\alpha}{2})}{\Gamma(\frac{2N-\alpha}{2})}
\Big\{\frac{\Gamma(\frac{N}{2})}{\Gamma(N)}\Big\}^{-\frac{N-\alpha}{N}}.
\end{equation*}
In this case, the equality in \eqref{add2.1} is achieved if and only if $f\equiv \text{(const.)}g$ and
\begin{equation*}
g(x)=A(\widetilde{\gamma}^{2}+|x-\widetilde{a}|^{2})^{-\frac{(2N-\alpha)}{2}}
\end{equation*}
for some $A\in \mathbb{C}$,  $\widetilde{a}\in \R^N$ and  $0\neq \widetilde{\gamma }\in \R$.
 \end{lem}

To study  system \eqref{S-4} by variational methods,  we define the energy functional $I(u,v): H \to \R$ by
\begin{align*}
  I(u,v)&=\frac{1}{2}\int_{\R^N}\Big(|\nabla u|^{2}+|\nabla v|^{2}+ \big(V_{1}(x)+\lambda_1 )u^2+(V_2(x)+\lambda_2 \big)v^2 \Big)dx \\
  & \quad -\frac{1}{4}\int_{\R^N}\int_{\R^N}\frac{\mu_1|u^+(x)|^2|u^{+}(y)|^2+2\beta |u^+(x)|^2|v^+(y)|^2+\mu_2 |v^+(x)|^2|v^{+}(y)|^2 }{|x-y|^4}dxdy.
\end{align*}
Since $V_1(x), V_2(x)\in L^{\frac{N}{2}}(\R^N)$, then in view of H\"{o}lder inequality and Hardy-Littlewood-Sobolev inequality,  $I$ is well defined in $H$ and belongs to $C^{1}(H,\R)$. In what follows, we define the Nehari manifold
\begin{equation*}
 \mathcal{ N}:=\{(u,v) \in {H} \mid (u,v)\neq (0,0), \langle I'(u,v),(u,v)\rangle=0\},
\end{equation*}
where
\begin{align*}
  \langle I'(u,v),(u,v) \rangle &=\int_{\R^N}\Big(|\nabla u|^{2}+|\nabla v|^{2}+ \big(V_{1}(x)+\lambda_1)u^2+(V_2(x)+\lambda_2\big)v^2 \Big)dx \\
  & \quad -\int_{\R^N}\int_{\R^N}\frac{\mu_1|u^+(x)|^2|u^{+}(y)|^2+2\beta |u^+(x)|^2|v^+(y)|^2+\mu_2 |v^+(x)|^2|v^{+}(y)|^2 }{|x-y|^4}dxdy.
\end{align*}
Moreover, we set
\begin{equation*}
  c=\inf_{(u,v)\in \mathcal{N}}I(u,v).
\end{equation*}

We remark that the existence of non-trivial solutions to  system \eqref{S-4} is closely related to the corresponding  problem
\begin{equation}\label{S2}
  \begin{cases}
    -\Delta u=\mu_1(|x|^{-4}*|u^+|^2)u^++\beta (|x|^{-4}*|v^+|^2)u^+, \ \ x\in \R^N, \\
    -\Delta v=\mu_2(|x|^{-4}*|v^+|^2)v^++\beta (|x|^{-4}*|u^+|^2)v^+, \ \ x\in \R^N.
  \end{cases}
\end{equation}
Throughout the paper we denote by $I_{\infty}: H_0\to \R$ the energy functional whose critical points are solutions of system \eqref{S2}, that is
\begin{align*}
  I_\infty(u,v)&=\frac{1}{2}\int_{\R^N}\big(|\nabla u|^{2}+|\nabla v|^{2}\big)dx\\
  &\quad -\frac{1}{4}\int_{\R^N}\int_{\R^N}\frac{\mu_1|u^+(x)|^2|u^{+}(y)|^2+2\beta |u^+(x)|^2|v^+(y)|^2+\mu_2 |v^+(x)|^2|v^{+}(y)|^2 }{|x-y|^4}dxdy.
\end{align*}
Also the analogues of  $\mathcal{N}$  and $c$ can be denoted by $\mathcal{N}_{\infty}$ and $c_{\infty}$  respectively, more precisely,
\begin{equation*}
 \mathcal{N}_{\infty}:=\{(u,v) \in H_0 \mid (u,v)\neq (0,0), \,\, \langle I'_{\infty}(u,v),(u,v)\rangle=0\}
\end{equation*}
and
\begin{equation*}
  c_{\infty}=\inf_{(u,v)\in \mathcal{N}_{\infty}}I_{\infty}(u,v).
\end{equation*}

To avoid semi-trivial solutions in searching for non-trivial solutions of system \eqref{S-4},  we first  investigate the single elliptic equation
\begin{equation*}
   -\Delta u+(V_j(x)+\lambda_j )u=\mu_j(|x|^{-4}*|u^+|^2)u^+, \ \ j=1,2,
\end{equation*}
of which the associated energy functional  $J_{j}:H^{1}(\R^N) \to \R$ (or $D^{1,2}(\R^N) \to \R$, if $\lambda_{j}=0$ ) by
\begin{equation*}
  J_{j}(u)=\frac{1}{2}\int_{\R^N}|\nabla u|^{2}dx+\frac{1}{2}\int_{\R^N}(V_{j}(x)+\lambda_j)u^2dx-\frac{\mu_j}{4}\int_{\R^N}\int_{\R^N}\frac{|u^+(x)|^2|u^{+}(y)|^2}{|x-y|^4}dxdy.
\end{equation*}
Consider the infimum
\begin{equation*}
  m_{j}=\inf_{u\in \mathcal{M}_{j}}J_{j}(u),
\end{equation*}
defined on the Nehari manifold
$$
\mathcal{M}_{j}=
\big\{ u\in H^{1}(\R^N)\mid u\neq 0, \langle J^{\prime}_{j}(u), u\rangle=0\big\}, \quad \text{if}\ \lambda_j>0,
$$
or
$$
\mathcal{M}_{j}=
\big\{ u\in D^{1,2}(\R^N)\mid u\neq 0, \langle J^{\prime}_{j}(u), u\rangle=0\big\}, \quad \text{if}\ \lambda_j=0.
$$
%\begin{equation*}
%\mathcal{M}_{i}=
%\Big\{ u\in H^{1}(\R^N)\mid u\neq 0, \int_{\R^N}(|\nabla u|^{2}+V_{i}(x)u^2+\lambda_i u^2)dx=\int_{\R^N}\int_{\R^N}\frac{\mu_i|u^+(x)|^2|u^{+}(y)|^2}{|x-y|^4}dxdy   \Big\}.
%\end{equation*}
If $V_j=\lambda_j=0$, we then denote the analogues of $J_j$, $m_j$, $\mathcal{M}_j$ by $J_{j}^{\infty}$, $m_{j}^{\infty}$, $\mathcal{M}_{j}^{\infty}$ respectively.

%In the sequel, we  first introduce some well-known results on infimum $\gamma_{j}^{\infty}$.
It is well-known that infimum $m_{j}^{\infty}$ is attained by
$$W_{\delta,z,j}(x)=\mu_{j}^{-\frac{1}{2}}U_{\delta,z}(x),$$
 where
\begin{equation}\label{e-2.2}
  U_{\delta,z}(x)=C_N\Big(\frac{\delta}{\delta^2+|x-z|^{2}}\Big)^{\frac{N-2}{2}}, \ \ \delta>0, \ \ z\in \R^N
\end{equation}
is the unique positive solution of
\begin{equation}\label{e-2.3}
  -\Delta u=(|x|^{-4}*u^{2})u,
\end{equation}
and constant $C_N$ only depends on dimension $N$, see \cite{DY,GHPS}. Moreover, by calculation we have
\begin{equation*}
  \int_{\R^N}|\nabla U_{\delta,z}(x)|^2dx=\int_{\R^N}\int_{\R^N}\frac{|U_{\delta,z}(x)|^{2}|U_{\delta,z}(y)|^2}{|x-y|^4}dxdy=\mathcal{S}^2_{HL}=C(N,4)^{-1}\mathcal{S}^2.
\end{equation*}
and
$$m_{j}^{\infty}=\frac{\mathcal{S}_{HL}^2}{4\mu_j}.$$

%Combining the results in \cite{BC,GLMY}, the following lemma can be established. Here, we only give a sketch of proof for the reader's convenience.
\begin{lem}\label{lm2.1}
Let $N\geq 5$ and nonnegative function $V_{j}(x)\in L^{\frac{N}{2}}(\R^N)$, then
$$
m_j=m_{j}^{\infty}=\frac{\mathcal{S}_{HL}^{2}}{4\mu_j}.
$$
\end{lem}
\begin{proof}
For arbitrary $u\in \mathcal{M}_j$, since $V_{j}(x)$ and $\lambda_{j}$  are both nonnegative, then
$$
J_{j}(u)\geq m_{j}^{\infty},
$$
which implies that $m_j\geq m_{j}^{\infty}$.

To show that the opposite inequality holds, we then consider the sequence
\begin{equation*}
\Phi_{n}(x)=\mu_{j}^{-\frac{1}{2}}U_{\delta_n,0}(x),
\end{equation*}
where $\delta_n\to 0^+$ as $n\to \infty$. On  one hand, since $N\geq 5$, then
%\begin{equation}\label{e-2.4}
%  \int_{\R^4}|\nabla \Phi_n|^{2}dx=\mu_{i}^{-1}\int_{\R^4}|\nabla U_{\delta_n,0} |^{2}dx+o_{n}(1),
%\end{equation}
%\begin{equation}\label{e-2.5}
%\int_{\R^4}|\Phi_n|^{4}dx=\mu_{i}^{-2}\int_{\R^4}|U_{\delta_n,0}|^{4}dx+o_{n}(1)
%\end{equation}
%and
\begin{equation*}
  \lambda_j \int_{\R^N}|\Phi_{n}|^{2}dx=\frac{\lambda_j \delta_{n}^{2}}{\mu_j}\int_{\R^N}|U_{1,0}(x)|^2dx=o_{n}(1),
\end{equation*}
in which we use the property below
\begin{align*}
  \int_{\R^N}|U_{1,0}(x)|^{2}dx&\leq C\int_{\R^N}\frac{1}{(1+|x|^{2})^{N-2}}dx \\
  &\leq  C\int_{B_{1}(0)}\frac{1}{(1+|x|^{2})^{N-2}}dx + C\int_{\R^N \setminus B_{1}(0)}\frac{1}{(1+|x|^{2})^{N-2}}dx \\
  & \leq C_{1}+ C\int_{\R^N\setminus B_{1}(0)}\frac{1}{|x|^{2N-4}}dx=C_{1}+C\int_{1}^{+\infty}\frac{1}{r^{N-3}}dr<C_{2}.
\end{align*}
On the other hand,  by H\"{o}lder inequality  we have for each $\rho>0$,
\begin{align*}
  \int_{\R^N}V_{j}(x)|\Phi_n|^{2}dx&=\int_{B_{\rho}(0)}V_{j}(x)|\Phi_n|^{2}dx+\int_{\R^N\setminus B_{\rho}(0)}V_{j}(x)|\Phi_n|^{2}dx \\
  &\leq \|\Phi_n\|_{L^{2^{*}}}^{2}\Big( \int_{B_{\rho}(0)}|V_{j}(x)|^{\frac{N}{2}}dx  \Big)^{\frac{2}{N}}+\|V_{j}\|_{L^{\frac{N}{2}}}\Big( \int_{\R^N\setminus B_{\rho}(0)}|\Phi_n(x)|^{2^{*}}dx  \Big)^{\frac{2}{2^*}}.
\end{align*}
Direct calculations show that
\begin{equation*}
  \lim_{n\to\infty}\int_{\R^N \setminus B_{\rho}(0)}|\Phi_n|^{2^{*}}dx=0, \ \ \ \  \sup_{n\in \N}\|\Phi_n\|_{L^{2^*}}<+\infty
\end{equation*}
and
\begin{equation*}
\lim_{\rho \to 0}\int_{B_{\rho}(0)}|V_{j}(x)|^{\frac{N}{2}}dx=0,
\end{equation*}
then
\begin{equation*}
  \lim_{n\to\infty}\int_{\R^N}V_{j}(x)|\Phi_n|^{2}dx=0.
\end{equation*}
Furthermore,
\begin{equation}\label{e-2.7}
 \lim_{n\to\infty}\int_{\R^N}(V_{j}(x)+\lambda_j )|\Phi_n|^{2}dx=0.
\end{equation}
For $t_n>0$ defined by
\begin{equation*}
t_{n}^{2}=\frac{\int_{\R^N}|\nabla \Phi_n|^{2}dx+\int_{\R^N}(V_{j}(x)+\lambda_i )|\Phi_{n}|^2dx}{\int_{\R^N}\int_{\R^N}\frac{\mu_{j}|\Phi_{n}(x)|^{2}|\Phi_n(y)|^2}{|x-y|^4}dxdy},
\end{equation*}
we can infer that $t_{n}\Phi_{n}\in \mathcal{M}_{j}$ and $t_n \to 1$ as $n\to \infty$. Therefore,  by \eqref{e-2.7}  we get
\begin{equation}\label{e-2.8}
 m_{j}\leq J_{j}(t_n \Phi_n)=\frac{t_{n}^{2}}{4}\Big( \int_{\R^N}|\nabla \Phi_{n}|^2dx+\int_{\R^N}(V_{j}(x)+\lambda_j )|\Phi_n|^{2}dx     \Big)
 = \frac{\mathcal{S}_{HL}^2}{4\mu_{j}}+o_{n}(1).
\end{equation}
Taking the limit $n\to \infty$ in \eqref{e-2.8}, we get
$$m_{j}\leq \frac{\mathcal{S}_{HL}^2}{4\mu_{j}}=m_{j}^{\infty}.$$
Thus, the whole proof is completed.
\end{proof}

\begin{cor}
\label{nontrivial}
Let $(u,v)\in \mathcal{N}$ be a critical point of $I$ constrained on $\mathcal{N}$.
Assume that
$$I(u,v)<\min\left\{\frac{\mathcal{S}_{HL}^2}{4\mu_{1}},\frac{\mathcal{S}_{HL}^2}{4\mu_{2}}\right\},$$
 then $u\not\equiv0$ and $v\not\equiv 0$.
\end{cor}

\begin{lem}\label{lm2.2}
  For arbitrary  $(u,v)\in H\setminus \{(0,0)\}$, let $\big( \tau_{(u,v)}u, \tau_{(u,v)}v \big)$, $\big(t_{(u,v)}u, t_{(u,v)} v \big)$ be the projections of $(u,v)$ on $\mathcal{N}_{\infty}$ and $\mathcal{N}$ respectively. Then
  \begin{equation*}
    \tau_{(u,v)}\leq t_{(u,v)}.
  \end{equation*}
\end{lem}
\begin{proof}
Since $V_{j}(x)$ and $\lambda_{j}$ are both nonnegative, then by
direct calculations we get
\begin{align*}
  \tau_{(u,v)}^{2}&=\frac{\int_{\R^N}|\nabla u|^{2}dx+\int_{\R^N}|\nabla v|^{2}dx}{\int_{\R^N}\int_{\R^N}\frac{\mu_1|u^+(x)|^2|u^{+}(y)|^2+2\beta |u^+(x)|^2|v^+(y)|^2+\mu_2 |v^+(x)|^2|v^{+}(y)|^2 }{|x-y|^4}dxdy} \\
  &\leq \frac{\int_{\R^N}|\nabla u|^{2}dx+\int_{\R^N}|\nabla v|^{2}dx+\int_{\R^N}(V_{1}(x)+\lambda_1)u^2dx+\int_{\R^N}(V_2(x)+\lambda_2 )v^2dx}{\int_{\R^N}\int_{\R^N}\frac{\mu_1|u^+(x)|^2|u^{+}(y)|^2+2\beta |u^+(x)|^2|v^+(y)|^2+\mu_2 |v^+(x)|^2|v^{+}(y)|^2 }{|x-y|^4}dxdy} \\
  & = t_{(u,v)}^{2},
\end{align*}
which implies that  $\tau_{(u,v)}\leq t_{(u,v)}$.

\end{proof}

In what follows, we always require $\beta>\max\{\mu_1,\mu_2\}$. Let
$$
k_{1}=\frac{\beta-\mu_2}{\beta^{2}-\mu_1\mu_2}, \ \ \ \ k_2=\frac{\beta-\mu_1}{\beta^{2}-\mu_1\mu_2},
$$
The next two lemmas are essentially proved  in \cite{GLMY} and we omit them.
\begin{lem}(\cite[Lemma 2.3]{GLMY})
If $\beta>\max\{\mu_1,\mu_2 \}$, then
$$c_{\infty}=\frac{1}{4}(k_1+k_2)\mathcal{S}_{HL}^{2}$$
and any ground state solutions of system \eqref{S2} must be of the form
\begin{equation*}
  (u,v)=(\sqrt{k_1}U_{\delta,z}, \sqrt{k_2}U_{\delta,z})
\end{equation*}
for some $\delta \in \R$ and $z\in \R^N$.
\end{lem}

\begin{lem}\label{lm2.4}(\cite[Corollary 3.5]{GLMY})
Let $\beta>\max\{\mu_1,\mu_2\}$. If $(u,v)\in H_0$ is a non-trivial classical positive solution of system \eqref{S2}, then we have
\begin{equation*}
  (u,v)=(\sqrt{k_1}U_{\delta,z}, \sqrt{k_2}U_{\delta,z})
\end{equation*}
for some $\delta>0$ and $z\in \R^N$. Moreover, each non-trivial classical positive solution  $(u,v)\in H_0$ of system \eqref{S2} is a ground state solution.
\end{lem}

 In \cite{GLMY}, Gao et al. proved the uniqueness result above by using of the method of moving spheres in integral form.  We would like to point out that the uniqueness result  plays an important role in proving the compactness of Palais-Smale sequence in a suitable energy interval.

%Under the hypotheses $(A_1)$-$(A_2)$, it is easy to find that  the infimum $c$ cannot be attained, so the system \eqref{S-4} does not have any ground state solutions. Therefore, we need to find nontrivial  solutions for  system \eqref{S-4} at a higher energy level than the ground state energy level.

\begin{lem}\label{lm2.6}
Suppose that $N\geq 5$, $\beta>\max\{\mu_1,\mu_2 \}$, $\lambda_1, \lambda_2\geq 0$ and $\lambda:=\max\{\lambda_1,\lambda_2\}>0$. If $V_{1}(x), V_{2}(x)$ satisfy assumptions $(A_1)$-$(A_2)$,
% and
%$$\|V_1\|_{L^{\frac{N}{2}}}+\|V_2\|_{L^{\frac{N}{2}}}>0,$$
 then
$c_{\infty}=c$
and $c$ is not attained.
\end{lem}

\begin{proof}
Let $\delta_n \to 0^+$ as $n\to \infty$, we define  $\tilde{\Phi}_{n}(x)=U_{\delta_n,0}(x)$. Following by the arguments in the proof of Lemma \ref{lm2.1}, we then get
\begin{equation}\label{add-guoguo}
 \int_{\R^N}(V_{1}(x)+\lambda_1 )|\tilde{\Phi}_{n}(x)|^2dx+\int_{\R^N}(V_{2}(x)+\lambda_2 )|\tilde{\Phi}_{n}(x)|^2dx =o_{n}(1).
\end{equation}
Let $t_n>0$ be such that
\begin{equation*}
t_{n}^{2}=\frac{(k_1+k_2)\int_{\R^N}|\nabla\tilde{\Phi}_n|^{2}dx+k_{1}\int_{\R^N}(V_{1}(x)+\lambda_1 )|\tilde{\Phi}_n|^{2}dx+k_{2}\int_{\R^N}(V_{2}(x)+\lambda_2 )|\tilde{\Phi}_n|^{2}dx}
{(k_1+k_2)\int_{\R^N}\int_{\R^N}\frac{|\tilde{\Phi}_n(x)|^{2}|\tilde{\Phi}_n(y)|^2}{|x-y|^4}dxdy},
\end{equation*}
then $(t_n\sqrt{k_1} \tilde{\Phi}_n,  t_n\sqrt{k_2} \tilde{\Phi}_n )\in \mathcal{N}$ and $t_n\to 1$ as $n\to \infty$, in which we use \eqref{add-guoguo} and
$$
\mu_{1}k_{1}^{2}+\mu_{2}k_{2}^{2}+2\beta k_1k_2=k_1+k_2.
$$
 Moreover,
\begin{align*}
c&\leq I(t_n\sqrt{k_1} \tilde{\Phi}_n,  t_n\sqrt{k_2} \tilde{\Phi}_n) \\
&=\frac{t_{n}^{2}(k_1+k_2)}{4}\int_{\R^N}|\nabla\tilde{\Phi}_n|^{2}dx+\frac{t_{n}^{2}k_{1}}{4}\int_{\R^N}(V_{1}(x)+\lambda_1 )|\tilde{\Phi}_n|^{2}dx
+\frac{t_{n}^{2}k_{2}}{4}\int_{\R^N}(V_{2}(x)+\lambda_2 )|\tilde{\Phi}_n|^{2}dx \\
&=\frac{1}{4}(k_1+k_2)\int_{\R^N}|\nabla\tilde{\Phi}_n|^{2}dx+o_{n}(1)\\
&=\frac{1}{4}(k_1+k_2)\mathcal{S}_{HL}^{2}+o_{n}(1).
\end{align*}
Taking the  limit $n\to \infty$ in the equality above, then
$$c\leq \frac{1}{4}(k_1+k_2)\mathcal{S}_{HL}^{2}=c_{\infty}.$$

 Let $(u,v)\in \mathcal{N}$ be arbitrarily chosen and set $\tau_{(u,v)}>0$ be such that $(\tau_{(u,v)}u, \tau_{(u,v)} v)\in \mathcal{N}_{\infty}$. Recalling that $\lambda_j$ and $V_j(x)$ are both nonnegative,    then by Lemma \ref{lm2.2} we get $\tau_{(u,v)}\leq 1$.
%\begin{align*}
%  t_{(u,v)}^{2}&=\frac{\int_{\R^4}|\nabla u|^{2}dx+\int_{\R^4}|\nabla v|^{2}dx}{\int_{\R^4}(\mu_1u^4+2\beta u^{2}v^2+\mu_{2}v^4)dx} \\
%  & \leq \frac{\int_{\R^4}|\nabla u|^{2}dx+\int_{\R^4}|\nabla v|^{2}dx+\int_{\R^4}(V_1(x)+\lambda_1)u^{2}dx+\int_{\R^4}(V_2(x)+\lambda_2)v^{2}dx}{\int_{\R^4}(\mu_1u^4+2\beta u^{2}v^2+\mu_{2}v^4)dx} \leq1.
%\end{align*}
Furthermore, by computation we get
\begin{align*}
  c_{\infty}&\leq I_{\infty}(\tau_{(u,v)}u, \tau_{(u,v)}v)\\
  &=\frac{1}{4}\tau_{(u,v)}^{2}\big(\int_{\R^N}|\nabla u|^{2}dx+ \int_{\R^N}|\nabla v|^{2}dx\big) \\
  &\leq \frac{1}{4}\big(\int_{\R^N}|\nabla u|^{2}dx+ \int_{\R^N}|\nabla v|^{2}dx\big)+ \frac{1}{4}\int_{\R^N}(V_1(x)+\lambda_1 )u^{2}dx+\frac{1}{4}\int_{\R^N}(V_2(x)+\lambda_2 )v^{2}dx \\
  &= I(u,v),
\end{align*}
which leads to $c_{\infty}\leq c$. Therefore, $c=c_{\infty}$.

In the end of the proof, we shall to prove that $c$  is not  attained. We argue by contradiction and assume that $c$ is attained by $(u_0,v_0)\in \mathcal{N}$.  It follows by Lemma \ref{lm2.2} that there exists $\tau_0\in (0,1]$ such that $(\tau_0u_0,\tau_0v_0)\in \mathcal{N}_{\infty}$. Thus, by calculation  we have
\begin{align*}
  c_{\infty}&\leq I_{\infty}(\tau_0u_0,\tau_0v_0)\\
   &=\frac{\tau_{0}^{2}}{4}\big(\int_{\R^N}|\nabla u_0|^{2}dx+ \int_{\R^N}|\nabla v_0|^{2}dx\big)\\
   &\leq \frac{1}{4}\big(\int_{\R^N}|\nabla u_0|^{2}dx+ \int_{\R^N}|\nabla v_0|^{2}dx\big)+\frac{1}{4}\int_{\R^N}(V_{1}(x)+\lambda_1) | u_0|^{2}dx+\frac{1}{4}\int_{\R^N}(V_{2}(x)+\lambda_2)|v_0|^{2}dx\\
   &=c=c_\infty,
\end{align*}
which implies that
\begin{equation}\label{e-2.9}
\tau_0=1 \ \ \text{and} \ \ \int_{\R^N}(V_{1}(x)+\lambda_1)|u_0|^{2}dx+\int_{\R^N}(V_{2}(x)+\lambda_2)|v_0|^{2}dx=0.
\end{equation}
Moreover, $c_\infty$ is attained by $(u_0,v_0)$.  Since $\beta > \max \{\mu_1,\mu_2\}$, then  by Lemma \ref{lm2.1}  we get
\begin{equation*}
  c_\infty=\frac{1}{4}(k_1+k_2)\mathcal{S}_{HL}^{2}<\min\{\frac{1}{4\mu_1}\mathcal{S}_{HL}^{2}, \frac{1}{4\mu_2}\mathcal{S}_{HL}^{2}  \}=\min\{m_{1}^{\infty}, m_{2}^{\infty} \},
\end{equation*}
which implies that $u_0 \neq 0$ and $v_0\neq 0$. By the strong maximum principle,  $u_0>0$ and $v_0>0$ for all $x\in \R^N$. Thus, there is a contradiction with  \eqref{e-2.9} because of $V_{1}(x), V_2(x)\geq 0$ and  $\lambda_1, \lambda_2\geq 0$ with $ \lambda:=\max\{\lambda_1,\lambda_2\}>0$.  Therefore, based on the arguments above, we prove that $c$ is not attained.

%Note that $\lambda_i\geq 0$, $V_{i}(x)\geq 0$ for $i=1,2$ and $\lambda:=\max\{\lambda_1,\lambda_2 \}>0$, we have
%\begin{equation*}
%  \int_{\R^N}(V_1(x)+\lambda_1)|u_0|^{2}dx+\int_{\R^N}(V_2(x)+\lambda_2)|v_0|^2dx >0,
%\end{equation*}
%which contradicts to \eqref{e-2.9}.
%If $\lambda_1,\lambda_2>0$, then
%\begin{align*}
%c_{\infty}&=c=I(u,v)\\
%&=\frac{1}{4}\big(\int_{\R^4}|\nabla u_0|^{2}dx+ \int_{\R^4}|\nabla v_0|^{2}dx\big)+\frac{1}{4}\int_{\R^4}(V_{1}(x)+\lambda_1)u^{2}dx+\frac{1}{4}\int_{\R^4}(V_{2}(x)+\lambda_2)v^{2}dx \\
%& > \frac{1}{4}\big(\int_{\R^4}|\nabla u_0|^{2}dx+ \int_{\R^4}|\nabla v_0|^{2}dx\big) \\
%& \geq \frac{t_{0}^2}{4}\big(\int_{\R^4}|\nabla u_0|^{2}dx+ \int_{\R^4}|\nabla v_0|^{2}dx\big)\\
%& =I_{\infty}(t_0u_0,t_0v_0)\geq c_{\infty},
%\end{align*}
%which is impossible.

\end{proof}

\section{Some compactness results}\label{s3}

\qquad In this paper, the problem we studied is affected by  lack of compactness due to the unboundedness of $\R^N$ and to the critical exponent. In this section, we introduce a global compactness result and investigate the behavior of Palais-Smale sequences of $I$, see Proposition \ref{th3.1} below. We would like to point out that, a version of that compactness result for single equations and coupled systems can be found in \cite{BC,LL,PPW,S} by different approach.

We recall
that a sequence $\{(u_n,v_n)\}\subset H$ is called a Palais-Smale sequence for $I$, if
$$
\sup_{n}|I(u_n,v_n)|<+\infty, \ \ I'(u_n,v_n)\to 0 \ \ \text{in} \ \  H^{-1},
$$
where $H^{-1}$ is the dual space of $H$.

Before stating the global compactness result,  we  first introduce  a Br\'{e}zis-Lieb type lemma for nonlocal terms, which is proved in \cite{GLMY}, so we omit it here.

\begin{lem}\label{lm3.2}(\cite[Lemma 4.1]{GLMY}) Let $N\geq 5$ and  $\{(u_n,v_n)\}$ to be a bounded sequence in $L^{\frac{2N}{N-2}}(\R^N)\times L^{\frac{2N}{N-2}}(\R^N)$ such that $(u_n,v_n) \to (u,v)$ a.e. in $\R^N$ as $n\to \infty$, then
\begin{align*}
  &\quad \lim_{n\to \infty} \Big(\int_{\R^N}(|x|^{-4}*|u_n^{+}|^2)|u_n^+|^{2}dx-\int_{\R^N}(|x|^{-4}*|(u_n-u)^+|^{2})|(u_n-u)^+|^2dx \Big)\\
  &=\int_{\R^N}(|x|^{-4}*|u^{+}|^2)|u^+|^{2}dx
  \end{align*}
and
\begin{align*}
  &\quad \lim_{n\to \infty} \Big(\int_{\R^N}(|x|^{-4}*|u_n^{+}|^2)|v_n^+|^{2}dx-\int_{\R^N}(|x|^{-4}*|(u_n-u)^+|^{2})|(v_n-v)^+|^2dx \Big)\\
  &=\int_{\R^N}(|x|^{-4}*|u^{+}|^2)|v^+|^{2}dx,
  \end{align*}
\end{lem}

\begin{lem}\label{lm2.5}
Suppose that $\lambda_1,\lambda_2\geq 0$ with $\lambda:=\max\{\lambda_1,\lambda_2\}>0$ and suppose that $(u,v)$ is the solution of
\begin{equation}\label{S3}
  \begin{cases}
    -\Delta u+ \lambda_1 u=\mu_1(|x|^{-4}*u^{2})u+\beta (|x|^{-4}*v^{2})u, \ \  x\in \R^N, \\
    -\Delta v+ \lambda_2 v=\mu_2(|x|^{-4}*v^{2})v+\beta (|x|^{-4}*u^{2})v, \ \  x\in \R^N.
  \end{cases}
\end{equation}
$(a)$ If $\lambda_1>0,\lambda_2>0$, then $(u,v)=(0,0)$.
\\
$(b)$ If $\lambda_1=0$, $\lambda_2> 0$, then $v=0$ and $u$ solves
\begin{equation}\label{eq2.19}
-\Delta u=\mu_1(|x|^{-4}*u^{2})u.
\end{equation}
$(c)$ If $\lambda_1> 0$, $\lambda_2=0$, then $u=0$  and $v$ solves
$$
-\Delta v=\mu_2(|x|^{-4}*v^{2})v.
$$
\end{lem}
\begin{proof}
 In what follows, we  prove the lemma by Pohozaev identity. More precisely, our proof follows by a classical strategy of testing the equation against $x\cdot \nabla u$, which is made rigorous by multiplying by cut-off functions, see  \cite[Appendix B]{WM} and  \cite[Proposition 3.5]{MS2} for details.

Since $\lambda_1,\lambda_2\geq 0$ with $\lambda:=\max\{\lambda_1,\lambda_2\}>0$, then the following cases can happen: $\lambda_1,\lambda_2>0$, or $\lambda_1=0,\lambda_2>0$, or $\lambda_1>0,\lambda_2=0$.

Let $(u,v)$ be a solution of system \eqref{S3}. Then by the standard elliptic regularity theory,  $(u,v)\in W_{loc}^{2,p}(\R^N)\times W_{loc}^{2,p}(\R^N)$, $p\geq 1$, see \cite[Theorem 2]{MS2}. Let $\phi(x)\in C_{0}^{1}(\R^N)$ be a cut-off function such that $\phi(x)=1$ on $B_{1}(0)$.  Multiplying the first equation by $u_{\rho}=\phi(\rho x)(x\cdot \nabla u)$, the second equation by $v_{\rho}=\phi(\rho x)(x\cdot\nabla v)$ and integrating by parts,  then the following properties hold:
   \begin{equation}\label{e-2.10}
     \int_{\R^N}\nabla u \cdot \nabla u_{\rho}dx+ \lambda_1 \int_{\R^N}uu_{\rho}dx=\mu_1\int_{\R^N}(|x|^{-4}*u^{2})uu_{\rho}dx+\beta\int_{\R^N}(|x|^{-4}*v^{2})uu_{\rho}dx
   \end{equation}
 and
  \begin{equation}\label{e-2.11}
     \int_{\R^N}\nabla v \cdot \nabla v_{\rho}dx+\lambda_2 \int_{\R^N}vv_{\rho}dx=\mu_2\int_{\R^N}(|x|^{-4}*v^{2})vv_{\rho}dx+\beta\int_{\R^N}(|x|^{-4}*u^{2})v v_{\rho}dx.
   \end{equation}
We compute for every $\rho>0$,
\begin{align*}
  &\quad\int_{\R^N}\nabla u \cdot \nabla u_{\rho}dx\\
  &=\int_{\R^N}\phi(\rho x)\Big(  |\nabla u|^{2}+ x\cdot \nabla\big(\frac{|\nabla u|^{2}}{2}\big)(x)\Big)dx+\int_{\R^N}\big(\nabla u\cdot\nabla \phi(\rho x) \big)  (\rho x\cdot \nabla u)dx \\
  &=\int_{\R^N}\Big( (2-N) \phi(\rho x)-\rho x \cdot \nabla \phi (\rho x) \Big)\frac{|\nabla u(x)|^{2}}{2}dx+\int_{\R^N}\big(\nabla u\cdot\nabla \phi(\rho x)\big)  (\rho x\cdot \nabla u)dx .
\end{align*}
By Lebesgue's dominated convergence theorem, we have
\begin{equation}\label{add-1}
  \lim_{\rho \to 0} \int_{\R^N}\nabla u \cdot \nabla u_{\rho} dx=\frac{2-N}{2}\int_{\R^N}|\nabla u|^{2}.
\end{equation}
Next, by calculation we have
\begin{align*}
  \int_{\R^N}uu_{\rho}dx&=  \int_{\R^N}u(x)\phi(\rho x)x\cdot \nabla u(x)dx\\
  &= \int_{\R^N}\phi(\rho x)x\cdot \nabla (\frac{u^{2}}{2})(x)dx  \\
  &=  - \int_{\R^N}\big(N \phi(\rho x)+\rho x \cdot \nabla \phi(\rho x)\big)\frac{|u(x)|^{2}}{2}dx.
\end{align*}
By  Lebesgue's dominated convergence theorem again, it holds that
\begin{equation}\label{e-2.12}
\lim_{\rho\to0}  \int_{\R^N}u u_{\rho}dx=-\frac{N}{2} \int_{\R^N}|u|^{2}dx.
\end{equation}
By a similar argument as above, we can also use the Lebesgue's dominated convergence theorem to conclude that
\begin{align}\label{e-2.13}
\lim_{\rho \to 0}\int_{\R^N}(|x|^{-4}*u^{2})uu_{\rho}dx &=\lim_{\rho \to 0}\int_{\R^N}(|x|^{-4}*u^{2})\phi(\rho x)x\cdot \nabla (\frac{u^{2}}{2})(x)dx\nonumber\\
&=-\lim_{\rho \to 0}\int_{\R^N}(|x|^{-4}*u^{2})\frac{u^{2}}{2}[N\phi(\rho x)+\rho x\cdot \nabla \phi(\rho x)]dx\nonumber\\
&\qquad +\lim_{\rho \to 0}\int_{\R^N}\int_{\R^N}\frac{2x(x-y)|u(x)|^2|u(y)|^2\phi(\rho x)}{|x-y|^6}dxdy\nonumber\\
& =-\frac{N}{2}\int_{\R^N}(|x|^{-4}*u^{2})u^2dx+\int_{\R^N}\int_{\R^N}\frac{(2x^2-2xy)|u(x)|^2|u(y)|^2}{|x-y|^6}dxdy\nonumber\\
&=-\frac{N}{2}\int_{\R^N}(|x|^{-4}*u^{2})u^2dx+\int_{\R^N}\int_{\R^N}\frac{(x^2-2xy+y^2)|u(x)|^2|u(y)|^2}{|x-y|^6}dxdy\nonumber\\
&=-\frac{N-2}{2}\int_{\R^N}(|x|^{-4}*u^{2})u^2dx.
\end{align}
Arguing as above,  we can also prove that
%\begin{equation}
%  \lim_{\rho \to 0} \int_{\R^N}(|x|^{-4}*v^{2})vv_{\rho}=-\frac{N-2}{2}\int_{\R^N}(|x|^{-4}*u^{2})u^2dx.
%\end{equation}
%and
\begin{align}\label{e-2.14}
 &\quad \lim_{\rho \to 0} \int_{\R^N}(|x|^{-4}*v^{2})uu_{\rho} dx\nonumber\\
  &=  -\frac{N}{2}\int_{\R^N}(|x|^{-4}*v^{2})u^2dx+\int_{\R^N}\int_{\R^N}\frac{2x(x-y)|u(x)|^2|v(y)|^2}{|x-y|^6}dxdy.
\end{align}
 Therefore, by \eqref{add-1}-\eqref{e-2.14},  we  prove  that \eqref{e-2.10} equals to
 \begin{align}\label{e-2.15}
   &\quad -\frac{N-2}{2}\int_{\R^N}|\nabla u|^{2}dx-\frac{ N\lambda_1}{2} \int_{\R^N}|u|^{2}dx   \nonumber\\
   &= -\frac{N-2}{2}\mu_1\int_{\R^N}(|x|^{-4}*u^{2})u^2dx-\frac{N}{2}\beta\int_{\R^N}(|x|^{-4}*v^{2})u^2dx  \\
   & \quad +\beta\int_{\R^N}\int_{\R^N}\frac{(2x^2-2xy)|u(x)|^2|v(y)|^2}{|x-y|^6}dxdy.     \nonumber
 \end{align}
Repeating the argument as above, we  can  also prove that  \eqref{e-2.11} equals to
\begin{align}\label{e-2.16}
  & \quad -\frac{N-2}{2}\int_{\R^N}|\nabla v|^{2}dx-\frac{N\lambda_2}{2} \int_{\R^N}|v|^{2}dx  \nonumber \\
& =-\frac{N-2}{2}\mu_2\int_{\R^N}(|x|^{-4}*v^{2})v^2dx    -\frac{N}{2}\beta\int_{\R^N}(|x|^{-4}*u^{2})v^2dx \\
  & \quad +\beta\int_{\R^N}\int_{\R^N}\frac{(2x^2-2xy)|v(x)|^2|u(y)|^2}{|x-y|^6}dxdy.  \nonumber
 \end{align}
 Note that
 \begin{align*}
&\quad \int_{\R^N}\int_{\R^N}\frac{\beta(2x^2-2xy)|u(x)|^2|v(y)|^2}{|x-y|^6}dxdy+\int_{\R^N}\int_{\R^N}\frac{\beta(2x^2-2xy)|v(x)|^2|u(y)|^2}{|x-y|^6}dxdy\\
 &=\int_{\R^N}\int_{\R^N}\frac{\beta(2x^2-2xy)|u(x)|^2|v(y)|^2}{|x-y|^6}dxdy+\int_{\R^N}\int_{\R^N}\frac{\beta(2y^2-2xy)|u(x)|^2|v(y)|^2}{|x-y|^6}dxdy\\
 &=\int_{\R^N}\int_{\R^N}\frac{\beta(2x^2-4xy+2y^2)|u(x)|^2|v(y)|^2}{|x-y|^6}dxdy\\
 &=2\beta\int_{\R^N}\int_{\R^N}\frac{|u(x)|^2|v(y)|^2}{|x-y|^4}dxdy,
 \end{align*}
by \eqref{e-2.15}-\eqref{e-2.16} we get
  \begin{align}\label{e-2.17}
 & \ \  -\frac{N-2}{2}\int_{\R^N}(|\nabla u|^{2}+|\nabla v|^{2})dx-\frac{N}{2}\int_{\R^N}(\lambda_1u^{2}+ \lambda_2 v^2)dx \\
 &= -\frac{N-2}{2}\Big[\mu_1 \int_{\R^N}(|x|^{-4}*u^{2})u^2dx+ \mu_2\int_{\R^N}(|x|^{-4}*v^{2})v^2dx \Big]-(N-2)\beta\int_{\R^N}(|x|^{-4}*v^{2})u^2dx.      \nonumber
 \end{align}
On the  other hand, since $(u,v)$ is a pair solution of  system \eqref{S3}, then
\begin{align}\label{e-2.18}
 & \quad \int_{\R^N}(|\nabla u|^{2}+|\nabla v|^{2})dx+\int_{\R^N}(\lambda_1u^{2}+\lambda_2 v^2)dx
 \\
  & =\mu_1\int_{\R^N}(|x|^{-4}*u^{2})u^2dx+\mu_2\int_{\R^N}(|x|^{-4}*v^{2})v^2dx+2\beta\int_{\R^N}(|x|^{-4}*v^{2})u^2dx. \nonumber
\end{align}
Thus, it follows by \eqref{e-2.17} and \eqref{e-2.18} that
\begin{equation}\label{Guo-Non}
  \int_{\R^N} (\lambda_1u^{2}+\lambda_2v^{2})dx=0.
\end{equation}
If $\lambda_1>0 $ and $\lambda_2>0$,  then by \eqref{Guo-Non}, we have $(u,v)=(0,0)$, thus $(a)$ occurs. If $\lambda_1=0$ and $\lambda_2>0$, then by \eqref{Guo-Non} again, we get $v$=0, moreover  $u$  solves
\begin{equation*}
-\Delta u=\mu_1(|x|^{-4}*u^{2})u.
\end{equation*}
Thus, $(b)$ is right. Finally, repeating the arguments above, we can also prove that $(c)$ is true.
\end{proof}

\begin{thm}\label{th3.1}Suppose that $\lambda_1,\lambda_2\geq 0$ with $\lambda:=\max\{\lambda_1,\lambda_2\}>0$, let  $(A_1)$-$(A_2)$ hold and $\{(u_n,v_n)\}$ $\subset H$ be a Palais-Smale sequence  of functional $I$ at level $d$. \\
 $(a)$\,If $\lambda_1>0,\lambda_2>0$, then  there exist a solution $(u^0,v^0)$ of system \eqref{S-4}, $\ell$ sequences of positive numbers $\{\sigma_{n}^{k}\}$ $(1\leq k\leq \ell)$ and $\ell$  sequences of points $\{y_{n}^{k}\}$ $(1\leq k\leq \ell)$ in $\R^N$, such that  \\
\begin{align*} \label{e-3.1}
  \|(u_n,v_n)\|_{H}^{2}=\|(u^0,v^0)\|_{H}^{2}
  +\sum_{k=1}^{\ell}\|\Big((\sigma_{n}^{k})^{-\frac{N-2}{2}}u^k(\frac{x-y_{n}^{k}}{\sigma_{n}^{k}}),\ (\sigma_{n}^{k})^{-\frac{N-2}{2}}v^k(\frac{x-y_{n}^{k}}{\sigma_{n}^{k}})\Big)\|_{H_{0}}^{2}+o_{n}(1)
\end{align*}
and
\begin{equation*} \label{e-3.2}
I(u_n,v_n)=I(u^0,v^0)+\sum_{k=1}^{\ell}I_{\infty}(u^k,v^k)+o_{n}(1),
\end{equation*}
where $\sigma_{n}^{k}\to 0$ as $n\to\infty$ and  $(u^k,v^k)\neq (0,0)$ solves  system \eqref{S2}. \\
$(b)$\,If $\lambda_1=0,\lambda_2>0$, then there exist a solution $(u^0,v^0)$ of system \eqref{S-4}, $\ell_{1}$ sequence of points $\{z_n^{k}\}$ $(1\leq k\leq \ell_1)$, $\ell_2$ sequences of positive numbers $\{\sigma_{n}^{k}\}$ $(1\leq k\leq \ell_2)$ and $\ell_2$  sequences of points $\{y_{n}^{k}\}$ $(1\leq k\leq \ell_2)$ in $\R^N$ such that
\begin{align*}
  \|(u_n,v_n)\|_{H}^{2}&=\|(u^0,v^0)\|_{H}^{2}+\sum_{k=1}^{\ell_1}\|(w^{k}(x-z_{n}^k),0)\|^2_{H_0}\\
  &\quad+\sum_{k=1}^{\ell_2}\|\Big((\sigma_{n}^{k})^{-\frac{N-2}{2}}u^k(\frac{x-y_{n}^{k}}{\sigma_{n}^{k}}),\ (\sigma_{n}^{k})^{-\frac{N-2}{2}}v^k(\frac{x-y_{n}^{k}}{\sigma_{n}^{k}})\Big)\|_{H_0}^{2}+o_{n}(1)
\end{align*}
and
\begin{equation*}
I(u_n,v_n)=I(u^0,v^0)+\sum_{k=1}^{\ell_1}I_{\infty}(w^k,0)+\sum_{k=1}^{\ell_2}I_{\infty}(u^k,v^k)+o_{n}(1),
\end{equation*}
where $|z_{n}^k|\to \infty$, $\sigma_{n}^{k}\to 0$ as $n\to\infty$, $(u^k,v^k)\not\equiv(0,0)$ solves system \eqref{S2} and $(w^{k},0)$ is a semi-trivial solution  of system \eqref{S2}. \\
$(c)$\,If $\lambda_1>0,\lambda_2= 0$, then there exist a solution $(u^0,v^0)$ of system \eqref{S-4}, $\ell_{1}$ sequence of points $\{z_n^{k}\}$ $(1\leq k\leq \ell_1)$, $\ell_2$ sequences of positive numbers $\{\sigma_{n}^{k}\}$ $(1\leq k\leq \ell_2)$ and $\ell_2$  sequences of points $\{y_{n}^{k}\}$ $(1\leq k\leq \ell_2)$ in $\R^N$ such that
\begin{align*}
  \|(u_n,v_n)\|_{H}^{2}&=\|(u^0,v^0)\|_{H}^{2}+\sum_{k=1}^{\ell_1}\|(0, w^{k}(x-z_{n}^k))\|^2_{H_0}\\
  &\quad+\sum_{k=1}^{\ell_2}\|\Big((\sigma_{n}^{k})^{-\frac{N-2}{2}}u^k(\frac{x-y_{n}^{k}}{\sigma_{n}^{k}}),\ (\sigma_{n}^{k})^{-\frac{N-2}{2}}v^k(\frac{x-y_{n}^{k}}{\sigma_{n}^{k}})\Big)\|_{H_0}^{2}+o_{n}(1)
\end{align*}
and
\begin{equation*}
I(u_n,v_n)=I(u^0,v^0)+\sum_{k=1}^{\ell_1}I_{\infty}(0,w^k)+\sum_{k=1}^{\ell_2}I_{\infty}(u^k,v^k)+o_{n}(1),
\end{equation*}
where $|z_{n}^k|\to \infty$, $\sigma_{n}^{k}\to 0$ as $n\to\infty$, $(u^k,v^k)\not\equiv(0,0)$ solves system \eqref{S2} and $(0,w^{k})$ is a semi-trivial solution  of system \eqref{S2}.
\end{thm}

%\begin{proposition}\label{th3.2}
%Assume that $\lambda_1=0$ and $\lambda_2>0$. Let  $\{(u_n,v_n)\}\subset H$ be a $(P.S.)_{d}$ sequence  for  functional $I$. Then   there exist a solution $(u^0,v^0)$ of system \eqref{S-4}, $\ell_{0}$ sequence of points $\{y_n^{k}\}$ $(1\leq k\leq \ell_0)$, $\ell$ sequences of positive numbers $\{\sigma_{n}^{k}\}$ $(1\leq k\leq \ell)$ and $\ell$  sequences of points $\{x_{n}^{k}\}$ $(1\leq k\leq \ell)$ in $\R^N$ such that
%\begin{align*}
%  \|(u_n,v_n)\|_{H}^{2}&=\|(u^0,v^0)\|_{H}^{2}+\sum_{k=1}^{\ell_0}\|(w^{k}(x-y_{n}^k),0)\|^2_{H_0}\\
%  &\quad+\sum_{k=1}^{\ell}\|\Big((\sigma_{n}^{k})^{-\frac{N-2}{2}}u^k(\frac{x-x_{n}^{k}}{\sigma_{n}^{k}}),\ (\sigma_{n}^{k})^{-\frac{N-2}{2}}v^k(\frac{x-x_{n}^{k}}{\sigma_{n}^{k}})\Big)\|_{H_0}^{2}+o_{n}(1)
%\end{align*}
%and
%\begin{equation}
%I(u_n,v_n)=I(u^0,v^0)+\sum_{k=1}^{\ell_0}I_{\infty}(w^k,0)+\sum_{k=1}^{\ell}I_{\infty}(u^k,v^k)+o_{n}(1),
%\end{equation}
%where $|y_{n}^k|\to \infty$, $\sigma_{n}^{k}\to 0$ as $n\to\infty$, $(u^k,v^k)\not\equiv(0,0)$ solves system \eqref{S2} and $(w^{k},0)$ is a semitrivial solution  of system \eqref{S2}.
%\end{proposition}

\begin{proof}
$(a)$\,Let us first consider the case $\lambda_1, \lambda_2>0$,  and now $H=H^{1}(\R^N)\times H^{1}(\R^N)$.  Since $\{(u_n,v_n)\}\subset H$ is a sequence of Palais-Smale sequence  for the functional $I$, then by computation we prove that $\{(u_n,v_n)\}$ is bounded in $H$. Without loss of generality, we may suppose that $(u_n,v_n) \rightharpoonup (u^0,v^0)$ in $H$,  $(u_n,v_n)\to (u^0,v^0)$ a.e. in $\R^N\times\R^N$ and $(u_n,v_n)\to (u^0,v^0)$ in $L_{loc}^{2}(\R^N) \times L_{loc}^{2}(\R^N) $, where $(u^0,v^0)$ is a pair of solution to system \eqref{S-4}.

Let $(u_n^{1},v_{n}^{1})=(u_n,v_n)-(u^0,v^0)$, then
\begin{equation}\label{cgl-eq3.13}
(u_n^1,v_n^1) \rightharpoonup (0,0)\ \text{in} \ H,  \quad  (u_{n}^{1},v_{n}^{1})\to (0,0) \ \ a.e. \ \ \text{in} \  \R^N\times\R^N \ \ \text{and in} \ \ L_{loc}^{2}(\R^N) \times L_{loc}^{2}(\R^N).
\end{equation}
It  follows by Lemma \ref{lm3.2} and the classical Br\'{e}zis-Lieb lemma \cite{WM} that
\begin{equation*}
I_{\lambda,\infty}(u_{n}^{1},v_{n}^{1})=I(u_n,v_n)-I(u^0,v^0)+o_{n}(1)
\end{equation*}
and
\begin{equation*}
 I'_{\lambda,\infty}(u_{n}^{1},v_{n}^{1})=I'(u_n,v_n)-I'(u^0,v^0)+o_{n}(1)=o_{n}(1),
\end{equation*}
where
\begin{equation*}
  I_{\lambda,\infty}(u,v):=I_{\infty}(u,v)+\frac{\lambda_1}{2}\int_{\R^N}u^{2}dx+\frac{\lambda_2}{2}\int_{\R^N}v^2dx.
\end{equation*}
Therefore, $\{(u_n^{1},v_{n}^{1})\}$ is a Palais-Smale sequence for $I_{\lambda,\infty}$.

If $(u_n^{1},v_{n}^{1})\to (0,0)$ in $H$, then we have done. If $(u_n^{1},v_{n}^{1})\rightharpoonup (0,0)$, but $(u_n^{1},v_{n}^{1})\nrightarrow (0,0)$ in $H$.
Then there exists a positive constant $\tilde{c}>0$ such that
\begin{equation*}
  \|(u_{n}^{1},v_{n}^{1})\|_{H}^2\geq \tilde{c}>0,
\end{equation*}
that is
\begin{equation*}
  \int_{\R^N}\big(|\nabla u_n^1|^{2}+\lambda_1  |u_n^1|^2\big)dx+\int_{\R^N}\big(|\nabla v_n^1|^{2}+\lambda_2  |v_{n}^{1}|^{2}\big)dx\geq \tilde{c}.
\end{equation*}
As we known that $\{(u_n^{1},v_{n}^{1})\}$ is a Palais-Smale sequence for $I_{\lambda,\infty}$,  then by Hardy-Littlewood-Sobolev inequality and Young inequality,
\begin{align*}
\tilde{c}&\leq \int_{\R^N}\big(|\nabla u_n^1|^{2}+ \lambda_1  |u_n^1|^2\big)dx+\int_{\R^N}\big(|\nabla v_n^1|^{2}+{\lambda_2}  |v_{n}^{1}|^{2}\big)dx \\
&=\mu_1\int_{\R^N}(|x|^{-4}*|(u_n^1)^+|^{2})|(u_n^1)^+|^2dx+\mu_2\int_{\R^N}(|x|^{-4}*|(v_n^1)^+|^{2})|(v_n^1)^+|^2dx\\
&\quad+2\beta\int_{\R^N}(|x|^{-4}*|(v_n^1)^+|^{2})|(u_n^1)^+|^2dx+o_{n}(1) \\
& \leq C \Big[\Big(\int_{\R^N}|(u_{n}^{1})^+|^{2^{*}}dx\Big)^{\frac{4}{2^{*}}}+\Big(\int_{\R^N}|(v_n^1)^+|^{2^*}dx\Big)^{\frac{4}{2^*}}\Big]+o_n(1),
\end{align*}
which implies that there exists  a positive constant $\hat{c}>0$ such that
\begin{equation}
  \int_{\R^N}|(u_{n}^{1})^+|^{2^*}dx+ \int_{\R^N}|(v_n^1)^+|^{2^*} dx\geq \hat{c}.
\end{equation}
We  assume without loss of generality that $\|(v_n^1)^+\|_{L^{2^*}}\geq \frac{1}{2}\hat{c}$ and set
$$
t_{n}=\frac{\|(u_n^1)^+\|_{L^{2^*}}^2}{\|(v_n^1)^+\|_{L^{2^*}}^2}\geq 0,
$$
then by the Sobolev inequality and Hardy-Littlewood-Sobolev inequality, we get
\begin{align*}
&\quad\mathcal{S}(t_n+1)\|(v_n^1)^+\|_{L^{2^*}}^2\\
&=\mathcal{S}(\|(u_n^1)^+\|_{L^{2^*}}^2+\|(v_n^1)^+\|_{L^{2^*}}^2)\\
&\leq \|(u_{n}^1, v_{n}^1)\|_{H}^2\\
&=\mu_1\int_{\R^N}(|x|^{-4}*|(u_n^1)^+|^{2})|(u_n^1)^+|^2dx+\mu_2\int_{\R^N}(|x|^{-4}*|(v_n^1)^+|^{2})|(v_n^1)^+|^2dx\\
&\quad+2\beta\int_{\R^N}(|x|^{-4}*|(v_n^1)^+|^{2})|(u_n^1)^+|^2dx+o_{n}(1)\\
&\leq C(N,4)(\mu_1 \|(u_n^1)^+\|_{L^{2^*}}^4+\mu_2 \|(v_n^1)^+\|_{L^{2^*}}^4+2\beta\|(u_n^1)^+\|_{L^{2^*}}^2\|(v_n^1)^+\|_{L^{2^*}}^2)+o_n(1)\\
&=C(N,4)(\mu_1 t_n^2+2\beta t_n+\mu_2)\|v_n^+\|_{L^{2^*}}^4+o_n(1),
\end{align*}
which leads to
$$
\|v_n^+\|_{L^{2^*}}^2\geq \frac{\mathcal{S}(t_n+1)-o_n(1)}{C(N,4)(\mu_1 t_n^2+2\beta t_n+\mu_2)}.
$$
A direct computation shows that
$$
\inf_{t\ge 0}\frac{(1+t)^2}{\mu_1 t^2+2\beta t+\mu_2}=k_1+k_2
$$
(see  \cite[Lemma 2.3]{LL}), then
\begin{align}\label{li-a-eq3.8}
\int_{\R^N}\big(|\nabla{u}_n^1|^{2}+|\nabla{v}_n^1|^{2}\big)dx&\geq \mathcal{S}(\|(u_n^1)^+\|_{L^{2^*}}^2+\|(v_n^1)^+\|_{L^{2^*}}^2)\nonumber\\
&=\mathcal{S}(t_n+1)\|(v_n^1)^+\|_{L^{2^*}}^2\nonumber\\
&\geq \frac{\mathcal{S}^2(t_n+1)^2}{C(N,4)(\mu_1 t_n^2+2\beta t_n+\mu_2)}-o_{n}(1)\nonumber\\
&\geq (k_1+k_2)\mathcal{S}_{HL}^2-o_{n}(1)=4c_\infty-o_{n}(1).
\end{align}
We claim that
\begin{equation}\label{li-3.8}
  d_{n}^{1}:=\max_{i\in \N}\Big(\int_{P_i} \Big( |(u_n^1)^+|^{2^*}+|(v_n^1)^+|^{2^*}\Big)dx  \Big)^{\frac{1}{2^*}}\geq C_0>0,
\end{equation}
where $P_i$ are hypercubes with disjoint interior and unitary sides with $\R^N=\sum_{i\in \N}P_i$ , $i\in \N$. Indeed, by calculation we have
\begin{align*}
0<\hat{c} &\leq  \int_{\R^N}\Big( |(u_{n}^{1})^+|^{2^*}+|(v_n^1)^+|^{2^*}\Big)dx=
\sum_{i=1}^{\infty} \int_{Q_i} \Big( |(u_{n}^{1})^+|^{2^*}+|(v_n^1)^+|^{2^*}\Big)dx \\
&\leq (d_{n}^{1})^{2^{*}-2}  \sum_{i=1}^{\infty}  \Big[\int_{P_i} \Big( |(u_{n}^{1})^+|^{2^*}+|(v_n^1)^+|^{2^*}\Big)dx \Big]^{\frac{2}{2^{*}}} \\
& \leq (d_{n}^{1})^{2^{*}-2}  \sum_{i=1}^{\infty}  \Big[\Big(  \int_{P_{i}}|(u_n^1)^+|^{2^*}dx  \Big)^{\frac{2}{2^*}}+  \Big(  \int_{P_{i}}|(v_n^1)^+|^{2^*}dx  \Big)^{\frac{2}{2^*}}\Big]               \\
& \leq C_1 (d_{n}^{1})^{2^*-2} \sum_{i=1}^{\infty}  \Big[ \int_{P_i} \big(|\nabla u_n^1|^2+\lambda_1  |u_{n}^{1}|^{2}\big)dx+\int_{P_i}\big(|\nabla v_n^1|^2+\lambda_2 |v_{n}^{1}|^{2}\big)dx \Big]\\
& = C_1 (d_{n}^{1})^{2^*-2} \Big[ \int_{\R^N} \big(|\nabla u_n^1|^2+\lambda_1  |u_{n}^{1}|^{2}\big)dx+\int_{\R^N}\big(|\nabla v_n^1|^2+\lambda_2 |v_{n}^{1}|^{2}\big)dx \Big]\\
& \leq C_{2}(d_{n}^{1})^{2^{*}-2},
\end{align*}
where $C_2$  is a positive constant independent of $i$ and the last inequality is due to the boundedness of $(u_n^1,v_{n}^{1})$ in $H$.
Thus, the claim holds true. We can also observe that
$$
\max_{i\in \N} \Big(\int_{P_i} \big(|\nabla u_n^1|^{2}+|\nabla v_n^1|^{2}\big)dx\Big)\geq \max_{i\in \N}\Big(\int_{P_i} \Big( |(u_n^1)^+|^{2^*}+|(v_n^1)^+|^{2^*}\Big)dx  \Big)^{\frac{2}{2^*}}\mathcal{S}\geq \mathcal{S}C_0^2.
$$

Define a concentration function of $(u_{n}^{1},v_{n}^{1})$ by
$$
Q_{n}(r)=\underset{y\in \R^N}{\sup} \int_{B_{r}(y)}\big(|\nabla u_n^1|^{2}+|\nabla v_n^1|^{2}\big)dx
$$
Since $Q_{n}(0)=0$, $Q_n(\infty)\geq 4c_\infty-o_{n}(1)$ and $Q_n(r)$ is continuous in $r$, then there exists $y_{n}^1\in\R^N$, $\sigma_n>0$ such that
\begin{equation}\label{Guo-y}
Q_{n}(\sigma_n)=\int_{B_{\sigma_n}(y_n^1)}\big(|\nabla u_n^1|^{2}+|\nabla v_n^1|^{2}\big)dx=\delta<\min\Big\{\frac{2c_\infty}{L},\frac{\mathcal{S}C_0^2}{2}\Big\}
\end{equation}
where $L$ is the least number of balls with radius $1$ covering a ball of radius $2$, and $\delta>0$ is independent of $n$.   We see that  $\sigma_{n}$ is bounded, otherwise,  for large $n$,
$$
Q_n(\sigma_n)\geq \max_{i\in \N} \Big(\int_{P_i} \big(|\nabla u_n^1|^{2}+|\nabla v_n^1|^{2}\big)dx\Big)\geq \mathcal{S}C_0^2>\delta,
$$
which leads to a contradiction.

We set
\begin{equation}\label{cgl-eq3.18}
(\hat{u}_{n}^1,\hat{v}_{n}^1):=\sigma_{n}^{\frac{N-2}{2}}(u_n^1(\sigma_n x+y_n^1), v_n^1(\sigma_n x+y_n^1)),
\end{equation}
then
$$
\int_{\R^N}\big(|\nabla\hat{u}_n^1|^{2}+|\nabla\hat{v}_n^1|^{2}\big)dx=\int_{\R^N}\big(|\nabla u_n^1|^{2}+|\nabla v_n^1|^{2}\big)dx<\infty
$$
and
\begin{equation}\label{114-3.10}
\int_{B_{1}(0)}\big(|\nabla\hat{u}_n^1|^{2}+|\nabla\hat{v}_n^1|^{2}\big)dx=\delta.
\end{equation}
Hence, there exists $(u^1, v^1)\in H_0$ such that $(\hat{u}_{n}^1,\hat{v}_{n}^1)\rightharpoonup (u^1, v^1)$ in $H_0$ and $(\hat{u}_{n}^1,\hat{v}_{n}^1)\to (u^1, v^1)$ a.e. on $\R^N\times\R^N$.

We claim that $(u^1, v^1)\neq (0,0)$ is a  solution of system \eqref{S2}. Indeed,  arguing as in \cite{S} (see also  \cite[Lemma 3.6]{PPW}, \cite[Theorem 3.1]{GLLM}), we can find $\rho\in [1,2]$ such that
$\hat{u}_{n}^{1}-u^1\to 0$ and $\hat{v}_{n}^{1}-v^1\to 0$ in $H^{1/2,2}(\partial B_\rho(0))$.
Then, the solutions $\phi_{1,n}$, $\phi_{2,n}$ of
\begin{equation*}
\begin{cases}
-\Delta \phi=0,\quad x\in B_{3}(0)\backslash B_{\rho}(0),\\
\phi|_{\partial B_{\rho}(0)}=\hat{u}_n^1-u^1, \phi|_{\partial B_{3}(0)}=0,
\end{cases}
\end{equation*}
and
\begin{equation*}
\begin{cases}
-\Delta \phi=0,\quad x\in B_{3}(0)\backslash B_{\rho}(0),\\
\phi|_{\partial B_{\rho}(0)}=\hat{v}_n^1-v^1, \phi|_{\partial B_{3}(0)}=0,
\end{cases}
\end{equation*}
respectively, satisfying
\begin{equation}\label{li-3.10}
\phi_{1,n}\to 0, \quad \phi_{2,n}\to 0\quad \text{in}\ H^{1}( B_{3}(0)\backslash B_{\rho}(0)).
\end{equation}
Let
\begin{equation}\label{li-3.11}
\varphi_{1,n}=
\begin{cases}
\hat{u}_n^1-u^1,\quad x\in B_{\rho}(0),\\
\phi_{1,n}, \quad x\in B_{3}(0)\backslash B_{\rho}(0), \\
0, \quad x\in\R^N \backslash B_{3}(0)
\end{cases}
\end{equation}
and
\begin{equation}\label{li-3.12}
\varphi_{2,n}=
\begin{cases}
\hat{v}_n^1-v^1,\quad x\in B_{\rho}(0),\\
\phi_{2,n},\quad x\in B_{3}(0)\backslash B_{\rho}(0),\\
0, \quad x\in\R^N \backslash B_{3}(0),
\end{cases}
\end{equation}
then
$$
\|\varphi_{j,n}\|_{L^2(\R^N)}\to 0, \quad {as}\ n\to \infty,\ j=1,2.
$$
Let
$$
\hat{\varphi}_{j,n}=\sigma_n^{-\frac{N-2}{2}}\varphi_{j,n}(\frac{x}{\sigma_n}),\quad j=1,2.
$$
By \eqref{li-3.10}-\eqref{li-3.12}, we get
\begin{equation}\label{eq-114-3.14}
\|\hat{\varphi}_{1,n}\|_{\lambda_1}^2=\|\varphi_{1,n}\|_{D^{1,2}}^2+\lambda_1\sigma_n^2\|\varphi_{1,n}\|_{L^2(\R^N)}^2=\|\hat{u}_n^1-u^1\|_{D^{1,2}(B_{\rho}(0))}^2+o_n(1).
\end{equation}
Similarly,
\begin{equation}\label{eq-114-3.15}
\|\hat{\varphi}_{2,n}\|_{\lambda_2}^2=\|\varphi_{2,n}\|_{D^{1,2}}^2+\lambda_2\sigma_n^2\|\varphi_{2,n}\|_{L^2(\R^N)}^2=\|\hat{v}_n^1-v^1\|_{D^{1,2}(B_{\rho}(0))}^2+o_n(1).
\end{equation}
Note that  $\{(u_{n}^{1}, v_{n}^{1})\}$  is a sequence of Palais-Smale  sequence  for $I_{\lambda,\infty}$, then
$$
\langle I^{\prime}_{\infty}(\hat{u}_{n}^{1}, \hat{v}_{n}^{1}),(\varphi_{1,n},\varphi_{2,n})\rangle=\langle I^{\prime}_{\lambda,\infty}(u_{n}^{1}, v_{n}^{1}),(\hat{\varphi}_{1,n},\hat{\varphi}_{2,n})\rangle+o_{n}(1)=o_{n}(1).
$$
Furthermore,  we get
\begin{align}\label{Guo-1}
 o_{n}(1)&= \langle I'_{\infty}(\hat{u}_{n}^{1},\hat{v}_{n}^{1} ), (\varphi_{1,n}, \varphi_{2,n})\rangle  \nonumber \\
 &= \int_{B_{\rho}(0)} \nabla  \hat{u}_{n}^{1} \nabla(\hat{u}_{n}^{1}-u^1)dx+ \int_{B_{\rho}(0)} \nabla  \hat{v}_{n}^{1} \nabla(\hat{v}_{n}^{1}-v^1)dx  \nonumber  \\
 & \ \ - \mu_{1} \int_{B_{\rho}(0)} \int_{\R^N} \frac{|(\hat{u}_{n}^{1})^+(x)|^2(\hat{u}_{n}^{1})^{+}(y)(\hat{u}_{n}^{1}-u^1)(y)}{|x-y|^{4}}dxdy \\
 & \ \ - \mu_{2} \int_{B_{\rho}(0)} \int_{\R^N} \frac{ |(\hat{v}_{n}^{1})^+(x)|^2 (\hat{v}_{n}^{1})^{+}(y)(\hat{v}_{n}^{1}-v^1)(y)}{|x-y|^{4}}dxdy \nonumber \\
 & \ \ - \beta \int_{B_{\rho}(0)} \int_{\R^N} \frac{(\hat{u}_{n}^{1})^{+}(x)(\hat{u}_{n}^{1}-u^1)(x) |(\hat{v}_{n}^{1})^+(y)|^2}{|x-y|^{4}}dx dy \nonumber \\
 & \ \ - \beta \int_{B_{\rho}(0)}\int_{\R^N} \frac{ |(\hat{u}_{n}^{1})^+(x)|^2 (\hat{v}_{n}^{1})^{+}(y)(\hat{v}_{n}^{1}-v^1)(y)}{|x-y|^{4}}dx dy +o_{n}(1).  \nonumber
\end{align}
Since $(\hat{u}_{n}^1,\hat{v}_{n}^1)$ is bounded in $H_0$, then $\{\hat{u}_{n}^{1}\} $ is bounded in $L^{\frac{2N}{N-2}}$ and $\hat{u}_{n}^{1} \to u^{1}$ a.e. in $\R^N$, we have $|(\hat{u}_{n}^{1})^+|^2 \rightharpoonup |(u^1)^+|^2$ in $L^{\frac{N}{N-2}}(\R^N)$. By the Hardy-Littlewood-Sobolev inequality, the Riesz potential defines a linear continuous map from $L^{\frac{N}{N-2}}$ to $L^{\frac{N}{2}}(\R^N)$ and then
\begin{equation*}
  \int_{\R^N}\frac{|(\hat{u}_{n}^{1})^{+}(x)|^{2}}{|x-y|^{4}}dx \rightharpoonup \int_{\R^N}\frac{|(u^{1})^+|^2}{|x-y|^{4}}dx \ \ \text{in} \ \ L^{\frac{N}{2}}(\R^N).
\end{equation*}
Combining this with  $(\hat{u}_{n}^{1})^+ \rightharpoonup (u^1)^{+}$ in $L^{\frac{2N}{N-2}}$, we then get
$$
(\hat{u}_{n}^{1})^+ (y)\int_{\R^N} \frac{|(\hat{u}_{n}^{1})^+(x)|^{2}}{|x-y|^{4}}dx\rightharpoonup
(u^1)^{+}(y)\int_{\R^N}\frac{|(u^1)^{+}(x)|^2}{|x-y|^{4}}dx \ \ \text{in} \ \ L^{\frac{2N}{N+2}}(\R^N),
$$
which implies that
\begin{align*}
  \lim_{n \to \infty}\int_{B_{\rho}(0)} \int_{\R^N} \frac{|(\hat{u}_{n}^{1})^+(x)|^2(\hat{u}_{n}^{1})^{+}(y)u^1(y)}{|x-y|^{4}}dxdy
  &=\int_{B_{\rho}(0)} \int_{\R^N} \frac{|(u^{1})^+(x)|^2(u^{1})^{+}(y)u^1(y)}{|x-y|^{4}}dxdy \\
  &=\int_{B_{\rho}(0)} \int_{\R^N} \frac{|(u^{1})^+(x)|^2|(u^{1})^{+}(y)|^2}{|x-y|^{4}}dxdy.
\end{align*}
Therefore,
\begin{align*}
  & \quad \int_{B_{\rho}(0)} \int_{\R^N} \frac{|(\hat{u}_{n}^{1})^+(x)|^2(\hat{u}_{n}^{1})^{+}(y)(\hat{u}_{n}^{1}-u^1)(y)}{|x-y|^{4}}dxdy \\
  &= \int_{B_{\rho}(0)} \int_{\R^N} \frac{|(\hat{u}_{n}^{1})^+(x)|^2|(\hat{u}_{n}^{1})^+(y)|^2}{|x-y|^{4}}dxdy
  -  \int_{B_{\rho}(0)} \int_{\R^N} \frac{|(u^{1})^+(x)|^2|(u^{1})^+(y)|^2}{|x-y|^{4}}dxdy+o_{n}(1) \\
  &= \int_{B_{\rho}(0)} \int_{\R^N} \frac{|(\hat{u}_{n}^{1}-u^1)^+(x)|^2|(\hat{u}_{n}^{1}-u^1)^+(y)|^2}{|x-y|^{4}}dxdy+o_{n}(1).
\end{align*}
Similarly, we can also prove the following properties hold:
\begin{equation*}
   \int_{B_{\rho}(0)} \int_{\R^N} \frac{ |(\hat{v}_{n}^{1})^+(x)|^2 (\hat{v}_{n}^{1})^{+}(y)(\hat{v}_{n}^{1}-v^1)(y)}{|x-y|^{4}}dxdy
   =\int_{B_{\rho}(0)} \int_{\R^N} \frac{ |(\hat{v}_{n}^{1}-v^1)^+(x)|^2 |(\hat{v}_{n}^{1}-v^1)^+(y)|^2}{|x-y|^{4}}dxdy+o_{n}(1),
\end{equation*}
\begin{equation*}
   \int_{B_{\rho}(0)} \int_{\R^N} \frac{ |(\hat{v}_{n}^{1})^+(x)|^2 (\hat{u}_{n}^{1})^{+}(y)(\hat{u}_{n}^{1}-u^1)(y)}{|x-y|^{4}}dxdy
   =\int_{B_{\rho}(0)} \int_{\R^N} \frac{ |(\hat{v}_{n}^{1}-v^1)^+(x)|^2 |(\hat{u}_{n}^{1}-u^1)^+(y)|^2}{|x-y|^{4}}dxdy+o_{n}(1),
\end{equation*}
and
\begin{equation*}
   \int_{B_{\rho}(0)} \int_{\R^N} \frac{ |(\hat{u}_{n}^{1})^+(x)|^2 (\hat{v}_{n}^{1})^{+}(y)(\hat{v}_{n}^{1}-v^1)(y)}{|x-y|^{4}}dxdy
   =\int_{B_{\rho}(0)} \int_{\R^N} \frac{ |(\hat{u}_{n}^{1}-u^1)^+(x)|^2 |(\hat{v}_{n}^{1}-v^1)^+(y)|^2}{|x-y|^{4}}dxdy+o_{n}(1).
\end{equation*}
Inserting the equalities above into \eqref{Guo-1}, we then get
\begin{align*}
o_{n}(1)&= \int_{B_{\rho}(0)} \nabla  \hat{u}_{n}^{1} \nabla(\hat{u}_{n}^{1}-u^1)dx + \int_{B_{\rho}(0)} \nabla  \hat{v}_{n}^{1} \nabla(\hat{v}_{n}^{1}-v^1)dx  \nonumber  \\
 & \ \ - \mu_{1} \int_{B_{\rho}(0)} \int_{\R^N} \frac{|(\hat{u}_{n}^{1}-u^1)^+(x)|^2|(\hat{u}_{n}^{1}-u^1)^+(y)|^2}{|x-y|^{4}}dxdy\\
 & \ \ - \mu_{2} \int_{B_{\rho}(0)} \int_{\R^N} \frac{ |(\hat{v}_{n}^{1}-v^1)^+(x)|^2 |(\hat{v}_{n}^{1}-v^1)^+(y)|^2}{|x-y|^{4}}dxdy \nonumber\\
 & \ \ - \beta \int_{B_{\rho}(0)} \int_{\R^N} \frac{ |(\hat{v}_{n}^{1}-v^1)^+(x)|^2 |(\hat{u}_{n}^{1}-u^1)^+(y)|^2}{|x-y|^{4}}dxdy \nonumber  \\
 & \ \ - \beta \int_{B_{\rho}(0)} \int_{\R^N} \frac{ |(\hat{u}_{n}^{1}-u^1)^+(x)|^2 |(\hat{v}_{n}^{1}-v^1)^+(y)|^2}{|x-y|^{4}}dxdy+o_{n}(1).  \nonumber
\end{align*}
By \eqref{li-3.10} and the scale invariance again, we have
\begin{align}\label{guo-3.16}
o_{n}(1)&= \int_{\R^N} | \nabla \varphi_{1,n}|^2 dx+ \int_{\R^N} |\nabla \varphi_{2,n}|^{2}dx
 - \mu_{1} \int_{\R^N} \int_{\R^N} \frac{|\varphi_{1,n}^+(x)|^2|\varphi_{1,n}^+(y)|^2}{|x-y|^{4}}dxdy  \nonumber \\
 & - \mu_{2} \int_{\R^N} \int_{\R^N} \frac{|\varphi_{2,n}^+(x)|^2|\varphi_{2,n}^+(y)|^2}{|x-y|^{4}}dxdy
 - 2\beta \int_{\R^N}\int_{\R^N} \frac{ |\varphi_{1,n}^{+}(x)|^{2} |\varphi_{2,n}^{+}(y)|^2}{|x-y|^{4}} dxdy \nonumber  \\
  &= \int_{\R^N} | \nabla \hat{\varphi}_{1,n}|^2 dx+ \int_{\R^N} |\nabla \hat{ \varphi}_{2,n}|^{2}dx
 - \mu_{1} \int_{\R^N} \int_{\R^N} \frac{|\hat{\varphi}_{1,n}^+(x)|^2|\hat{\varphi}_{1,n}^+(y)|^2}{|x-y|^{4}}dxdy  \\
 & - \mu_{2} \int_{\R^N} \int_{\R^N} \frac{|\hat{\varphi}_{2,n}^+(x)|^2|\hat{\varphi}_{2,n}^+(y)|^2}{|x-y|^{4}}dxdy
 - 2\beta \int_{\R^N}\int_{\R^N} \frac{ |\hat{\varphi}_{1,n}^{+}(x)|^{2} |\hat{\varphi}_{2,n}^{+}(y)|^2}{|x-y|^{4}} dxdy. \nonumber
\end{align}
If $\big(\hat{\varphi}_{1,n}^{+}, \hat{\varphi}_{2,n}^{+}\big)\neq (0,0)$, we define $t_n>0$  by
\begin{equation*}
  t_{n}^{2}=\frac{\|(\hat{\varphi}_{1,n},\hat{\varphi}_{2,n})\|_{H_0}^2}{ \int_{\R^N}\int_{\R^N}\frac{\mu_1|\hat{\varphi}_{1,n}^+(x)|^2|\hat{\varphi}_{1,n}^+(y)|^2+
  \mu_2|\hat{\varphi}_{2,n}^+(x)|^2|\hat{\varphi}_{2,n}^+(y)|^2+2\beta|\hat{\varphi}_{1,n}^{+}(x)|^{2} |\hat{\varphi}_{2,n}^{+}(y)|^2 }{|x-y|^{4}}   dxdy}.
\end{equation*}
Then $(t_n\hat{\varphi}_{1,n}, t_n\hat{\varphi}_{2,n} )\in \mathcal{N}_{\infty}$ and moreover,
\begin{align*}
c_{\infty}&\leq I_{\infty}(t_n\hat{\varphi}_{1,n}, t_n\hat{\varphi}_{2,n} )\\
&=\frac{1}{4}t_{n}^{2}\int_{\R^N}(|\nabla \hat{\varphi}_{1,n}|^{2}+|\nabla \hat{\varphi}_{2,n}|^{2})dx    \nonumber \\
&\leq \frac{1}{4}\frac{\|(\hat{\varphi}_{1,n},\hat{\varphi}_{2,n})\|_{H_0}^4}{ \int_{\R^N}\int_{\R^N}\frac{\mu_1|\hat{\varphi}_{1,n}^+(x)|^2|\hat{\varphi}_{1,n}^+(y)|^2+
  \mu_2|\hat{\varphi}_{2,n}^+(x)|^2|\hat{\varphi}_{2,n}^+(y)|^2+2\beta|\hat{\varphi}_{1,n}^{+}(x)|^{2} |\hat{\varphi}_{2,n}^{+}(y)|^2 }{|x-y|^{4}}   dxdy},
\end{align*}
which implies that
\begin{align*}
  & \quad \int_{\R^N}\int_{\R^N}\frac{\mu_1|\hat{\varphi}_{1,n}^+(x)|^2|\hat{\varphi}_{1,n}^+(y)|^2
 + \mu_2|\hat{\varphi}_{2,n}^+(x)|^2|\hat{\varphi}_{2,n}^+(y)|^2+2\beta|\hat{\varphi}_{1,n}^{+}(x)|^{2} |\hat{\varphi}_{2,n}^{+}(y)|^2}{|x-y|^{4}}   dxdy \\
  & \leq \frac{1}{4c_{\infty}} \big(\int_{\R^N}(|\nabla \hat{\varphi}_{1,n}|^{2}+|\nabla \hat{\varphi}_{2,n}|^{2})dx\big)^{2}.
\end{align*}
Obviously, the above inequality is also true for $\big(\hat{\varphi}_{1,n}^{+}, \hat{\varphi}_{2,n}^{+}\big)=(0,0)$.
On the other hand, by calculation we have
\begin{align}\label{guo-3.15}
 & \quad \int_{\R^N}(|\nabla \hat{\varphi}_{1,n}|^{2}+ |\nabla \hat{\varphi}_{2,n}|^{2})dx   \nonumber \\
 & = \int_{\R^N}(|\nabla \varphi_{1,n}|^{2}+ |\nabla \varphi_{2,n}|^{2})dx  \nonumber \\
 &= \int_{B_{\rho}(0)}|\nabla (\hat{u}_{n}^{1}-u^1)|^{2}dx+ \int_{B_{\rho}(0)}|\nabla (\hat{v}_{n}^{1}-v^1)|^{2}dx +o_{n}(1)  \nonumber \\
  &= \int_{B_{\rho}(0)}|\nabla \hat{u}_{n}^{1}|^{2}dx-\int_{B_{\rho}(0)}|\nabla u^1|^{2}dx+ \int_{B_{\rho}(0)}|\nabla \hat{v}_{n}^{1}|^{2}dx-\int_{B_{\rho}(0)}|\nabla v^1|^{2}dx+o_{n}(1)  \nonumber \\
  & \leq \int_{B_{\rho}(0)}|\nabla \hat{u}_{n}^{1}|^{2}dx+  \int_{B_{\rho}(0)}|\nabla \hat{v}_{n}^{1}|^{2}dx  +o_{n}(1).
\end{align}
Then by  \eqref{guo-3.16}-\eqref{guo-3.15}, we get
\begin{align*}
  o_{n}(1) \geq \Big(1-\frac{1}{4c_{\infty}} \int_{B_{\rho}(0)}(|\nabla \hat{u}_{n}^1|^{2}+|\nabla \hat{v}_{n}^1|^{2})dx\Big) \cdot \int_{\R^N}(|\nabla \hat{\varphi}_{1,n}|^{2}+|\nabla \hat{\varphi}_{2,n}|^{2})dx.
\end{align*}
Note that from \eqref{114-3.10},
\begin{equation*}
  \int_{B_{\rho}(0)}(|\nabla \hat{u}_{n}^1|^{2}+|\nabla \hat{v}_{n}^1|^{2})dx\leq L \int_{B_{1}(0)}(|\nabla \hat{u}_{n}^1|^{2}+|\nabla \hat{v}_{n}^1|^{2})dx<2c_{\infty}.
\end{equation*}
Thus for $n\to \infty$,
\begin{equation*}
  \int_{\R^N}(|\nabla \hat{\varphi}_{1,n}|^{2}+|\nabla \hat{\varphi}_{2,n}|^{2})dx \to 0.
\end{equation*}
 From \eqref{eq-114-3.14} and \eqref{eq-114-3.15}, we then get
\begin{equation*}
  \int_{B_{\rho}(0)}|\nabla (\hat{u}_{n}^{1}- u^1)^{2}|dx \to 0,  \ \ \ \
  \int_{B_{\rho}(0)}|\nabla (\hat{v}_{n}^{1}- v^1)^{2}|dx \to 0.
\end{equation*}
Thus we assert that $(u^{1},v^{1})\neq (0,0)$ due to \eqref{114-3.10}.

Furthermore, we point out that $\sigma_{n}\to 0$ as $n\to \infty$. Indeed, if not, we may assume that $\sigma_n\to \sigma^*>0$ as $n\to\infty$.
Let
\begin{equation*}\label{li-eq3.9}
(\tilde{u}_{n}^{1}(x), \tilde{v}_{n}^{1}(x)):=\big( u_{n}^{1}(x+y_n^1),  \,\, v_{n}^{1}(x+y_n^1) \big),
\end{equation*}
where $y_n^1$ is defined in \eqref{Guo-y}. By \eqref{cgl-eq3.13}, \eqref{cgl-eq3.18} and $(u^{1},v^{1})\neq (0,0)$, we get $|y_n^1|\to\infty$ as $n\to\infty$. Then, $\{ (\tilde{u}_{n}^{1}, \tilde{v}_{n}^{1}) \}$ is a Palais-Smale sequence of $I_{\lambda,\infty}$ and also bounded in $H$. We may assume that $ (\tilde{u}_{n}^{1}(x), \tilde{v}_{n}^{1}(x)) \rightharpoonup ( \tilde{u}^{1}, \tilde{v}^{1} )$ in $H$. Moreover, we can find that $( \tilde{u}^{1}, \tilde{v}^{1} )$ solves the system \eqref{S3}. Then by  Lemma \ref{lm2.5}$(a)$, we assert that $(\tilde{u}^{1},\tilde{v}^{1})=(0,0)$, which implies that $(\tilde{u}_{n}^{1},  \tilde{v}_{n}^{1})\rightharpoonup (0,0)$ in $H$. Thus, $\tilde{u}_{n}^{1}\to 0,  \tilde{v}_{n}^{1}\to 0$  in $L_{loc}^2(\R^N)$.
On the other hand, using again $(\hat{u}_{n}^1,\hat{v}_{n}^1)\rightharpoonup (u^1, v^1)\neq (0,0)$ in $H_0$, there exists a $r_0>0$ such that
$$
\int_{B_{r_0(0)}}(|u^1|^2+|v^1|^2)dx>0.
$$
Moreover,
\begin{align*}
\lim_{n\to\infty}\int_{B_{r_0\sigma^*}(0)}(|\tilde{u}_n^1|^2+|\tilde{v}_n^1|^2)dx&=\lim_{n\to\infty}\int_{B_{r_0\sigma^*}(0)}\sigma_n^{-(N-2)}(|\hat{u}_n^1(\sigma_n^{-1}x)|^2+|\hat{v}_n^1(\sigma_n^{-1}x)|^2)dx\\
&=\lim_{n\to\infty}\int_{B_{r_0\sigma^*/\sigma_n}(0)}\sigma_n^{2}(|\hat{u}_n^1(y)|^2+|\hat{v}_n^1(y)|^2)dy\\
&=\lim_{n\to\infty}\int_{B_{r_0}(0)}(\sigma^*)^{2}(|{u}^1(y)|^2+|{v}^1(y)|^2)dy>0.
\end{align*}
This contradicts with the fact that $\tilde{u}_n^1\to 0$, $\tilde{v}_n^1\to 0$ in $L^2_{loc}(\R^N)$.

 In what follows,  we want to prove that  $(u^{1},v^{1})$ is a nonzero solution of  system \eqref{S2}. For arbitrary $\varphi_1,\varphi_2 \in C_{0}^{\infty}(\R^N)$, then
$$
\langle I^{\prime}_{\infty}(\hat{u}_{n}^{1}, \hat{v}_{n}^{1}),(\varphi_{1},\varphi_{2})\rangle=\langle I^{\prime}_{\lambda,\infty}(\tilde{u}_{n}^{1}, \tilde{v}_{n}^{1}),(\hat{\varphi}_{1},\hat{\varphi}_{2})\rangle+o_{n}(1)=o_{n}(1),
$$
where $\hat{\varphi}_{j}=\sigma_n^{-\frac{N-2}{2}}\varphi_{j}(\frac{x}{\sigma_n}), j=1,2$. Then, $(u^1, v^1)$  solves the system \eqref{S2}.

Note that the bubble $\sigma_{n}^{-\frac{N-2}{2}}\big(u^{1}(\frac{x-y_n^1}{\sigma_n}), v^{1}(\frac{x-y_n^1}{\sigma_n})\big)$ does not belong to $H$. Technically, we need to modify it so that we can eliminate this bubble from the original sequence $(u_n^{1},v_{n}^{1})$ to obtain a new Palais-Smale sequence  of  $I_{\lambda,\infty}$  at a lower level.  We define
\begin{equation*}
\big(u_{n}^{2}(x), v_{n}^{2}(x)\big):=(u_{n}^{1}(x),v_{n}^{1}(x) )-\sigma_{n}^{-\frac{N-2}{2}} \Big(\varphi(\frac{x-y_n^1}{\sigma_{n}^{\frac{1}{2}}})u^{1}(\frac{x-y_n^1}{\sigma_n}),  \varphi(\frac{x-y_n^1}{\sigma_{n}^{\frac{1}{2}}})v^{1}(\frac{x-y_n^1}{\sigma_n})  \Big),
\end{equation*}
where $\varphi(x)$ is a  cut-off function satisfying $\varphi(x)=1$ if $|x|\in (0,1]$, $\varphi(x)=0$ if $|x|>2$.

In the following, we want to prove that $\{(u_{n}^{2}(x), v_{n}^{2}(x))\}$ is a sequence of Palais-Smale  sequence for $I_{\lambda,\infty}$ in $H$ and $(u_{n}^{2}(x), v_{n}^{2}(x))\rightharpoonup (0,0)$ in $H$. In fact,  in order to prove
$$
(u_{n}^{2}(x), v_{n}^{2}(x))\rightharpoonup (0,0)\quad \text{in}\ H,
$$
it is sufficient to prove that
 \begin{equation}\label{e-3.10}
\sigma_{n}^{-\frac{N-2}{2}}\Big(\varphi(\frac{x-y_n^1}{\sigma_{n}^{\frac{1}{2}}})u^{1}(\frac{x-y_n^1}{\sigma_n}),  \varphi(\frac{x-y_n^1}{\sigma_{n}^{\frac{1}{2}}})v^{1}(\frac{x-y_n^1}{\sigma_n}) \Big)\rightharpoonup (0,0) \ \  \text{in} \ \  H.
\end{equation}
Recalling that $\varphi$ is a cut-off functions, then by calculation we have
\begin{align}\label{e-3.11}
& \ \  \|  \big(\varphi(\frac{x-y_n^1}{\sigma_{n}^{\frac{1}{2}}})\sigma_{n}^{-\frac{N-2}{2}}u^{1}(\frac{x-y_n^1}{\sigma_n}),   \varphi(\frac{x-y_n^1}{\sigma_{n}^{\frac{1}{2}}})\sigma_{n}^{-\frac{N-2}{2}}v^{1}(\frac{x-y_n^1}{\sigma_n}) \big) \|_{L^{2}\times L^{2}}^{2}  \nonumber \\
& = \sigma_{n}^{2}\int_{\R^N} |\varphi(\sigma_{n}^{\frac{1}{2}}x)|^{2} |u^{1}(x)|^{2}dx+\sigma_{n}^{2}\int_{\R^N} |\varphi(\sigma_{n}^{\frac{1}{2}}x)|^2 |v^{1}(x)|^{2}dx  \nonumber \\
&=\sigma_{n}^{2}\int_{|x|\leq 2\sigma_{n}^{-\frac{1}{2}}} |u^{1}(x)|^{2}dx+\sigma_{n}^{2}\int_{|x|\leq 2\sigma_{n}^{-\frac{1}{2}}}  |v^{1}(x)|^{2}dx \nonumber \\
&\leq \sigma_{n}^{2} \Big( \int_{|x|\leq 2\sigma_{n}^{-\frac{1}{2}}}|u^1|^{2^*}dx\Big)^{\frac{2}{2^*}}\Big( \int_{|x|\leq 2\sigma_{n}^{-\frac{1}{2}}}1dx\Big)^{\frac{2^*-2}{2^*}}    \nonumber\\
&\quad+ \sigma_{n}^{2} \Big( \int_{|x|\leq 2\sigma_{n}^{-\frac{1}{2}}}|v^1|^{2^*}dx\Big)^{\frac{2}{2^*}}\Big( \int_{|x|\leq 2\sigma_{n}^{-\frac{1}{2}}}1dx\Big)^{\frac{2^{*}-2}{2^*}}  \nonumber \\
& \leq  C \sigma_n  \Big( \int_{|x|\leq 2\sigma_{n}^{-\frac{1}{2}}}|u^1|^{2^*}dx\Big)^{\frac{2}{2^{*}}}+C \sigma_n  \Big( \int_{|x|\leq 2\sigma_{n}^{-\frac{1}{2}}}|v^1|^{2^*}dx\Big)^{\frac{2}{2^{*}}}\nonumber \\
&\leq C \sigma_n  \Big( \int_{\R^N}|u^1|^{2^*}dx\Big)^{\frac{2}{2^*}}+C \sigma_n  \Big( \int_{\R^N}|v^1|^{2^*}dx\Big)^{\frac{2}{2^*}}\to 0, \ \ \text{as} \,\, n\to \infty,
\end{align}
where $\|(u,v)\|_{L^{p}\times L^{p}}^{p}=\int_{\R^N}|u|^{p}dx+\int_{\R^N}|v|^{p}dx$.
On the other hand, direct calculation shows that
\begin{align}\label{e-3.12}
&\quad \| \varphi(\frac{x-y_n^1}{\sigma_{n}^{\frac{1}{2}}})\sigma_{n}^{-\frac{N-2}{2}}u^{1}(\frac{x-y_n^1}{\sigma_n})-\sigma_{n}^{-\frac{N-2}{2}}u^{1}(\frac{x-y_n^1}{\sigma_n}) \|_{D^{1,2}}^2\nonumber\\
&=\| \varphi(\sigma_{n}^{\frac{1}{2}}x) u^{1}- u^1 \|_{D^{1,2}}^2 \to 0,
\end{align}
as $\sigma_n\to 0$. Since $\sigma_{n}^{-\frac{N-2}{2}}u^{1}(\frac{x-y_n^1}{\sigma_n})\rightharpoonup 0$ in $D^{1,2}(\R^N)$, then $\varphi(\frac{x-y_n^1}{\sigma_{n}^{\frac{1}{2}}})\sigma_{n}^{-\frac{N-2}{2}}u^{1}(\frac{x-y_n^1}{\sigma_n})\rightharpoonup 0$ in $D^{1,2}(\R^N)$. Similar to \eqref{e-3.12}, we can also prove that $\varphi(\frac{x-y_n^1}{\sigma_{n}^{\frac{1}{2}}})\sigma_{n}^{-\frac{N-2}{2}}v^{1}(\frac{x-y_n^1}{\sigma_n})\rightharpoonup 0$ in $D^{1,2}(\R^N)$. Therefore,
\begin{align}\label{add-e-3.13}
\sigma_{n}^{-\frac{N-2}{2}}\Big(\varphi(\frac{x-y_n^1}{\sigma_{n}^{\frac{1}{2}}})u^{1}(\frac{x-y_n^1}{\sigma_n}), \,\, \varphi(\frac{x-y_n^1}{\sigma_{n}^{\frac{1}{2}}})v^{1}(\frac{x-y_n^1}{\sigma_n})     \Big)
   \rightharpoonup (0,0) \,\, \text{in} \, \, H_{0}.
\end{align}
Thus, by \eqref{e-3.11} and \eqref{add-e-3.13}, we then prove that \eqref{e-3.10} is true. Notice that $\big(u_{n}^{1}(x), v_{n}^{1}(x)\big)\rightharpoonup (0,0) $ in $H$, then by \eqref{e-3.10},
\begin{equation*}
\big(u_{n}^{2}(x), v_{n}^{2}(x)\big)\rightharpoonup (0,0) \ \ \text{in} \ \ H.
\end{equation*}
Based on the properties \eqref{e-3.11} and \eqref{e-3.12}, together  with  Lemma \ref{lm3.2},  we have
\begin{equation}\label{e-3.14}
  I_{\lambda, \infty}(u_{n}^{2},v_{n}^{2})=I_{\lambda, \infty}(u_{n}^{1},v_{n}^{1})-I_{\infty}(u^1,v^1)+o_n(1)
\end{equation}
and
\begin{align*}
  &\quad\|(u_n,v_n)\|_{H}^2\\
  &=\|(u^0,v^0)\|_{H}^{2}+\|(u_n^1,v_{n}^{1})\|_{H}^{2}+o_{n}(1)\\
  %&=\|(u^0,v^0)\|_{H}^{2}+\|(\tilde{u}_n^1,\tilde{v}_{n}^{1})\|_{H}^{2}\\
  &=\|(u^0,v^0)\|_{H}^{2}+\|(u_n^2,v_{n}^{2})\|_{H}^{2}+\|(\sigma_{n}^{-\frac{N-2}{2}}u^{1}(\frac{x-y_n^1}{\sigma_n}), \sigma_{n}^{-\frac{N-2}{2}}v^{1}(\frac{x-y_n^1}{\sigma_n})   \|_{H_0}^2+o_{n}(1).
\end{align*}
Taking test function $(\tilde{\varphi},\tilde{\psi})\in H$ with $\|(\tilde{\varphi}, \tilde{\psi})\|_{H}\leq 1$, then by Lemma \ref{lm3.2} again, we have
\begin{align*}
& \ \ \langle I'_{\lambda,\infty}(u_{n}^{2},v_{n}^{2}), (\tilde{\varphi},\tilde{\psi}) \rangle\\
&=\langle  I'_{\lambda,\infty}(u_{n}^{1},v_{n}^{1}), (\tilde{\varphi},\tilde{\psi}) \rangle  \\
&\quad- \langle  I'_{\lambda,\infty}\big(\varphi(\frac{x-y_n^1}{\sigma_{n}^{\frac{1}{2}}})\sigma_{n}^{-\frac{N-2}{2}}u^{1}(\frac{x-y_n^1}{\sigma_n}),  \psi(\frac{x-y_n^1}{\sigma_{n}^{\frac{1}{2}}})\sigma_{n}^{-\frac{N-2}{2}}v^{1}(\frac{x-y_n^1}{\sigma_n}) \big), (\tilde{\varphi},\tilde{\psi}) \rangle +o_{n}(1)  \\
&=\langle  I'_{\lambda,\infty}(u_{n}^{1},v_{n}^{1}), (\tilde{\varphi},\tilde{\psi}) \rangle - \langle I'_{\infty}(u^1,v^1), \big(\sigma_n^{\frac{N-2}{2}}\tilde{\varphi}(\sigma_nx+y_n^1), \sigma_n^{\frac{N-2}{2}}\tilde{\psi}(\sigma_nx+y_n^1)\big)     \rangle+o_{n}(1) \\
&=\langle  I'_{\lambda,\infty}(u_{n}^{1},v_{n}^{1}), (\tilde{\varphi},\tilde{\psi}) \rangle+o_{n}(1).
\end{align*}
Note that $\{(u_{n}^{1},v_{n}^{1})\}$ is a sequence of Palais-Smale sequence for $I_{\lambda,\infty}$, then
\begin{align*}
  \|  I'_{\lambda,\infty}(u_{n}^{2},v_{n}^{2}) \|_{H^{-1}}& = \sup_{(\tilde{\varphi},\tilde{\psi})\in H,\, \|(\tilde{\varphi},\tilde{\psi})\|_{H}\leq 1} \langle I'_{\lambda,\infty}(u_{n}^{2},v_{n}^{2}), (\tilde{\varphi},\tilde{\psi}) \rangle    \\
  &= \|I'_{\lambda,\infty}(u_{n}^{1},v_{n}^{1})\|_{H^{-1}}+o_{n}(1) \to 0, \ \ \text{as} \ \ n\to \infty.
\end{align*}
Thus $\{(u_{n}^{2},v_{n}^{2})\}$ is a sequence of Palais-Smale sequence for $I_{\lambda,\infty}$. Moreover, by \eqref{e-3.14} we get
\begin{align*}
I(u_n,v_n)&=I_{\lambda,\infty}(u_{n}^{1},v_{n}^{1})+I(u^0,v^0)+o_{n}(1)=I_{\lambda, \infty}(u_{n}^{2},v_{n}^{2})+I_{\infty}(u^1,v^1)+I(u^0,v^0)+o_{n}(1).
%& \geq I_{\infty}(u_1,v_1)+I(u^0,v^0)+o_{n}(1).
\end{align*}

If $(u_{n}^{2},v_{n}^{2}) \to (0,0)$ in $H$, then we have done. Otherwise,  we can iterate the procedure. Taking into account that at every step $\ell$, we  then prove that
\begin{align*}
  \|(u_n,v_n)\|_{H}^{2}&=\|(u^0,v^0)\|_{H}^{2}+\|(u_n^{\ell+1},v_{n}^{ \ell+1})\|_{H}^{2}\\
  &\quad+\sum_{k=1}^{ \ell}\|\Big((\sigma_{n}^{k})^{-\frac{N-2}{2}}u^k(\frac{\cdot-y_{n}^{k}}{\sigma_{n}^{k}}),(\sigma_{n}^{k})^{-\frac{N-2}{2}}v^k(\frac{\cdot-y_{n}^{k}}{\sigma_{n}^{k}})\Big)\|_{ H_0}^{2}+o_{n}(1)
\end{align*}
and
\begin{equation*}
  I(u_n,v_n)=I(u^0,v^0)+\sum_{k=1}^{\ell}I_{\infty}(u^{k},v^k)+ I_{\lambda,\infty}(u_{n}^{\ell+1},  v_{n}^{\ell+1})+o_{n}(1).
\end{equation*}
Since $I_{\infty}(u^{k},v^k)\geq \frac{1}{4}(k_1+k_2)\mathcal{S}_{HL}^{2}$,  then
\begin{equation*}
  I(u_n,v_n) \geq I(u^0,v^0)+\frac{\ell}{4}(k_1+k_2)\mathcal{S}_{HL}^{2}+o_{n}(1).
\end{equation*}
Hence, the iteration must terminate at a finite index $\ell$ such that $\|(u_{n}^{\ell+1},v_{n}^{\ell+1})\|_{H}\to 0$. So
\begin{equation*}
  \|(u_n,v_n)\|_{H}^{2}=\|(u^0,v^0)\|_{H}^{2}+\sum_{k=1}^{\ell}\|\Big((\sigma_{n}^{k})^{-\frac{N-2}{2}}u^k(\frac{x-y_{n}^{k}}{\sigma_{n}^{k}}),(\sigma_{n}^{k})^{-\frac{N-2}{2}}v^k(\frac{x-y_{n}^{k}}{\sigma_{n}^{k}})\Big)\|_{ H_0}^{2}+o_{n}(1)
\end{equation*}
and
\begin{equation*}
I(u_n,v_n)=I(u^0,v^0)+\sum_{k=1}^{\ell}I_{\infty}(u^k,v^k)+o_{n}(1).
\end{equation*}

$(b)$\,Let us now  consider the case $\lambda_1=0, \lambda_2>0$, in this case $H:= D^{1,2}(\R^N)\times H^{1}(\R^N)$. We remark that, to obtain the proof of case $(b)$, one need to adjust some arguments displayed in the proof of case $(a)$.

Repeat the  argument in the previous case, we have $(u_n,v_n) \rightharpoonup (u^0,v^0)$ in $H$, where $(u^0,v^0)$ is a pair of solution to system \eqref{S-4}. Let  $(u_n^{1},v_{n}^{1})=(u_n,v_n)-(u^0,v^0)$, then $\{(u_n^{1},v_{n}^{1})\}$ is a sequence of Palais-Smale sequence of $I_{\lambda,\infty}$.
If $(u_n^{1},v_{n}^{1})\to (0,0)$ in $H$, then we have done. If $(u_n^{1},v_{n}^{1})\rightharpoonup (0,0)$, but $(u_n^{1},v_{n}^{1})\nrightarrow (0,0)$ in $H$. Setting
\begin{equation*}
(\tilde{u}_{n}^{1}(x), \tilde{v}_{n}^{1}(x)):=\big( u_{n}^{1}(x+y_n^1),  \,\, v_{n}^{1}(x+y_n^1) \big),
\end{equation*}
then we prove that $(\tilde{u}_{n}^{1}, \tilde{v}_{n}^{1})\rightharpoonup(\tilde{u}^1,\tilde{v}^1)$ in $H$, where $(\tilde{u}^1,\tilde{v}^1)$ solves   system \eqref{S3}. By Lemma \ref{lm2.5}$(b)$,  $(\tilde{u}^1,\tilde{v}^1)=(w^1,0)$, where $w^1$ solves  scalar  equation \eqref{eq2.19}.

If $w^1=0$, then by the same argument as in proof of case $(a)$, we can prove $(u^1, v^1)\neq (0,0)$. Next, we need to modify $(u_{n}^{2}(x), v_{n}^{2}(x))$ as follows:
\begin{equation*}
\big(u_{n}^{2}(x), v_{n}^{2}(x)\big):=(u_{n}^{1}(x),v_{n}^{1}(x) )-\sigma_{n}^{-\frac{N-2}{2}} \Big(u^{1}(\frac{x-y_n^1}{\sigma_n}),  \varphi(\frac{x-y_n^1}{\sigma_{n}^{\frac{1}{2}}})v^{1}(\frac{x-y_n^1}{\sigma_n})  \Big).
\end{equation*}
It is easy to check that $\big(u_{n}^{2}(x), v_{n}^{2}(x)\big) \rightharpoonup (0,0)$ in $D^{1,2}(\R^N)\times H^1(\R^N)$.

If $w^1\neq 0$, then  $(\tilde{u}_{n}^{1}, \tilde{v}_{n}^{1})\rightharpoonup(w^1,0)\neq (0,0)$ in $D^{1,2}(\R^N)\times H^1(\R^N)$. We directly modify
\begin{equation*}
\big(u_{n}^{2}(x), v_{n}^{2}(x)\big):=(u_{n}^{1}(x),v_{n}^{1}(x) )- (w^{1}(x-y_n^1),  0).
\end{equation*}
We claim that $y_n^1\to\infty$ as $n\to\infty$. If not,  we then suppose that $\{y_n^1\}$ is bounded. Since $u_{n}^{1}\rightharpoonup 0$ in $D^{1,2}(\R^N)$,  we have $u_{n}^1\to 0$ in $L^2_{loc}(\R^N)$. Hence,  $u_{n}^{1}(x+y_n^1)\to 0$  in $L^2_{loc}(\R^N)$, namely, $\tilde{u}_{n}^1 \to 0$  in $L^2_{loc}(\R^N)$. On the other hand,  as $\tilde{u}_{n}^{1}\rightharpoonup w^1$ in $D^{1,2}(\R^N)$, we have $\tilde{u}_{n}^1 \to w^1$  in $L^2_{loc}(\R^N)$. This is impossible due to $w^1\neq 0$.  Since $|y_n^1|\to\infty$, then $w^{1}(x-y_n^1)\rightharpoonup 0$ in $D^{1,2}(\R^N)$. So $\big(u_{n}^{2}(x), v_{n}^{2}(x)\big) \rightharpoonup (0,0)$ in $D^{1,2}(\R^N)\times H^1(\R^N)$, and  we set $z_{n}^{1}=y_{n}^{1}$ in this case.

Based on above discussion, the rest proof is standard.  We iterate the procedure as in proof in case $(a)$ and prove that
\begin{align*}
  \|(u_n,v_n)\|_{H}^{2}&=\|(u^0,v^0)\|_{H}^{2}+\|(u_n^{\ell+1},v_{n}^{\ell+1})\|_{H}^{2}+\sum_{k=1}^{\ell_1}\|( w^{k}(x-z_{n}^k),0)\|^2_{H_0}\\
  &\quad+\sum_{k=1}^{\ell_2}\|\Big((\sigma_{n}^{k})^{-\frac{N-2}{2}}u^k(\frac{x-y_{n}^{k}}{\sigma_{n}^{k}}), \,\, (\sigma_{n}^{k})^{-\frac{N-2}{2}}v^k(\frac{x-y_{n}^{k}}{\sigma_{n}^{k}})\Big)\|_{ H_0}^{2}+o_{n}(1)
\end{align*}
and
\begin{equation*}
  I(u_n,v_n)=I(u^0,v^0)+\sum_{k=1}^{\ell_1} I_{\infty}(w^k,0)+\sum_{k=1}^{\ell_2}I_{\infty}(u^{k},v^k)+ I_{\lambda,\infty}(u_{n}^{\ell+1},  v_{n}^{\ell+1})+o_{n}(1).
\end{equation*}
Recalling that $I_{\infty}(u^{k},v^k)\geq \frac{1}{4}(k_1+k_2)\mathcal{S}_{HL}^{2}$,  then
 the iteration must terminate at a finite index ${\color{red}{\ell}}$ such that $\|(u_{n}^{\ell+1},v_{n}^{\ell+1})\|_{H}\to 0$, and then $\ell=\ell_1+\ell_2$. Therefore,
\begin{align*}
  \|(u_n,v_n)\|_{H}^{2}&=\|(u^0,v^0)\|_{H}^{2}+\sum_{k=1}^{\ell_1}\|(w^{k}(x-z_{n}^k),0)\|^2_{H_0}\\
  &\quad+\sum_{k=1}^{\ell_2}\|\Big((\sigma_{n}^{k})^{-\frac{N-2}{2}}u^k(\frac{x-y_{n}^{k}}{\sigma_{n}^{k}}),\ (\sigma_{n}^{k})^{-\frac{N-2}{2}}v^k(\frac{x-y_{n}^{k}}{\sigma_{n}^{k}})\Big)\|_{H_0}^{2}+o_{n}(1)
\end{align*}
and
\begin{equation*}
I(u_n,v_n)=I(u^0,v^0)+\sum_{k=1}^{\ell_1}I_{\infty}(w^k,0)+\sum_{k=1}^{\ell_2}I_{\infty}(u^k,v^k)+o_{n}(1),
\end{equation*}

$(c)$\,The proof of case $(c)$ is similar to the proof of  case $(b)$, in the sequel we only omit it for keeping our paper a suitable length.

\end{proof}

\begin{cor}\label{co3.4}
  If $\{(u_n,v_n)\}  \subset \mathcal{N}$ is a Palais-Smale sequence for the constrained functional $I|_{\mathcal{N}}$ at level $d\in (c_{\infty}, \min\{\frac{\mathcal{S}_{HL}^{2}}{4\mu_1}, \frac{\mathcal{S}_{HL}^{2}}{4\mu_2}, 2c_{\infty} \})$, then $\{(u_n,v_n)\}$ is relatively compact.
\end{cor}

\begin{proof}
Supppse $\{(u_n,v_n)\}$ is a sequence of Palais-Smale sequence for $I|_{\mathcal{N}}$ at level $d$, that is, $I(u_n,v_n)\to d$ and $(I|_{\mathcal{N}})'(u_n,v_n) \to 0$ as $n\to +\infty$. Since $\mathcal{N}$ is a natural constraint, then  $I'(u_n,v_n)\to0$ as $n\to +\infty$.
%Indeed, for any $(u_n,v_N)\in \mathcal{N}$, we define $G(u_n,v_n)=\langle I'(u_n,v_n), (u_n,v_n) \rangle$. Since $\{(u_n,v_n)\}$ is a $(P.S.)$ sequence of $I$ constrained on $\mathcal{N}$, then there exists $\lambda_{n}\in \R$ such that
%  \begin{equation*}\label{add-gl-2.7}
%  o_{n}(1)=I'(u_n,v_n) - \lambda_{n}G'(u_n,v_n).
%  \end{equation*}
%Notice that $\{(u_n,v_n)\}$ is bounded in $H$, then
%\begin{equation*}\label{add-gl-2.8}
%o_{n}(1)=\langle I'(u_n,v_n),(u_n,v_n)   \rangle- \lambda_{n} \langle G'(u_n,v_n),(u_n,v_n)\rangle=-\lambda_{n} \langle G'(u_n,v_n),(u_n,v_n)\rangle.
%\end{equation*}
%Since $\|(u_n,v_n)\|_{H}\geq C>0$ and $V_{1}$, $V_{2}$ are nonnegative, then
%\begin{equation*}
%  \langle G'(u_n,v_n),(u_n,v_n)\rangle
%  %=-2 \int_{\R^4}\Big(\mu_1 ((u_n^1)^+)^{4}+2\beta ((u_n^1)^+)^{2}((v_n^1)^+)^2+\mu_2((v_{n}^{1})^+)^{4}\Big)dx
%  =-2 \int_{\R^4} \Big(|\nabla u_n|^{2}+|\nabla v_n|^{2}+\big(V_{1}(x)+\lambda)u_n^2+(V_2(x)+\lambda\big)v_n^2 \Big)dx < 0,
%\end{equation*}
%furthermore we have $\lambda_{n} \to 0$ as $n\to \infty$. Moreover, by the boundedness of $\{(u_n,v_n)\}$, we assert that $\{G'_{\varepsilon}(u_n,v_n)\}$ is bounded in $H$. Then $I'_{\varepsilon}(u_n,v_n)\to 0$ as $n\to \infty$.
It is easy to prove that $\{(u_n,v_n)\}$ is bounded  in $H$. Without loss of generality, we may assume that $(u_n,v_n)\rightharpoonup (u^0,v^0)$ in $H$ up to a subsequence. In the following, we will divide the proof into three cases depending on the signs of $\lambda_1$ and $\lambda_2$.

%If $(u_n,v_n) \to (u^0,v^0)$ in $H$, then we have done. If not, then we divide the proof into two cases.

{\bf Case 1:} If $\lambda_1,\lambda_2>0$.  It follows by Theorem \ref{th3.1}$(a)$  that there exist a solution $(u^0,v^0)$ of system \eqref{S-4}, $\ell$ sequences $\{(u^k,v^k)\}$ of solutions of system \eqref{S2},  $1\leq k \leq \ell$, such that
% \begin{align*}
%  \|(u_n,v_n)\|_{H}^{2}&=\|(u^0,v^0)\|_{H}^{2}+\sum_{i=1}^{l}\|(w^{i}(\cdot-y_{n}^i),0)\|^2_{D^{1,2}\times D^{1,2}}\\
%  &\quad+\sum_{j=1}^{k}\|\Big((\sigma_{n}^{j})^{-\frac{N-2}{2}}u^j(\frac{\cdot-x_{n}^{j}}{\sigma_{n}^{j}}),\ (\sigma_{n}^{j})^{-\frac{N-2}{2}}v^j(\frac{\cdot-x_{n}^{j}}{\sigma_{n}^{j}})\Big)\|_{D^{1,2}\times D^{1,2}}^{2}+o_{n}(1)
%\end{align*}
%and
\begin{equation*}
I(u_n,v_n)=I(u^0,v^0)+\sum_{k=1}^{\ell}I_{\infty}(u^k,v^k)+o_{n}(1).
\end{equation*}
If $\ell=0$, then we have done.  Otherwise, we  may suppose  $\ell>0$. Since $d<2c_{\infty}$,  then $\ell=1$, that is
 \begin{equation*}
d=I(u^0,v^0)+I_{\infty}(u^1,v^1).
\end{equation*}

 If $(u^0,v^0)\neq (0,0)$, then from Lemma \ref{lm2.6} we get
 \begin{equation*}
   d=I(u^0,v^0)+I_{\infty}(u^1,v^1)> c+c_{\infty}=2c_{\infty},
 \end{equation*}
which contradicts to the assumption $d\in (c_\infty, \min\{\frac{\mathcal{S}_{HL}^2}{4\mu_1},  \frac{\mathcal{S}_{HL}^2}{4\mu_2},  2c_{\infty}\})$.

 If $(u^0,v^0)=(0,0)$. Since $u^{1}\geq 0$ and $v^{1}\geq 0$,  by Lemma \ref{lm2.4} and the uniqueness of positive solutions for the scalar equation
\begin{equation*}
-\Delta u=\mu_j(|x|^{-4}*|u|^2)u, \ \ j=1,2, \ \  x\in \R^N,
\end{equation*}
we desert that $(u^1,v^1)$ must be, up to translation and dilation, one of the following three  solutions
$$(\sqrt{k_1}U_{1,0},\sqrt{k_2}U_{1,0} ), \ \ \ \ (\frac{1}{\sqrt{\mu_1}}U_{1,0},0), \ \ \ \  (0, \frac{1}{\sqrt{\mu_2}}U_{1,0}),$$
where $k_1=\frac{\beta-\mu_2}{\beta^2-\mu_1\mu_2}, k_2=\frac{\beta-\mu_1}{\beta^2-\mu_1\mu_2}$. Therefore,
$$
\text{either} \ \ d=c_{\infty}, \ \ \text{or} \ \ d=\frac{1}{4\mu_{1}}\|U_{1,0}\|_{D^{1,2}}^{2}=\frac{\mathcal{S}_{HL}^2}{4\mu_1}, \ \ \text{or} \ \ d=\frac{1}{4\mu_{2}}\|U_{1,0}\|_{D^{1,2}}^{2}=\frac{\mathcal{S}_{HL}^2}{4\mu_2},
$$
which also contradicts to assumption $d\in (c_\infty, \min\{\frac{\mathcal{S}_{HL}^2}{4\mu_1},  \frac{\mathcal{S}_{HL}^2}{4\mu_2},  2c_{\infty}\})$.  Therefore $(u_n,v_n) \to (u^0,v^0)$ in $H$.

{\bf Case 2:} If $\lambda_1=0,\lambda_2>0$. Then from Theorem \ref{th3.1}$(b)$  we  assert that there exist  $ \ell_2$ sequences $\{(u^k,v^k)\}$ of solutions of system \eqref{S2},  $1\leq k \leq \ell_2$, $\ell_1$ sequences  $ \{(w^{k},0) \}$ of semi-trivial solution  of system \eqref{S2}, $1\leq k \leq \ell_1$,  such that
\begin{equation*}
I(u_n,v_n)=I(u^0,v^0)+\sum_{k=1}^{\ell_1}I_{\infty}(w^k,0)+\sum_{k=1}^{\ell_2}I_{\infty}(u^k,v^k)+o_{n}(1).
\end{equation*}
\begin{equation*}
I(u_n,v_n)\geq I(u^0,v^0)+\sum_{k=1}^{\ell_1}I_{\infty}(w^k,0)+o_{n}(1).
\end{equation*}
First,  we are going to show $\ell_1=0$. We suppose by contradiction that $\ell_1>0$, then we have
\begin{equation*}
I(u_n,v_n)\geq I(u^0,v^0)+I_{\infty}(w^1,0)+o_{n}(1)\geq I_{\infty}(w^1,0)+o_{n}(1),
\end{equation*}
Which implies that $d\geq \frac{\mathcal{S}_{HL}^2}{4\mu_1}$. thus, there is a contradiction with $d\in (c_\infty, \min\{\frac{\mathcal{S}_{HL}^2}{4\mu_1}, \frac{\mathcal{S}_{HL}^2}{4\mu_2},  2c_{\infty}\})$. Hence, we get
$$
I(u_n,v_n)=I(u^0,v^0)+\sum_{k=1}^{\ell_2}I_{\infty}(u^k,v^k)+o_{n}(1).
$$
By  the same arguments in the proof of case 1, we can also get $\ell_2=0$. So $(u_n,v_n)\to (u^0,v^0)$ in $H$.

{\bf Case 3:} If $\lambda_1>0,\lambda_2=0$. In this case,  we can prove $(u_n,v_n)\to (u^0,v^0)$ in $H$ by repeating the arguments in the proof of case 2, so we omit it here.
\end{proof}

\section{Some basic estimates}\label{s4}

For $(u,v)\in H_0\setminus \{(0,0)\}$,  we introduce a barycenter function
\begin{equation*}
\xi(u,v)=\frac{\int_{\R^N}\frac{x}{1+|x|}\big(\mu_1(u^+)^{2^*}+2\beta (u^+)^{\frac{2^*}{2}}(v^+)^{\frac{2^*}{2}}+\mu_2 (v^+)^{2^*}\big)dx  }{\int_{\R^N}\big(\mu_1(u^+)^{2^*}+2\beta (u^+)^{\frac{2^*}{2}}(v^+)^{\frac{2^*}{2}}+\mu_2(v^+)^{2^*}\big)dx  }
\end{equation*}
and set
\begin{equation*}
  \gamma(u,v)=\frac{\int_{\R^N}|\frac{x}{1+|x|}-\xi(u,v)|\big(\mu_1(u^+)^{2^*}+2\beta (u^+)^{\frac{2^*}{2}}(v^+)^{\frac{2^*}{2}}+\mu_2 (v^+)^{2^*}\big)dx  }{\int_{\R^N}\big(\mu_1(u^+)^{2^*}+2\beta (u^+)^{\frac{2^*}{2}}(v^+)^{\frac{2^*}{2}}+\mu_2(v^+)^{2^*}\big)dx  }
\end{equation*}
to estimate the concentration of $(u,v)$ around its barycenter.
It is not difficult to find that the maps $\xi$ and $\gamma$ are continuous with respect to the $H_0$ norm and for $t>0$ and $(u,v)\in H_{0}$ such that $(u^+,v^+)\neq (0,0)$, 
\begin{equation}\label{e-4.1}
\xi(tu,tv)=\xi(u,v),  \ \ \gamma(tu,tv)=\gamma(u,v).
\end{equation}
\begin{lem}\label{lm4.1}
Suppose that $\lambda_1,\lambda_2\geq 0$, $\lambda=\max\{\lambda_1,\lambda_2\}>0$ and $\beta >\max\{\mu_1,\mu_2\}$. Then
\begin{equation}
c^*=\inf\{I(u,v)\mid (u,v)\in \mathcal{N},  \, \, \xi(u,v)=0, \, \, \gamma(u,v)=\frac{1}{2}  \}>c_{\infty}.
\end{equation}
\end{lem}
\begin{proof}
From Lemma \ref{lm2.6} we have $c^{*}\geq c= c_{\infty}$. We argue  by contradiction and assume that $c^*=c_{\infty}$. Then by Ekeland's variational principle, there exists a sequence of $\{(u_n,v_n)\}\subset \mathcal{N}$ such that, as $n\to \infty$,
\begin{equation}\label{e-4.3}
\begin{cases}
 \xi(u_n,v_n)\to 0, \ \    \gamma(u_n,v_n) \to \frac{1}{2}, \\
  I(u_n,v_n) \to c_{\infty}, \ \ (I|_{\mathcal{N}})'(u_n,v_n)  \to 0.
\end{cases}
\end{equation}
%Then, by Ekeland's variational principle,  there exists a sequence of $\{(u_n,v_n)\}\subset \mathcal{N}$ such that
%\begin{equation}\label{e-4.3}
%  \xi(u_n,v_n)\to 0, \ \ \gamma(u_n,v_n)\to \frac{1}{2}
%\end{equation}
%and
%\begin{equation*}
%  I(u_n,v_n)\to c_{\infty}, \ \ (I|_{\mathcal{N}})'(u_n,v_n)\to 0,
%\end{equation*}
%as $n \to \infty$. It is  easy to see that $I'(u_n,v_n)\to0$ as $n\to \infty$.
As Nehari manifold $\mathcal{N}$ is natural constraint, then $\{(u_n,v_n)\}$ is a sequence of Palais-Smale sequence of $I$ at  $c_{\infty}$. By Lemma \ref{lm2.6} and Theorem \ref{th3.1}, we deduce that there exist sequences of $\delta_n>0$ and $y_{n}\in \R^N$, and a nonzero solution $(u,v)$ of system \eqref{S2} satisfying
$I_{\infty}(u,v)=c_{\infty}$ such that
$$
\Big(\delta_{n}^{\frac{N-2}{2}}u_{n}(\delta_n x+y_n),  \delta_{n}^{\frac{N-2}{2}}v_{n}(\delta_n x+y_n)\Big) \to (u,v) \ \ \text{in} \ \ H_{0}.
$$
Then $(u,v)$ must be a positive solution of system \eqref{S2}. Otherwise, if $u\neq 0, v=0$ or $u= 0, v\neq 0$ , then $u>0$  or $v>0$ for $x\in \R^{N}$. By the uniqueness of positive solutions of
\begin{equation*}
-\Delta u=\mu_j(|x|^{-4}*|u|^2)u, \ \ j=1,2, \ \  x\in \R^N,
\end{equation*}
then $u=\frac{1}{\sqrt{\mu_1}}U_{\delta,y}$ or  $v=\frac{1}{\sqrt{\mu_2}}U_{\delta,y}$ for some $\delta>0$ and $y\in \R^N$. Furthermore, we deduce that
$$
c_{\infty}=I_{\infty}(u,0)=\frac{1}{4\mu_1}\mathcal{S}_{HL}^2, \ \ or \  c_{\infty}=I_{\infty}(0,v)=\frac{1}{4\mu_2}\mathcal{S}_{HL}^2 ,
$$
which is absurd due to
$$
c_{\infty}=\frac{k_1+k_2}{4}\mathcal{S}_{HL}^2, \ \ \frac{1}{\mu_1}>k_1+k_2, \ \ \frac{1}{\mu_2}>k_1+k_2.
$$
Thus we have
\begin{equation}\label{e-4.4}
  (u_n,v_n)=(\sqrt{k_1}U_{\delta_n,y_n}, \sqrt{k_2}U_{\delta_n,y_n})+(\varphi_n,\psi_n),
\end{equation}
where $(\varphi_n,\psi_n) \to (0,0)$ in $H_{0}$ as $n\to \infty$.

Next, we would like to show that
\begin{equation}\label{e-4.5}
\begin{cases}
\displaystyle \lim_{n\to \infty}\delta_n= \delta_0>0, \ \ & (a) \\
\displaystyle \lim_{n\to\infty}y_n=y_0\in \R^N,   \ \     & (b).
\end{cases}
\end{equation}
In order to prove \eqref{e-4.5}$(a)$, we start  to show that $\{\delta_n\}$ is bounded. Otherwise, then there exists a subsequence $\{\delta_{n_j}\}$ of $\{\delta_n\}$, still denoted by $\{\delta_n\}$, such that $\lim_{n\to\infty}\delta_n=\infty$. Then for  each $r>0$, we have
\begin{equation*}
 \lim_{n\to \infty}\int_{B_{r}(0)}| U_{\delta_n,y_n}(x)|^{2^*}dx=0.
\end{equation*}
From \eqref{e-4.3} we know that $\lim_{n\to\infty}\xi(u_n,v_n)=0$, then
\begin{align*}
  \frac{1}{2}&=\lim_{n\to \infty}\gamma(u_n,v_n) \\
  &=\lim_{n\to \infty}\frac{\int_{\R^N}\frac{|x|}{1+|x|} (\mu_1k_{1}^{\frac{2^*}{2}}+2\beta k_{1}^{\frac{2^*}{4}}k_{2}^{\frac{2^*}{4}}+\mu_2k_{2}^{2^*}) | U_{\delta_n,y_n}|^{2^*}dx }{\int_{\R^N}(\mu_1k_{1}^{\frac{2^*}{2}}+2\beta k_{1}^{\frac{2^*}{4}}k_{2}^{\frac{2^*}{4}}+\mu_2k_{2}^{2^*})
  | U_{\delta_n,y_n}|^{2^*}dx  }\\
 &=\lim_{n\to \infty}\frac{\int_{\R^N\setminus B_{r}(0)}\frac{|x|}{1+|x|}| U_{\delta_n,y_n}|^{2^*}dx  +o_{n}(1) }{\int_{\R^N\setminus B_{r}(0)}| U_{\delta_n,y_n}|^{2^*}dx +o_{n}(1)  }\\
  & \geq \frac{r}{1+r},
\end{align*}
which is impossible provided that $r>1$. Hence, $\{\delta_n\}$ must be bounded. Up to a subsequence, we  may suppose that $\delta_n \to \delta_0\geq 0$ as $n\to \infty$. If $\delta_0=0$, then for each $r>0$, we have
\begin{align*}
  \lim_{n\to\infty}\int_{\R^N\setminus B_{r}(y_n)}| U_{\delta_n,y_n}(x)|^{2^*}dx
 = \lim_{n\to\infty}\int_{\R^N\setminus B_{r}(0)}| U_{\delta_n,0}(x)|^{2^*}dx =0.
\end{align*}
By the assumption  $\lim_{n\to \infty}\xi(u_n,v_n)= 0$ again,  then  we get for each $r>0$,
\begin{align*}
\frac{|y_n|}{1+|y_n|}
&= \Big| \frac{y_n}{1+|y_n|}-\xi(u_n,v_n) \Big|+o_{n}(1) \\
%&= \Big| \frac{z_n}{1+|z_n|}-\xi(\sqrt{k_1}U_{\delta_n,z_n}, \sqrt{k_2}U_{\delta_n,z_n}) \Big|+o_{n}(1) \\
&= \frac{\int_{\R^N}\Big|\frac{x}{1+|x|}- \frac{y_n}{1+|y_n|} \Big| (\mu_1k_{1}^{\frac{2^*}{2}}+2\beta k_{1}^{\frac{2^*}{4}}k_{2}^{\frac{2^*}{4}}+\mu_2k_{2}^{2^*}) | U_{\delta_n,y_n}|^{2^*}dx   }
{\int_{\R^N}(\mu_1k_{1}^{\frac{2^*}{2}}+2\beta k_{1}^{\frac{2^*}{4}}k_{2}^{\frac{2^*}{4}}+\mu_2k_{2}^{2^*})| U_{\delta_n,y_n}|^{2^*}dx  }+o_{n}(1) \\
&=\frac{\int_{B_{r}(y_n)}\Big|\frac{x}{1+|x|}- \frac{y_n}{1+|y_n|} \Big| | U_{\delta_n,y_n}|^{2^*}dx  }{\int_{B_{r}(y_n)}| U_{\delta_n,y_n}|^{2^*}dx }+o_{n}(1)  \\
&\leq 2r+o_{n}(1).
\end{align*}
Thus, we assert that $|y_n|\to 0$ as $n\to \infty$. Repeating the calculation  above, one can easily find that
\begin{equation*}
 0\leq \gamma(u_n,v_n)\leq\frac{\int_{\R^N}\Big|\frac{x}{1+|x|}- \frac{y_n}{1+|y_n|} \Big| | U_{\delta_n,y_n}|^{2^*}dx  }{\int_{\R^N}| U_{\delta_n,y_n}|^{2^*}dx }+o_{n}(1) \leq 2r+o_{n}(1),
\end{equation*}
from which it follows that
$$\lim_{n\to \infty}\gamma(u_n,v_n)=0.$$
Obviously, there is a contradiction with\eqref{e-4.3}. Thus, \eqref{e-4.5}$(a)$ holds true.

Now we shall to show that $\{y_n\} \subset \R^N$ is bounded. We argue by contradiction and suppose that $\lim_{n\to \infty}|y_n|=\infty$. Then for each $\varepsilon>0$, we can find $r_0=r_{0}(\varepsilon)>0$ such that for all $n\in \N$,
\begin{align}\label{e-4.8}
   \int_{\R^N \setminus B_{r_0}(y_n)}| U_{\delta_0,y_n}|^{2^*}dx  =\int_{\R^N \setminus B_{r_0}(0)}| U_{\delta_0,0}|^{2^*}dx
    <\varepsilon.
\end{align}
 For such a fixed $r_0$, there  exists $N^*\in \N$  such that for each $n\geq N^*$  and $x\in B_{r_0}(y_n)$,
\begin{equation}\label{e-4.7}
  \Big|\frac{x}{1+|x|}-\frac{y_n}{1+|y_n|} \Big|<\varepsilon.
\end{equation}
By  \eqref{e-4.8} and \eqref{e-4.7}, for  $n$ large enough, we prove that
\begin{align*}
  \Big|\xi(u_n,v_n)-\frac{y_n}{1+|y_n|} \Big|
  &\leq \frac{\int_{\R^N}\Big|\frac{x}{1+|x|}- \frac{y_n}{1+|y_n|} \Big| | U_{\delta_0,y_n}|^{2^*}dx }{\int_{\R^N}| U_{\delta_0,y_n}|^{2^*}dx } \\
  &\leq \frac{ \varepsilon \int_{B_{r_0}(y_n)}| U_{\delta_0,y_n}|^{2^*}dx +o_{n}(1)}{ \int_{\R^N}| U_{\delta_0,y_n}|^{2^*}dx   } + \frac{2 \int_{\R^N  \setminus B_{r_0}(y_n)}| U_{\delta_0,y_n}|^{2^*}dx}{  \int_{\R^N}| U_{\delta_0,y_n}|^{2^*}dx} \\
  & \leq  C\varepsilon,
\end{align*}
where positive constant  $C$ is independent of $y_n$ and $r_0$. Based on the inequality above together with the assumption $\lim_{n\to \infty}|y_n|=\infty$, we get
\begin{equation*}
  \lim_{n\to \infty}|\xi(u_n,v_n)| =1,
\end{equation*}
which  contradicts with \eqref{e-4.3}. Therefore, we have proved that $\{y_n\}$ is bounded in $\R^N$ and  then \eqref{e-4.5}$(b)$ holds true.

Recalling that $V_1(x)$, $V_2(x)$ are nonnegative and $\lambda_1,\lambda_2 \geq 0$ with $\lambda=\max\{\lambda_1,\lambda_2\}>0$, by \eqref{e-4.4} and \eqref{e-4.5} we get
%\begin{equation*}
%(u_n,v_n)\to \big(\sqrt{k_1}U_{\delta_0,z_0}, \sqrt{k_2}U_{\delta_0,z_0}\big)  \ \ \text{in} \ \ H_0.
%\end{equation*}
% then
\begin{align*}
  4c_{\infty}
  %&=(k_1+k_2)\int_{\R^N}\int_{\R^N}\frac{|U_{\delta_0,z_0}(x)|^{2}|U_{\delta_0,z_0}(y)|^{2}}{|x-y|^4}dxdy\\
%  &=\lim_{n\to \infty}\int_{\R^N}\int_{\R^N}\frac{\mu_1|u_n^+(x)|^2|u_n^{+}(y)|^2+2\beta |u_n^+(x)|^2|v_n^+(y)|^2+\mu_2 |v_n^+(x)|^2|v_n^{+}(y)|^2 }{|x-y|^4}dxdy\\
  &=\lim_{n\to \infty} \Big[ \int_{\R^N}(|\nabla u_n|^{2}+|\nabla v_n|^{2})dx+\int_{\R^N}(V_{1}(x)+\lambda_1) |u_n|^{2}dx+ \int_{\R^N}(V_{2}(x)+\lambda_2) |v_n|^{2}dx \Big]\\
  & \geq  (k_1+k_2)\int_{\R^N}\nabla|U_{\delta_0,y_0}|^{2}dx + \int_{\R^N}\big(k_1V_1(x)+k_2V_2(x)\big)|U_{\delta_0,y_0}|^{2}dx  \\
  & \quad + \int_{B_{r_0}(y_0)}(k_1\lambda_1+ k_2\lambda_2)|U_{\delta_0,y_0}|^{2}dx \\
  &>(k_1+k_2)\mathcal{S}_{HL}^2=4c_{\infty},
\end{align*}
which is impossible and thus  the  proof is completed.
\end{proof}

\begin{lem}
Suppose that $\lambda_1,\lambda_2\geq 0$, $\lambda=\max\{\lambda_1,\lambda_2\}>0$ and $\beta >\max\{\mu_1,\mu_2\}$. Then
\begin{equation}\label{e-4.9}
c^{**}=\inf\{I(u,v)\mid (u,v)\in \mathcal{N},  \, \, \xi(u,v)=0, \, \, \gamma(u,v)\geq\frac{1}{2}  \}>c_{\infty}.
\end{equation}
\end{lem}
\begin{proof}
By an analogous argument in the proof of Lemma \ref{lm4.1}, we succeed in proving \eqref{e-4.9}.  Here we only give a brief proof for the reader's convenience. Obviously,
$c^{**}\geq c_{\infty}$.  If $c^{**}=c_{\infty}$, then it follows by the Ekeland's principle  that there exists a sequence $\{(u_n,v_n)\}\subset \mathcal{N}$ such that, as $n\to \infty$,
\begin{equation}\label{e-4.10}
 \xi(u_n,v_n)\to 0, \ \ \gamma(u_n,v_n)\geq \frac{1}{2}
\end{equation}
and
\begin{equation}\label{e-4.11}
  I(u_n,v_n)\to c_{\infty}, \ \ (I|_{\mathcal{N}})'(u_n,v_n)\to 0.
\end{equation}
Moreover, $\{(u_n,v_n)\}$ is a sequence of Palais-Smale sequence of $I$ at level $c_{\infty}$. By the same computations made in  Lemma \ref{lm4.1}, we then have
\begin{equation*}
(u_n,v_n)=(\sqrt{k_1}U_{\delta_n,y_n}, \sqrt{k_2}U_{\delta_n,y_n})+(\tilde{\varphi}_n,\tilde{\psi}_n)
\end{equation*}
where $\delta_{n} >0$, $y_n\in \R^N$ and $(\tilde{\varphi}_n,\tilde{\psi}_n) \to 0$ in $H_{0} $.

We claim that $\{\delta_n\}$ is bounded. Otherwise,   $\delta_{n} \to \infty$.  Then
\begin{align*}
  c_{\infty}&=\lim_{n\to \infty} I(u_n)-\frac{1}{4}\langle I'(u_n,v_n), (u_n,v_n)   \rangle  \\
  &=\lim_{n\to \infty} \frac{1}{4}\Big[ \int_{\R^N}(|\nabla u_n|^{2}dx+|\nabla v_n|^{2})dx+\int_{\R^N}(V_{1}(x)+\lambda_1) |u_n|^{2}dx \\
  & \qquad + \int_{\R^N}(V_{2}(x)+\lambda_2) |v_n|^{2}dx \Big]\\
  &  \geq  \liminf_{n\to \infty} \frac{1}{4}\Big[ \int_{\R^N}(|\nabla u_n|^{2}+|\nabla v_n|^{2})dx+ \lambda_1\int_{B_{\delta_n}(y_n)} |u_n|^{2}dx+ \lambda_2\int_{B_{\delta_n}(y_n)} |v_n|^{2}dx \Big]    \\
  &   \geq  \frac{k_1+k_2}{4}\mathcal{S}_{HL}^2+ \liminf_{n\to \infty} \frac{1}{4} \Big( \lambda_1 \int_{B_{\delta_n}(y_n)} |u_n|^{2}dx+ \lambda_2\int_{B_{\delta_n}(y_n)} |v_n|^{2}dx \Big)               \\
  & \geq \frac{k_1+k_2}{4}\mathcal{S}_{HL}^2+ \frac{\lambda_1 k_1}{4}\delta_{n}^{2}\int_{B_{1}(0)}|U_{1,0}|^{2}dx +\frac{\lambda_2 k_2}{4}\delta_{n}^{2}\int_{B_{1}(0)}|U_{1,0}|^{2}dx +o(\delta_n^2) \to \infty,
\end{align*}
which leads to a contradiction.
Hence, $\{\delta_n\}$ is bounded,  and up to a subsequence, we may assume that $\lim_{n\to \infty}\delta_{n}=\delta^*$.  Working again as the proof of \eqref{e-4.5} in Lemma \ref{lm4.1}, we can also prove that  $\delta^*>0$ and bounded sequence $\{y_n\}$  satisfies  $y_n \to y^*$.  Thus,
\begin{align*}
  c_{\infty}&=\lim_{n\to \infty}\frac{1}{4} \Big[ \int_{\R^N}(|\nabla u_n|^{2}+|\nabla v_n|^{2})dx+\int_{\R^N}(V_{1}(x)+\lambda_1 ) u_n^{2}dx+ \int_{\R^N}(V_{2}(x)+\lambda_2 ) v_n^{2}dx \Big]\\
  & \geq \frac{(k_1+k_2)}{4}\mathcal{S}_{HL}^2+\frac{\lambda_1 k_1}{4}\int_{B_{\delta^*}(y^*)}|U_{\delta^*,y^*}|^{2}dx +\frac{\lambda_2 k_2}{4}\int_{B_{\delta^*}(y^*)}|U_{\delta^*,y^*}|^{2}dx \\
  &>\frac{k_1+k_2}{4}\mathcal{S}_{HL}^2=c_{\infty}.
\end{align*}
Obviously, this is impossible. Therefore, $c^{**}=c_{\infty}$ cannot occur and  then we complete  the  whole proof.
\end{proof}

Note that  $k_1=\frac{\beta-\mu_2}{\beta^{2}-\mu_1\mu_2}$,  $k_2=\frac{\beta-\mu_1}{\beta^{2}-\mu_1\mu_2}$ and $\beta>\max\{\mu_1,\mu_2\}$. From assumption $(A_3)$, we have
 \begin{align*}
 0& < \frac{\beta-\mu_2}{2\beta-\mu_1-\mu_2}C(N,4)^{-\frac{1}{2}}\|V_1\|_{L^{\frac{N}{2}}}+\frac{\beta-\mu_1}{2\beta-\mu_1-\mu_2}C(N,4)^{-\frac{1}{2}}\|V_2\|_{L^{\frac{N}{2}}}  \\
 & <
 \min \Big\{ \sqrt{\frac{\beta^2-\mu_1\mu_2}{\mu_1(2\beta-\mu_1-\mu_2)}},  \sqrt{\frac{\beta^2-\mu_1\mu_2}{\mu_2(2\beta-\mu_1-\mu_2)}}, \sqrt{2}  \Big\} \mathcal{S}_{HL}-\mathcal{S}_{HL}.
 \end{align*}
%\begin{align}\label{e-4.13}
%  0&< \frac{\beta-\mu_2}{2\beta-\mu_1-\mu_2}\|V_1\|_{L^2}+\frac{\beta-\mu_1}{2\beta-\mu_1-\mu_2}\|V_2\|_{L^2} \nonumber \\
%  & < \min \Big\{ \sqrt{\frac{\beta^2-\mu_1\mu_2}{\mu_1(2\beta-\mu_1-\mu_2)}},  \sqrt{\frac{\beta^2-\mu_1\mu_2}{\mu_2(2\beta-\mu_1-\mu_2)}}, 2^{\frac{a}{2}}  \Big\} \mathcal{S}-\mathcal{S},
%\end{align}
Let us consider
$$
f(t)=2^{-\frac{1-t}{2}}\min \Big\{\sqrt{\frac{\beta^2-\mu_1\mu_2}{\mu_1(2\beta-\mu_1-\mu_2)}},
   \sqrt{\frac{\beta^2-\mu_1\mu_2}{\mu_2(2\beta-\mu_1-\mu_2)}}, \sqrt{2} \Big\} \mathcal{S}_{HL}-\mathcal{S}_{HL}.
$$
It is not difficult to prove that $f(t)$ is continuous and increasing in $[0,1]$. Moreover, $f(0)\leq 0$ and $f(1)>0$. Let us choose a constant
$a\in (0,1)$  such that $ f(a)>0$, and
\begin{align}\label{guo-e-4.15}
 0< & \frac{\beta-\mu_2}{2\beta-\mu_1-\mu_2}C(N,4)^{-\frac{1}{2}}\|V_1\|_{L^{\frac{N}{2}}}+\frac{\beta-\mu_1}{2\beta-\mu_1-\mu_2}C(N,4)^{-\frac{1}{2}}\|V_2\|_{L^{\frac{N}{2}}} \nonumber \\
   &= 2^{-\frac{1-a}{2}}\min \Big\{\sqrt{\frac{\beta^2-\mu_1\mu_2}{\mu_1(2\beta-\mu_1-\mu_2)}},
   \sqrt{\frac{\beta^2-\mu_1\mu_2}{\mu_2(2\beta-\mu_1-\mu_2)}}, \sqrt{2} \Big\} \mathcal{S}_{HL}-\mathcal{S}_{HL}.
\end{align}
Next, we fix a constant $\bar{c}$ such that
\begin{equation}\label{e-4.14}
c_{\infty}<\bar{c}<\min\{\frac{c^*+c_{\infty}}{2},  2^{1-a}c_{\infty}   \}.
\end{equation}
We remark that, as $f(a)>0$, then
\begin{align*}
  & \qquad  2^{-\frac{1-a}{2}}\min \Big\{\sqrt{\frac{\beta^2-\mu_1\mu_2}{\mu_1(2\beta-\mu_1-\mu_2)}},
   \sqrt{\frac{\beta^2-\mu_1\mu_2}{\mu_2(2\beta-\mu_1-\mu_2)}}, \sqrt{2} \Big\} \mathcal{S}_{HL}-\mathcal{S}_{HL}>0 \\
 &\Rightarrow  2^{-\frac{1-a}{2}}\min \Big\{\sqrt{\frac{\beta^2-\mu_1\mu_2}{\mu_1(2\beta-\mu_1-\mu_2)}},
   \sqrt{\frac{\beta^2-\mu_1\mu_2}{\mu_2(2\beta-\mu_1-\mu_2)}}, \sqrt{2} \Big\} >1 \\
   & \Rightarrow  \min \Big\{\sqrt{\frac{\beta^2-\mu_1\mu_2}{\mu_1(2\beta-\mu_1-\mu_2)}},
   \sqrt{\frac{\beta^2-\mu_1\mu_2}{\mu_2(2\beta-\mu_1-\mu_2)}}, \sqrt{2} \Big\} > 2^{\frac{1-a}{2}} \\
  & \Rightarrow  \min \Big\{\frac{\beta^2-\mu_1\mu_2}{\mu_1(2\beta-\mu_1-\mu_2)},
   \frac{\beta^2-\mu_1\mu_2}{\mu_2(2\beta-\mu_1-\mu_2)}, 2 \Big\} > 2^{1-a} \\
   & \Rightarrow  \min\{\frac{\mathcal{S}_{HL}^{2}}{4\mu_1}, \frac{\mathcal{S}_{HL}^2}{4\mu_2}, 2c_{\infty} \} \geq   2^{1-a}c_{\infty}.
\end{align*}
In the following,  $\vartheta(x)$ is a function that belongs to $C_{0}^{\infty}(B_{1}(0))$ and satisfies the  properties below:
\begin{equation}\label{e-4.15}
  \begin{cases}
$(i)$~\vartheta(x)\in C_{0}^{\infty}(B_{1}(0)),  \,\, \vartheta(x)\geq 0, \forall  \,  x\in B_{1}(0) ;\\
$(ii)$~\vartheta(x) \,\, \text{is symmetric and}~~|x_1|\leq |x_2|\Rightarrow\vartheta(x_1)>\vartheta(x_2);\\
$(iii)$~ \displaystyle \int_{\R^N}|\nabla \vartheta|^{2}dx=\int_{\R^N}(|x|^{-4}*|\vartheta|^{2})|\vartheta|^{2}dx>\mathcal{S}_{HL}^{2}; \\
$(iv)$~(\sqrt{k_1}\vartheta, \sqrt{k_2}\vartheta)\in \mathcal{N}_{\infty}, \ \ I_{\infty}(\sqrt{k_1}\vartheta, \sqrt{k_2}\vartheta)=\Sigma\in (c_{\infty}, \bar{c}).
  \end{cases}
\end{equation}
Moreover, for each $\delta>0$ and $y\in \R^N$, we set
\begin{equation*}
\vartheta_{\delta,y}=
  \begin{cases}
   \delta^{-\frac{N-2}{2}}\vartheta(\frac{x-y}{\delta}), &x\in B_{\delta}(y); \\
   0, \ \ &x\notin B_{\delta}(y).
  \end{cases}
\end{equation*}
We remark  that by the definition of $\vartheta_{\delta,y}$ and by variable change, it follows that for each $\delta>0$ and $y\in \R^N$,
\begin{equation}
  \|\vartheta\|_{L^{2^*}}=\|\vartheta\|_{L^{2^*}(B_1(0))}=\|\vartheta_{\delta,y}\|_{L^{2^*}(B_{\delta}(y))}=\|\vartheta_{\delta,y}\|_{L^{2^*}}.
\end{equation}

\begin{lem}\label{lm4.4}
Suppose that $(A_1)$-$(A_2)$ hold, then\\
$(a)$~$\displaystyle\limsup_{\delta\to 0}\{\int_{\R^N} V_{j}(x)|\vartheta_{\delta,y}|^{2}dx: y\in \R^N\}=0$, $j=1,2$;\\
$(b)$~$\displaystyle\limsup_{\delta \to  \infty} \{\int_{\R^N} V_{j}(x)|\vartheta_{\delta,y}|^{2}dx:y\in \R^N  \}=0 $, $j=1,2$; \\
$(c)$~$\displaystyle\limsup_{r\to  \infty} \{\int_{\R^N} V_{j}(x)|\vartheta_{\delta,y}|^{2}dx: |y|=r, \, y\in \R^N, \delta>0 \}=0$, $j=1,2$.
\end{lem}

 \begin{proof}
Let  $y\in \R^N$ be chosen arbitrarily  , then for each $\delta>0$,  by H\"{o}lder inequality we have
\begin{align*}
\int_{\R^N}V_j(x)|\vartheta_{\delta,y}(x)|^{2}dx&=\int_{B_{\delta}(y)}V_j(x)|\vartheta_{\delta,y}(x)|^{2}dx \nonumber \\
&\leq \|V_j\|_{L^{\frac{N}{2}}(B_{\delta}(y))}\|\vartheta\|_{L^{2^*}(B_{1}(0))}^{2} \leq C\|V_j\|_{L^{\frac{N}{2}}(B_{\delta}(y))},
\end{align*}
where $C$ is  a positive constant  independent of $\delta$. Then,
\begin{equation}\label{guo-e-4.17}
  \sup_{y\in \R^N} \int_{\R^N}V_j(x)|\vartheta_{\delta,y}(x)|^{2}dx
  \leq C \sup_{ y\in \R^N}  \|V_j\|_{L^{\frac{N}{2}}(B_{\delta}(y))}.
\end{equation}
Since
 \begin{equation*}
   \lim_{\delta\to0} \|V_j\|_{L^{\frac{N}{2}}(B_{\delta}(y))}=0  \,\, \text{uniformly in } \, y\in \R^N,
 \end{equation*}
then Lemma \ref{lm4.4} $(a)$ follows from \eqref{guo-e-4.17}.

To prove $(b)$, we  fix arbitrarily $y\in \R^N$ and note for  each $\rho>0$ and $\delta>0$,
\begin{align*}
  \int_{\R^N}V_j(x)|\vartheta_{\delta,y}|^{2}dx&=\int_{B_{\rho}(0)}V_j(x)|\vartheta_{\delta,y}|^{2}dx
+\int_{\R^N \setminus B_{\rho}(0)}V_j(x)|\vartheta_{\delta,y}|^{2}dx \\
  &\leq \|V_j\|_{L^{\frac{N}{2}}(B_{\rho}(0))}\|\vartheta_{\delta,y}\|_{L^{2^*}(B_{\rho}(0))}^{2}
+\|V_j\|_{L^{\frac{N}{2}}(\R^N \setminus B_{\rho}(0))}\|\vartheta_{\delta,y}\|_{L^{2^*}(\R^N \setminus B_{\rho}(0))}^{2} \\
& \leq  \|V_j\|_{L^{\frac{N}{2}}(B_{\rho}(0))}\sup_{y\in\R^N}\|\vartheta_{\delta,y}\|_{L^{2^*}(B_{\rho}(0))}^{2}
+\widetilde{C}\|V_j\|_{L^{\frac{N}{2}}(\R^N \setminus B_{\rho}(0))},
\end{align*}
in which constant $\widetilde{C}$ is independent of $\delta $ and $\rho$. By the fact that
\begin{equation*}
  \lim_{\delta \to  \infty}\|\vartheta_{\delta,y}\|_{L^{2^*}(B_{\rho}(0))}=0, \,\, \text{uniformly in} \, y \in \R^N,
\end{equation*}
then   we get for every  $\rho>0$,
\begin{equation*}
  \lim_{\delta \to  \infty}\sup_{y\in \R^N}\int_{\R^N}V_j(x)|\vartheta_{\delta,y}|^{2}dx\leq C\|V_j\|_{L^{\frac{N}{2}}(\R^N\backslash B_{\rho}(0))}.
\end{equation*}
Taking the limit $\rho \to \infty$ in the inequality above, then we prove that Lemma \ref{lm4.4} $(b)$ is true.

To verify $(c)$, we argue by contradiction and assume that there exist sequences of $\{y_n\}$ and $\{\delta_n\}$ with $\delta_{n}\in \R^+\setminus \{0\}$, such that
\begin{equation}\label{guo-e-4.18}
  |y_n|\to  \infty
\end{equation}
and
\begin{equation}\label{guo-e-4.19}
  \lim_{n\to \infty}\int_{\R^N}V_j(x)|\vartheta_{\delta_n,y_n}|^{2}dx>0.
\end{equation}
As a direct result of Lemma \ref{lm4.4} $(a)$ and Lemma \ref{lm4.4} $(b)$, we conclude that $\displaystyle\lim_{n\to \infty}\delta_{n}=\tilde{\delta}>0$. By \eqref{guo-e-4.18} and the assumption that $V_j(x)\in L^{\frac{N}{2}}(\R^N)$, we have
\begin{equation*}
  \lim_{n\to \infty} \|V_j\|_{L^{\frac{N}{2}}(B_{\delta_n}(y_n))}=0.
\end{equation*}
Furthermore, by H\"{o}lder inequality  we get
\begin{equation*}
  \lim_{n\to \infty}\int_{\R^N}V_j(x)|\vartheta_{\delta_n,y_n}|^{2}dx\leq \lim_{n\to \infty} \Big[  \|V_j\|_{L^{\frac{N}{2}}(B_{\delta_n}(y_n))} \cdot \|\vartheta_{\delta_n,y_n}\|_{L^{2^*}(B_{\delta_n}(y_n))}^{2}\Big]=0,
\end{equation*}
which contradicts with \eqref{guo-e-4.19}.  So Lemma \ref{lm4.4} $(c)$ holds true.
\end{proof}

\begin{lem}\label{lm4.5}
Let $\langle x|y\rangle_{\R^N}$ be the inner product of vector $x,y \in \R^N$. For fixed $r>0$,  the following relations hold: \\
$(a)$~$\displaystyle\lim_{\delta\to 0}\{ \gamma(\sqrt{k_1}\vartheta_{\delta,y}, \sqrt{k_2}\vartheta_{\delta,y}): y\in \R^N\}=0$;\\
$(b)$~$\displaystyle\lim_{\delta \to  \infty} \{ \gamma(\sqrt{k_1}\vartheta_{\delta,y}, \sqrt{k_2}\vartheta_{\delta,y}): |y|\leq r\}=1 $. \\
$(c)$~$\displaystyle \langle \xi( \sqrt{k_1}\vartheta_{\delta,y}, \sqrt{k_2}\vartheta_{\delta,y}) \mid y  \rangle_{\R^N}>0$, $\forall y \in \R^N\backslash\{0\}$, $\forall \delta>0$.
\end{lem}
\begin{proof}
$(a)$ For any $\delta >0$ and $y\in \R^N$, we get
\begin{align*}
  |\frac{y}{1+|y|}-\xi(\sqrt{k_1}\vartheta_{\delta,y},\sqrt{k_2}\vartheta_{\delta,y})|
 \leq \frac{\int_{\R^N} \big|\frac{y}{1+|y|}- \frac{x}{1+|x|}\big|| \vartheta_{\delta,y}|^{2^*}dx}{\int_{\R^N}| \vartheta_{\delta,y}|^{2^*}dx}\leq 2\delta,
\end{align*}
where in the last inequality we use the property that
$
|\frac{y}{1+|y|}-\frac{x}{1+|x|}|<2\delta
$
for any $x\in B_{\delta}(y)$. Then
\begin{align*}
  0 &\leq \gamma(\sqrt{k_1}\vartheta_{\delta,y}, \sqrt{k_2}\vartheta_{\delta,y} ) \\
  &  = \frac{\int_{\R^N} \big| \frac{x}{1+|x|}-\xi(\sqrt{k_1}\vartheta_{\delta,y}, \sqrt{k_2}\vartheta_{\delta,y})\big|| \vartheta_{\delta,y}|^{2^*}dx}{\int_{\R^N}| \vartheta_{\delta,y}|^{2^*}dx} \\
  & \leq  \frac{\int_{\R^N} \big| \frac{x}{1+|x|}-\frac{y}{1+|y|}\big|| \vartheta_{\delta,y}|^{2^*}dx}{\int_{\R^N}| \vartheta_{\delta,y}|^{2^*}dx} +  \frac{\int_{\R^N} \big| \frac{y}{1+|y|}-\xi(\sqrt{k_1}\vartheta_{\delta,y}, \sqrt{k_2}\vartheta_{\delta,y})\big|| \vartheta_{\delta,y}|^{2^*}dx}{\int_{\R^N}| \vartheta_{\delta,y}|^{2^*}dx} \\
  &\leq 4\delta,
\end{align*}
which implies that
$
\lim_{\delta \to 0} \gamma(\sqrt{k_1}\vartheta_{\delta,y}, \sqrt{k_2}\vartheta_{\delta,y} )=0
$
uniformly in $y\in \R^N$.

$(b)$ Let us first show that
\begin{equation}\label{add-guo}
  \lim_{\delta \to \infty} \xi (\sqrt{k_1}\vartheta_{\delta,y}, \sqrt{k_2}\vartheta_{\delta,y})=0,
\end{equation}
uniformly for $y\in B_{r}(0)$. In fact, since $\vartheta_{\delta,0}$ is radially symmetric, then
\begin{equation*}
  \frac{\int_{\R^N} \frac{x}{1+|x|}| \vartheta_{\delta,0}|^{2^*}dx}{\int_{\R^N}| \vartheta_{\delta,0}|^{2^*}dx}=0.
\end{equation*}
So, by calculations we get
\begin{align*}
 |\xi(\sqrt{k_1}\vartheta_{\delta,y}, \sqrt{k_2}\vartheta_{\delta,y})|
  &=\Big| \frac{\int_{\R^N} \frac{x}{1+|x|} (\mu_1 k_{1}^{\frac{2^*}{2}}+2\beta k_{1}^{\frac{2^*}{4}}k_{2}^{\frac{2^*}{4}}+\mu_2 k_{2}^{\frac{2^*}{2}}) |\vartheta_{\delta,y}|^{2^*}dx}{\int_{\R^N} (\mu_1 k_{1}^{\frac{2^*}{2}}+2\beta k_{1}^{\frac{2^*}{4}}k_{2}^{\frac{2^*}{4}}+\mu_2 k_{2}^{\frac{2^*}{2}})| \vartheta_{\delta,y}|^{2^*}dx}\Big|  \\
  &=\Big| \frac{\int_{\R^N} \frac{x}{1+|x|} |\vartheta_{\delta,y}|^{2^*}dx}{\int_{\R^N} | \vartheta_{\delta,y}|^{2^*}dx}\Big|  \\
  &= \Big| \frac{\int_{\R^N} \frac{x}{1+|x|} (| \vartheta_{\delta,y}|^{2^*}-| \vartheta_{\delta,0}|^{2^*})dx}{\int_{\R^N}| \vartheta_{\delta,y}|^{2^*}dx}\Big| \\
  & \leq \frac{\int_{\R^N}  | \vartheta_{\delta,y}^{2^*}-\vartheta_{\delta,0}^{2^*}|dx}{\int_{\R^N}| \vartheta_{\delta,y}|^{2^*}dx} \\
  &\leq  \frac{1}{\int_{\R^N}| \vartheta_{1,\frac{y}{\delta}}|^{2^*}dx} \int_{\R^N}| \vartheta_{1,\frac{y}{\delta}}^{2^*}-\vartheta_{1,0}^{2^*}|dx \to 0,
\end{align*}
as $\delta \to \infty$, uniformly for $y\in B_{r}(0)$. So, \eqref{add-guo} holds true.

Next, for any $\varepsilon>0$, we fix a constant $\rho=\rho(\varepsilon)>0$ such that $\frac{1}{1+\rho}< \frac{\varepsilon}{3}$. Then for such a  fixed constant $\rho >0$, it is easy to prove that
$$
\lim_{\delta \to \infty} \int_{B_{\rho}(0)}| \vartheta_{\delta,y}|^{2^*}dx=0 
$$
uniformly in $y\in B_{r}(0)$.
Thus, there exists a constant $\delta_0>0$ such that for each $\delta>\delta_0$,
$$
\frac{1}{\|\vartheta\|_{L^{2^*}}^{2^*}} \int_{B_{\rho}(0)}| \vartheta_{\delta,y}|^{2^*}dx< \frac{\varepsilon}{3} 
$$
uniformly in $y\in B_{r}(0)$.
For fixed constant $\delta_0$, considering \eqref{add-guo}, we can also deduce that for  each $\delta \in (\delta_0, \infty)$ and $y\in B_{r}(0)$, there holds
$$
|\xi(\sqrt{k_1}\vartheta_{\delta,y}, \sqrt{k_2}\vartheta_{\delta,y})|<\frac{\varepsilon}{3}.
$$
By the definition of function $\gamma$, we have
\begin{align*}
   \gamma(\sqrt{k_1}\vartheta_{\delta,y}, \sqrt{k_2}\vartheta_{\delta,y})
  \leq 1+ |\xi(\sqrt{k_1}\vartheta_{\delta,y}, \sqrt{k_2}\vartheta_{\delta,y})|
   \leq 1+ \frac{\varepsilon}{3}.
\end{align*}
On the other hand, for the same assumptions on $\delta$ and $y$, that is $\delta \in (\delta_0,\infty)$ and $y\in B_{r}(0)$, we get
\begin{align*}
  \gamma(\sqrt{k_1}\vartheta_{\delta,y}, \sqrt{k_2}\vartheta_{\delta,y})
  & = \frac{  \int_{\R^N}\Big| \frac{x}{1+|x|}-\xi(\sqrt{k_1}\vartheta_{\delta,y}, \sqrt{k_2}\vartheta_{\delta,y}) \Big| | \vartheta_{\delta,y}|^{2^*}dx}{\int_{\R^N}|\vartheta_{\delta,y}|^{2^*}dx} \\
 & \geq \frac{  \int_{\R^N} \frac{|x|}{1+|x|}  |\vartheta_{\delta,y}|^{2^*}dx}{\int_{\R^N}|\vartheta_{\delta,y}|^{2^*}dx} -  |\xi(\sqrt{k_1}\vartheta_{\delta,y}, \sqrt{k_2}\vartheta_{\delta,y})|            \\
  & \geq \frac{  \int_{\R^N\setminus B_{\rho}(0)} \frac{|x|}{1+|x|}  |\vartheta_{\delta,y}|^{2^*}dx}{\int_{\R^N}|\vartheta_{\delta,y}|^{2^*}dx} -  |\xi(\sqrt{k_1}\vartheta_{\delta,y}, \sqrt{k_2}\vartheta_{\delta,y})|\\
  &\geq \frac{\rho}{1+\rho}-\frac{\int_{B_{\rho}(0)}|\vartheta_{\delta,y}|^{2^*}}{\int_{\R^N}|\vartheta_{\delta,y}|^{2^*}dx}-|\xi(\sqrt{k_1}\vartheta_{\delta,y}, \sqrt{k_2}\vartheta_{\delta,y})|\\
  &\geq 1-\frac{1}{1+\rho}-\frac{\varepsilon}{3}-\frac{\varepsilon}{3}  \geq 1-\varepsilon.
\end{align*}
Finally, based on the arguments above, we can easily get
$$\lim_{\delta \to  \infty} \gamma(\sqrt{k_1}\vartheta_{\delta,y}, \sqrt{k_2}\vartheta_{\delta,y})=1$$
uniformly for $y\in B_{r}(0)$, and the proof of $(b)$ is completed.

$(c)$ For any $x\in \R^N$ with $\langle x\mid y\rangle_{\R^N}>0$, there holds
$|-x-y|>|x-y|$, which together with \eqref{e-4.15}$(ii)$ implies   that $\vartheta_{\delta,y}(x)\geq \vartheta_{\delta,y}(-x)$ for any $x\in \R^N$ with $\langle x \mid y\rangle_{\R^N}>0$. Moreover,  meas $\{x\in \R^N:\langle x \mid y\rangle_{\R^N}>0, \vartheta_{\delta,y}(x)>\vartheta_{\delta,y}(-x) \}>0$. Then by computation we get for $\delta>0$ and any $y\in \R^{N}\setminus \{0\}$, 
\begin{align*}
\int_{\R^N}\frac{\langle x \mid y\rangle_{\R^N}}{1+|x|} |\vartheta_{\delta,y}|^{2^*}dx
  & =\int_{\{x\in \R^N:\langle x , y\rangle_{\R^N}>0\}}\frac{\langle x \mid y\rangle_{\R^N}}{1+|x|} |\vartheta_{\delta,y}|^{2^*}dx+
  \int_{\{x\in \R^N:\langle x \mid y\rangle_{\R^N}<0\}}\frac{\langle x \mid y\rangle_{\R^N}}{1+|x|} |\vartheta_{\delta,y}|^{2^*}dx \\
  & =\int_{\{x\in \R^N:\langle x ,y \rangle_{\R^N}>0\}} \frac{\langle x \mid y\rangle_{\R^N}}{1+|x|} \big( |\vartheta_{\delta,y}(x)|^{2^*}- |\vartheta_{\delta,y}(-x)|^{2^*} \big) dx>0,
\end{align*}
which implies that 
\begin{align*}
\langle \xi( \sqrt{k_1}\vartheta_{\delta,y}, \sqrt{k_2}\vartheta_{\delta,y})\mid y  \rangle_{\R^N}
&= \frac{\int_{\R^N} \frac{\langle x \mid y \rangle_{\R^N}}{1+|x|} (\mu_1 k_{1}^{\frac{2^*}{2}}+2\beta k_{1}^{\frac{2^*}{4}}k_{2}^{\frac{2^*}{4}}+\mu_2 k_{2}^{\frac{2^*}{2}}) |\vartheta_{\delta,y}|^{2^*}dx}{\int_{\R^N} (\mu_1 k_{1}^{\frac{2^*}{2}}+2\beta k_{1}^{\frac{2^*}{4}}k_{2}^{\frac{2^*}{4}}+\mu_2 k_{2}^{\frac{2^*}{2}}) |\vartheta_{\delta,y}|^{2^*}dx}>0
 \end{align*}
 for  any $ y \in \R^N \setminus \{0\}$ and $\delta>0$.

\end{proof}

\section{Proof of main results}\label{s5}

\qquad We first introduce the notation $I_{0}$ to define the functional $I$ with $\lambda_1=\lambda_2 =0$, that is
\begin{align*}
  I_0(u,v)&=\frac{1}{2}\int_{\R^N}\Big(|\nabla u|^{2}+|\nabla v|^{2}+ V_{1}(x)u^2+V_2(x)v^2 \Big)dx\\
  &\quad -\frac{1}{4}\int_{\R^N}\int_{\R^N}\frac{\mu_1|u^+(x)|^2|u^{+}(y)|^2+2\beta |u^+(x)|^2|v^+(y)|^2+\mu_2 |v^+(x)|^2|v^{+}(y)|^2 }{|x-y|^4}dxdy.
\end{align*}
According to the definition above, we set  Nehari manifold
\begin{equation*}
 \mathcal{ N}_{0}:=\{(u,v)\in H_{0} \setminus \{(0,0)\} : \langle I'_{0}(u,v),(u,v)    \rangle=0  \}.
\end{equation*}
%For each $\vartheta_{\delta,y}$, we  denote by
%\begin{equation*}
%  \widehat{\vartheta}_{\delta,y}:=(t_{\delta,y,0}\sqrt{k_1}\vartheta_{\delta,y}, t_{\delta,y,0}\sqrt{k_2}\vartheta_{\delta,y})\in \mathcal{N}_{0}, \qquad
%  \widetilde{\vartheta}_{\delta,y}:=(t_{\delta,y}\sqrt{k_1}\vartheta_{\delta,y}, t_{\delta,y}\sqrt{k_2}\vartheta_{\delta,y})\in \mathcal{N}
%\end{equation*}
%the projections of $\vartheta_{\delta,y}$ respectively on $\mathcal{N}_{0}$ and $\mathcal{N}$,
%and set
%\begin{equation*}
%  \gamma \circ\widehat{\vartheta}_{\delta,y}:=\gamma(t_{\delta,y,0}\sqrt{k_1}\vartheta_{\delta,y}, t_{\delta,y,0}\sqrt{k_2}\vartheta_{\delta,y}), \qquad
%  \gamma \circ \widetilde{\vartheta}_{\delta,y}:=\gamma(t_{\delta,y}\sqrt{k_1}\vartheta_{\delta,y}, t_{\delta,z}\sqrt{k_2}\vartheta_{\delta,y}).
%\end{equation*}
For each $\vartheta_{\delta,y}$, we  set
$$
\overline{\vartheta}_{\delta,y}:=( \sqrt{k_1}\vartheta_{\delta,y},  \sqrt{k_2}\vartheta_{\delta,y})
$$
and the projections
\begin{equation}
\label{pr}
 \widetilde{\vartheta}_{\delta,y}:= t_{\delta,y}\overline{\vartheta}_{\delta,y} \in \mathcal{N},
  \qquad
 \widehat{\vartheta}_{\delta,y}:= t_{\delta,y,0}\overline{\vartheta}_{\delta,y} \in \mathcal{N}_{0},
   \qquad
   t_{\delta,y}, \,\, t_{\delta,y,0}>0.
\end{equation}
%(see \eqref{N0}).
Let us define
\begin{equation*}
  \xi\circ\overline {\vartheta}_{\delta,y}:=\xi( \sqrt{k_1}\vartheta_{\delta,y},  \sqrt{k_2}\vartheta_{\delta,y}), \qquad
  \gamma \circ \overline{\vartheta}_{\delta,y}:=\gamma( \sqrt{k_1}\vartheta_{\delta,y},  \sqrt{k_2}\vartheta_{\delta,y}),
\end{equation*}
and in analogous way
$\xi\circ\widetilde  {\vartheta}_{\delta,y}$,
$\xi\circ \widehat{\vartheta}_{\delta,y}$,
 $\gamma\circ\widetilde  {\vartheta}_{\delta,y}$,
$\gamma\circ \widehat{\vartheta}_{\delta,y}$.

Observe that, by \eqref{e-4.1},
\begin{equation}\label{inv}
\xi\circ\overline {\vartheta}_{\delta,y}=\xi\circ\widetilde  {\vartheta}_{\delta,y}=\xi\circ \widehat{\vartheta}_{\delta,y},
\qquad
\gamma\circ\overline {\vartheta}_{\delta,y}=\gamma\circ\widetilde  {\vartheta}_{\delta,y}=\gamma\circ \widehat{\vartheta}_{\delta,y}
\end{equation}

\begin{lem}\label{lm5.1}
Let $t_{\delta,y,0}$ be as in \eqref{pr}, then the  following relations hold: \\
$(a)$ \,$\displaystyle \lim_{\delta \to 0} \sup\{ |t_{\delta,y,0}-1|:y\in \R^N \}=0$; \\
$(b)$ \,$\displaystyle \lim_{\delta \to \infty} \sup\{ |t_{\delta,y,0}-1|:y\in \R^N \}=0$; \\
$(c)$ \,$\displaystyle \lim_{r \to  \infty} \sup\{ |t_{\delta,y,0}-1|:y\in \R^N, |y|=r,\delta>0 \}=0$.
\end{lem}
\begin{proof}
By \eqref{e-4.15} and the definition of $\vartheta_{\delta,y}$, we get $(\sqrt{k_1}\vartheta_{\delta,y}, \sqrt{k_2}\vartheta_{\delta,y})\in \mathcal{N}_{\infty}$. Then by calculation  we have
\begin{align*}
1&=\frac{\int_{\R^N}\big(|\nabla (\sqrt{k_1}\vartheta_{\delta,y})|^{2}+|\nabla (\sqrt{k_2}\vartheta_{\delta,y})|^{2}\big)dx}{\int_{\R^N}(\mu_1k_1^2+\mu_2k_2^2+2\beta k_1k_2)(|x|^{-4}*|\vartheta_{\delta,y}|^{2}) |\vartheta_{\delta,y}|^{2} dx}\\
%&=\frac{t_{\delta,y,0}^{2}[\mu_1(\sqrt{k_1}\vartheta_{\delta,y})^4+2\beta (\sqrt{k_1}\vartheta_{\delta,y})^2(\sqrt{k_2}\vartheta_{\delta,y})^2+\mu_2 (\sqrt{k_2}\vartheta_{\delta,y})^4]-\int_{\R^4}\Big( V_{1}(x)(\sqrt{k_1}\vartheta_{\delta,y})^2+V_2(x)(\sqrt{k_2}\vartheta_{\delta,y})^2 \Big)dx }{\int_{\R^4}\big(\mu_1(\sqrt{k_1}\vartheta_{\delta,y})^4+2\beta (\sqrt{k_1}\vartheta_{\delta,y})^2(\sqrt{k_2}\vartheta_{\delta,y})^2+\mu_2 (\sqrt{k_2}\vartheta_{\delta,y})^4  \big)dx}\\
&=\frac{t_{\delta,y,0}^{2}\int_{\R^N}(\mu_1k_1^2+\mu_2k_2^2+2\beta k_1k_2)(|x|^{-4}*|\vartheta_{\delta,y}|^{2}) |\vartheta_{\delta,y}|^{2} dx-\int_{\R^N}\Big( k_1V_{1}(x)+k_2V_2(x)\Big)|\vartheta_{\delta,y}|^2 dx }{\int_{\R^N}(\mu_1k_1^2+\mu_2k_2^2+2\beta k_1k_2)(|x|^{-4}*|\vartheta_{\delta,y}|^{2}) |\vartheta_{\delta,y}|^{2} dx  }.
\end{align*}
Obviously, the properties $(a)$-$(c)$ hold true by Lemma \ref{lm4.4}.
\end{proof}

\begin{lem}\label{lm5.2}
If $(A_1)$ and $(A_2)$ hold, then there exist constants $\bar{r}>0$ and $0<\delta_1<\frac{1}{2}<\delta_2$ such that \\
\begin{equation}\label{e-5.1}
\gamma \circ \widehat{\vartheta}_{\delta_1,y}<\frac{1}{2}, \ \forall y\in \R^N,
\end{equation}
\begin{equation}\label{e-5.1-b}
\gamma \circ \widehat{\vartheta}_{\delta_2,y}>\frac{1}{2}, \ \ \forall y\in \R^N, \ \ |y|<\bar{r}
\end{equation}
and
\begin{equation}\label{e-5.2}
  \sup\{I_{0} \circ\widehat{\vartheta}_{\delta,y} :(\delta,y) \in \partial \mathscr{H} \}<\bar{c},
\end{equation}
where $\bar{c}$ is defined in \eqref{e-4.14} and
\begin{equation*}
\mathscr{H}:=\{(\delta,y)\in \R^{+}\times \R^{N} : \delta\in [\delta_1,\delta_2], \,\, |y|<\bar{r}\}.
\end{equation*}
\end{lem}

\begin{proof}
For every $\delta>0$ and $y\in \R^N$, by calculation we have
\begin{align*}
  I_0\circ \widehat{\vartheta}_{\delta,y}&=I_{0}(t_{\delta,y,0}\sqrt{k_1}\vartheta_{\delta,y},t_{\delta,y,0}\sqrt{k_2}\vartheta_{\delta,y}) \\
  &=(k_1+k_2)t_{\delta,y,0}^{2} \int_{\R^N}|\nabla \vartheta_{\delta,y}|^{2}dx+ t_{\delta,y,0}^{2} \int_{\R^N}\big( k_1V_{1}(x)+k_2V_2(x) \big)|\vartheta_{\delta,y}|^{2}dx \\
  &  \quad -t_{\delta,y,0}^{4} \int_{\R^N}\big(\mu_1k_1^{2}+2\beta k_1k_2+\mu_2k_2^{2}\big)(|x|^{-4}*|\vartheta_{\delta,y}|^{2})|\vartheta_{\delta,y}|^{2} dx\\
  &=(k_1+k_2)t_{\delta,y,0}^{2} \int_{\R^N}|\nabla \vartheta_{\delta,y}|^{2}dx+   t_{\delta,y,0}^{2}  \int_{\R^N}\big( k_1V_{1}(x)+k_2V_2(x) \big)|\vartheta_{\delta,y}|^{2}dx \\
 &  \quad -(k_1+k_2)t_{\delta,y,0}^{4} \int_{\R^N}(|x|^{-4}*|\vartheta_{\delta,y}|^{2})|\vartheta_{\delta,y}|^{2} dx.
\end{align*}

 By  \eqref{inv}, Lemma \ref{lm4.4}$(a)$, Lemma \ref{lm4.5}$(a)$, Lemma \ref{lm5.1}$(a)$ and \eqref{e-4.15}, we prove that there exists $\delta_1\in (0,\frac{1}{2})$ such that $\gamma \circ \widehat{\vartheta}_{\delta_1,y}<\frac{1}{2}$ and $I_{0}\circ \widehat{\vartheta}_{\delta_1,y}<\bar{c}$ for every $y\in \R^N$. 
 
 Moreover, by Lemma \ref{lm4.4}$(c)$, Lemma \ref{lm5.1}$(c)$ and \eqref{e-4.15} again, we can choose $\bar{r}>0$ such that, if $|y|=\bar{r}$, then $I_{0}\circ \widehat{\vartheta}_{\delta,y}<\bar{c}$ is satisfied for every $\delta>0$. 
 
 In the end, once $\bar{r}$ is fixed, then the existence of $\delta_2>\frac{1}{2}$ such that $\gamma \circ \widehat{\vartheta}_{\delta_2,y}>\frac{1}{2}$  and   $I_{0}\circ \widehat{\vartheta}_{\delta_2,y}<\bar{c}$ for all $y\in \R^N$ with $|y|\leq \bar{r}$,  follows from \eqref{inv}, Lemma \ref{lm4.4}$(b)$, Lemma \ref{lm4.5}$(b)$, Lemma \ref{lm5.1}$(b)$ and \eqref{e-4.15}.

\end{proof}

\begin{lem}\label{lm5.3}
 Let $\delta_1,\delta_2,\bar{r}$ and $\mathscr{H}$ be defined as Lemma \ref{lm5.2}. Then there exist $(\widetilde{\delta},\widetilde{y})\in \partial \mathscr{H}$ and
 $(\bar{\delta},\bar{y})\in  \mathring{\mathscr{H}}$ such that
 \begin{equation}\label{e-5.3}
\xi(\sqrt{k_1}\vartheta_{\widetilde{\delta},\widetilde{y}}, \sqrt{k_2}\vartheta_{\widetilde{\delta},\widetilde{y}}  )=0, \qquad \gamma(\sqrt{k_1}\vartheta_{\widetilde{\delta},\widetilde{y}}, \sqrt{k_2}\vartheta_{\widetilde{\delta},\widetilde{y}}  )> \frac{1}{2}
 \end{equation}
and
 \begin{equation}\label{e-5.4}
\xi(\sqrt{k_1}\vartheta_{\bar{\delta},\bar{y}}, \sqrt{k_2}\vartheta_{\bar{\delta},\bar{y}}  )=0, \qquad \gamma(\sqrt{k_1}\vartheta_{\bar{\delta},\bar{y}}, \sqrt{k_2}\vartheta_{\bar{\delta},\bar{y}}  )=\frac{1}{2}.
 \end{equation}
\end{lem}

\begin{proof}
  By the symmetric property of $\vartheta_{\delta_2,0}$, we have $ \xi(\sqrt{k_1}\vartheta_{\delta_2,0}, \sqrt{k_2}\vartheta_{\delta_2,0}  )=0$.   Moreover, by  Lemma \ref{lm5.2} together with \eqref{inv}, we get $\gamma(\sqrt{k_1}\vartheta_{\delta_2,0}, \sqrt{k_2}\vartheta_{\delta_2,0}  )>\frac{1}{2}$. Hence, \eqref{e-5.3} is trivial  if we choose $(\widetilde{\delta},\widetilde{y})=(\delta_2,0)$.

For any $(\delta,y)\in \mathscr{H}$, we set
\begin{equation*}
  \Theta(\delta,y)=\big(\gamma(\sqrt{k_1}\vartheta_{\delta,y}, \sqrt{k_2}\vartheta_{\delta,y}),  \,\,  \xi(\sqrt{k_1}\vartheta_{\delta,y}, \sqrt{k_2}\vartheta_{\delta,y}  )\big)
\end{equation*}
and denote  the homotopy map $\mathcal{T}: [0,1] \times \partial \mathscr{H} \to \R \times \R^N$ as following
\begin{equation}\label{e-5.5}
  \mathcal{T}(\delta,y,s)=(1-s)(\delta,y)+s \Theta(\delta,y).
\end{equation}
To prove \eqref{e-5.4}, it is enough to prove that
\begin{equation}\label{e-5.6}
deg(\Theta, \mathring{\mathscr{H}}, (\frac{1}{2},0))=1.
\end{equation}
We remark that $deg(Id, \mathring{\mathscr{H}}, (\frac{1}{2},0))=1$,  if we can prove that for each $(\delta,y)\in \partial \mathscr{H}$ and $s\in [0,1]$, $\mathcal{T}(\delta,y,s)\neq (\frac{1}{2},0)$,  then  \eqref{e-5.6} follows easily from   the  topological degree homotopy invariance. 

Next,  we are going to verify that for any $(\delta,y)\in \partial \mathscr{H}$ and $ s\in [0,1]$,
\begin{equation}
  \big( (1-s)\delta+s\gamma(\sqrt{k_1}\vartheta_{\delta,y}, \sqrt{k_2}\vartheta_{\delta,y}), (1-s)y+s\xi(\sqrt{k_1}\vartheta_{\delta,y}, \sqrt{k_2}\vartheta_{\delta,y})   \big)\neq (\frac{1}{2},0).
\end{equation}
 Let $\partial \mathscr{H}=\mathscr{H}_{1}\cup \mathscr{H}_{2} \cup \mathscr{H}_{3}$ with
\begin{align*}
  \mathscr{H}_{1}&=\{(\delta,y)\in \partial \mathscr{H} :|y|\leq \bar{r}, \ \delta=\delta_1  \}, \\
  \mathscr{H}_{2}&=\{(\delta,y)\in \partial \mathscr{H} :|y|\leq \bar{r}, \ \delta=\delta_2  \}, \\
  \mathscr{H}_{3}&=\{(\delta,y)\in \partial \mathscr{H} :|y|=\bar{r}, \    \delta\in [\delta_1, \delta_2]  \}.
\end{align*}
If $(\delta,y)\in \mathscr{H}_{1}$, then by \eqref{inv}-\eqref{e-5.1}, we get
\begin{equation*}
(1-s)\delta_1+s\gamma(\sqrt{k_1}\vartheta_{\delta_1,y}, \sqrt{k_2}\vartheta_{\delta_1,y} )<\frac{1}{2}(1-s)+\frac{s}{2}<\frac{1}{2}.
\end{equation*}
If $(\delta,y)\in \mathscr{H}_{2}$, then by \eqref{inv} and \eqref{e-5.1-b},
\begin{equation*}
(1-s)\delta_2+s\gamma(\sqrt{k_1}\vartheta_{\delta_2,y}, \sqrt{k_2}\vartheta_{\delta_2,y} )>\frac{1}{2}(1-s)+\frac{s}{2}>\frac{1}{2}.
\end{equation*}
If $(\delta,z)\in \mathscr{H}_{3}$, then from Lemma \ref{lm4.5}$(c)$ we observe that
\begin{equation*}
\langle (1-s)y+s \xi(\sqrt{k_1}\vartheta_{\delta,y}, \sqrt{k_2}\vartheta_{\delta,y} ) \mid y   \rangle_{\R^N}=(1-s)|y|^2+s\langle \xi(\sqrt{k_1}\vartheta_{\delta,y}, \sqrt{k_2}\vartheta_{\delta,y} ) \mid y\rangle_{\R^N} >0,
\end{equation*}
which implies that
$$
(1-s)y+s\xi(\sqrt{k_1}\vartheta_{\delta,y}, \sqrt{k_2}\vartheta_{\delta,y} )\neq 0.
$$
\end{proof}

\begin{lem}\label{lm5.4}
Let $\delta_1$, $\delta_2$, $\bar{r}$ and $\mathscr{H}$ be defined as Lemma \ref{lm5.2}.  Assume that $(A_3)$ holds, then
\begin{equation}\label{e-5.8}
 \mathcal{K}=\sup\{I_{0}\circ \widehat{\vartheta}_{\delta,y}: (\delta,y ) \in \mathscr{H}\}< \min\{\frac{\mathcal{S}_{HL}^{2}}{4\mu_1}, \frac{\mathcal{S}_{HL}^{2}}{4\mu_2},2c_{\infty}\}.
\end{equation}
\end{lem}

\begin{proof}
%Recalling that $(\sqrt{k_1}\vartheta_{\delta,y}, \sqrt{k_2}\vartheta_{\delta,y} )\in \mathcal{N}_{\infty}$, then it follows by Lemma \ref{lm2.2} that $t_{\delta,y,0}\geq 1$.
For each $(\delta,y)\in \mathscr{H}$,   by the definition of $\widehat{\vartheta}_{\delta,y}$ and \eqref{e-4.15}  we get
\begin{align}\label{e-5.9}
  I_{0}\circ \widehat{\vartheta}_{\delta,y}&=\frac{t_{\delta,y,0}^{2}}{4}(k_1+k_2)\int_{\R^N} |\nabla\vartheta_{\delta,y}|^{2}dx+\frac{t_{\delta,y,0}^{2}}{4}\int_{\R^N}\big(k_1V_1(x)|\vartheta_{\delta,y}|^2+k_2V_{2}(x)|\vartheta_{\delta,y}|^2\big)dx \nonumber \\
  & \leq \frac{t_{\delta,y,0}^{2}}{4}(k_1+k_2)\|\vartheta_{\delta,y}\|_{D^{1,2}}^{2}+\frac{t_{\delta,y,0}^{2}k_{1}}{4}\|V_1\|_{L^{\frac{N}{2}}}\|\vartheta_{\delta,y}\|_{L^{2^{*}}}^{2}
  +\frac{t_{\delta,y,0}^{2}k_{2}}{4}\|V_2\|_{L^{\frac{N}{2}}}\|\vartheta_{\delta,y}\|_{L^{2^*}}^{2} \nonumber \\
  & \leq \frac{t_{\delta,y,0}^{2}}{4}(k_1+k_2)\|\vartheta_{\delta,y}\|_{D^{1,2}}^{2}+\frac{t_{\delta,y,0}^{2}k_1}{4\mathcal{S}}\|V_1\|_{L^{\frac{N}{2}}}\|\vartheta_{\delta,y}\|_{D^{1,2}}^{2}
  +\frac{t_{\delta,y,0}^{2}k_{2}}{4\mathcal{S}}\|V_2\|_{L^{\frac{N}{2}}}\|\vartheta_{\delta,y}\|_{D^{1,2}}^{2} \nonumber \\
  & = t_{\delta,y,0}^{2}\Big(1+\frac{k_1\|V_1\|_{L^{\frac{N}{2}}}}{(k_1+k_2)\mathcal{S}}+ \frac{k_2\|V_2\|_{L^{\frac{N}{2}}}}{(k_1+k_2)\mathcal{S}}     \Big)\Sigma.
\end{align}
Notice that $(t_{\delta,y,0}\sqrt{k_1}\vartheta_{\delta,y},  t_{\delta,y,0}\sqrt{k_2}\vartheta_{\delta,y} )\in \mathcal{N}_{0}$, then
\begin{align*}
 &  \quad t_{\delta,y,0}^{2}(k_1+k_2)\int_{\R^N}|\nabla \vartheta_{\delta,y}|^{2}dx+ t_{\delta,y,0}^{2} \int_{\R^N}\big(k_1V_1(x)+k_2V_2(x)\big)|\vartheta_{\delta,y}|^{2}dx \\
 &=t_{\delta,y,0}^{4}\int_{\R^N}(k_1^2\mu_1 +2\beta k_1 k_2  +k_2^{2}\mu_2^{2})(|x|^{-4}*|\vartheta_{\delta,y}|^{2})|\vartheta_{\delta,y}|^{2}dx\\
&=t_{\delta,y,0}^{4}(k_1+k_2)\int_{\R^N}(|x|^{-4}*|\vartheta_{\delta,y}|^{2})|\vartheta_{\delta,y}|^{2}dx,
\end{align*}
that is,
\begin{equation*}
t_{\delta,y,0}^{2}= \frac{(k_1+k_2)\int_{\R^N}|\nabla \vartheta_{\delta,y}|^{2}dx+ \int_{\R^N}\big(k_1V_1(x)+k_2V_2(x)\big)|\vartheta_{\delta,y}|^{2}dx}{(k_1+k_2)\int_{\R^N}(|x|^{-4}*|\vartheta_{\delta,y}|^{2})|\vartheta_{\delta,y}|^{2}dx}.
\end{equation*}
%Furthermore, notice that
%$$\Sigma>c_{\infty}=\frac{k_1+k_2}{4}\mathcal{S}_{HL}^{2},$$
Since $(\sqrt{k_1}\vartheta_{\delta,y}, \sqrt{k_2}\vartheta_{\delta,y} )\in \mathcal{N}_{\infty}$,  %\eqref{e-4.13}-\eqref{e-4.15},
then
\begin{align}\label{e-5.10}
  t_{\delta,y,0}^{2}
  &= \frac{(k_1+k_2)\int_{\R^N}|\nabla \vartheta_{\delta,y}|^{2}dx+ \int_{\R^N}\big(k_1V_1(x)+k_2V_2(x)\big)|\vartheta_{\delta,y}|^{2}dx}{(k_1+k_2)
  \int_{\R^N}(|x|^{-4}*|\vartheta_{\delta,y}|^{2})|\vartheta_{\delta,y}|^{2}dx} \nonumber \\
  & \leq 1+ \frac{(k_1\|V_1\|_{L^{\frac{N}{2}}}+k_2\|V_2\|_{L^{\frac{N}{2}}})\|\vartheta\|_{L^{2^*}}^2}{(k_1+k_2)\|\vartheta\|_{D^{1,2}}^{2}} \nonumber \\
  & \leq 1+\frac{k_1\|V_1\|_{L^{\frac{N}{2}}}}{(k_1+k_2)\mathcal{S}}+ \frac{k_2\|V_2\|_{L^{\frac{N}{2}}}}{(k_1+k_2)\mathcal{S}}.
\end{align}
Let us insert \eqref{e-5.10} into \eqref{e-5.9}, then by  \eqref{guo-e-4.15}  and \eqref{e-4.15} we get
\begin{align*}
I_{0}\circ \widehat{\vartheta}_{\delta,y}& \leq \Big( 1+\frac{k_1\|V_1\|_{L^{\frac{N}{2}}}}{(k_1+k_2)\mathcal{S}}+ \frac{k_2\|V_2\|_{L^{\frac{N}{2}}}}{(k_1+k_2)\mathcal{S}}    \Big)^{2}\Sigma \\
& < \Big(1+\frac{k_1\|V_1\|_{L^{\frac{N}{2}}}}{(k_1+k_2)\mathcal{S}}+ \frac{k_2\|V_2\|_{L^{\frac{N}{2}}}}{(k_1+k_2)\mathcal{S}}     \Big)^{2}\bar{c} \\
&=  \Big(1+\frac{\beta-\mu_2}{(2\beta-\mu_1-\mu_2)\mathcal{S}}\|V_1\|_{L^{\frac{N}{2}}}+  \frac{\beta-\mu_1}{(2\beta-\mu_1-\mu_2)\mathcal{S}}\|V_2\|_{L^{\frac{N}{2}}}    \Big)^{2}\bar{c}  \\
&=  \frac{\bar{c}}{\mathcal{S}^2}  \Big(\mathcal{S}+\frac{\beta-\mu_2}{(2\beta-\mu_1-\mu_2)}\|V_1\|_{L^{\frac{N}{2}}}+  \frac{\beta-\mu_1}{(2\beta-\mu_1-\mu_2)}\|V_2\|_{L^{\frac{N}{2}}}    \Big)^{2}\\
& \leq  \min \Big\{\frac{\mathcal{S}_{HL}^2}{4\mu_1}, \frac{\mathcal{S}_{HL}^2}{4\mu_2}, 2c_{\infty}  \Big\}.
\end{align*}
\end{proof}

%Next, we are going to consider  system \eqref{S-4} with $\max\{\lambda_1, \lambda_2\}>0$.  Recalling that, at the beginning of this section, the notation  $\widetilde{\vartheta}_{\delta,y}:=(t_{\delta,y}\sqrt{k_1}\vartheta_{\delta,y}, t_{\delta,y}\sqrt{k_2}\vartheta_{\delta,y})\in \mathcal{N}$ is introduced.

\begin{lem}\label{lm5.5}
 Let $\delta_1,\delta_2,\bar{r},\mathscr{H}$ as in Lemma \ref{lm5.2}, then there exists a number $\lambda^*>0$ such that for each $\lambda:=\max\{\lambda_1,\lambda_2\} \in (0,\lambda^*)$, the following relations hold:
 \begin{equation}\label{e-5.12}
   \gamma  \circ \widetilde{\vartheta}_{\delta_1,y} <\frac{1}{2}, \  \ \forall y\in \R^N,
 \end{equation}
 \begin{equation}\label{e-5.12-b}
   \gamma \circ \widetilde{\vartheta}_{\delta_2,y}>\frac{1}{2},
   \ \ \ \ \forall y\in \R^N, \ \ |y|<\bar{r}
 \end{equation}
 and
 \begin{equation}\label{e-5.13}
   \tilde{\mathcal{K}}:=\sup\{I \circ \widetilde{\vartheta}_{\delta,y} :(\delta,y)\in \partial \mathscr{H}  \}<\bar{c}.
 \end{equation}
 Furthermore, if $(A_3)$ holds, then $\lambda^{**}$ can be found such that  for each  $\lambda:=\max\{\lambda_1,\lambda_2\}  \in (0, \lambda^{**})$, in addition to \eqref{e-5.12}-\eqref{e-5.13},
 \begin{equation}\label{e-5.14}
   \tilde{s}:=\sup\{I\circ \widetilde{\vartheta}_{\delta,y}:(\delta,y)\in  \mathscr{H} \}< \min \Big\{\frac{\mathcal{S}_{HL}^2}{4\mu_1}, \frac{\mathcal{S}_{HL}^2}{4\mu_2}, 2c_{\infty}  \Big\}
 \end{equation}
 is also satisfied.
 \end{lem}

 \begin{proof}
 From \eqref{inv}, we know that
  \begin{equation*}
   \gamma \circ \widetilde{\vartheta}_{\delta,y}=\gamma \circ \widehat{\vartheta}_{\delta,y}, \ \ \forall (\delta,y)\in \R^+ \times \R^N.
  \end{equation*}
  Then \eqref{e-5.12} and \eqref{e-5.12-b} can be seen as a direct consequence of  Lemma \ref{lm5.2}. Recalling that $(\sqrt{k_1}\vartheta_{\delta,y},\sqrt{k_2}\vartheta_{\delta,y})\in \mathcal{N}_{\infty}$ and $\widetilde{\vartheta}_{\delta,y}:=(t_{\delta,y}\sqrt{k_1}\vartheta_{\delta,y}, t_{\delta,y}\sqrt{k_2}\vartheta_{\delta,y})\in \mathcal{N}$ ,   then  by computation we get
 \begin{align*}
   1&=\frac{(k_1+k_2)\int_{\R^N}|\nabla \vartheta_{\delta,y}|^{2}dx}{(\mu_1k_1^2+2\beta k_1k_2+ \mu_2 k_{2}^{2})\int_{\R^N}(|x|^{-4}*| \vartheta_{\delta,y}|^{2})| \vartheta_{\delta,y}|^{2}dx} \\
   &= t_{\delta,y}^{2}-\frac{k_1\int_{\R^N}(V_1(x)+\lambda_1 )|\vartheta_{\delta,y}|^{2}dx+k_2\int_{\R^N}(V_2(x)+\lambda_2 )|\vartheta_{\delta,y}|^{2}dx }{(k_1+k_2)\int_{\R^N}(|x|^{-4}*| \vartheta_{\delta,y}|^{2})| \vartheta_{\delta,y}|^{2}dx}
 \end{align*}
  and
  \begin{equation}\label{add-g1}
    \int_{\R^N}\lambda_j|\vartheta_{\delta,y}|^{2}dx=\lambda_j\delta^{2}\int_{B_{1}(0)}|\vartheta|^{2}dx,
  \end{equation}
  which implies that
  \begin{equation}\label{add-g2}
 \lim_{\lambda \to 0}\sup_{(\delta,y)\in \mathscr{H}}|t_{\delta,y}-t_{\delta,y,0}|=0,
  \end{equation}
 where $\lambda:=\max\{\lambda_1, \lambda_2\}$. Thus, if $\lambda$ is suitably small, then  \eqref{e-5.13} and \eqref{e-5.14} follows straightly by   \eqref{e-5.2}, \eqref{e-5.8}, \eqref{add-g1}
 and \eqref{add-g2}.
\end{proof}

With the help of the previous estimates, we can prove Theorem \ref{Th1.2} by deformation arguments that we can use the global compactness result obtained in Section 3. Before the proof, we first introduce a notation for the sublevel sets:
\begin{equation*}
I^{c}:=\{ (u,v) \in \mathcal{N}: I(u,v)\leq c\}.
\end{equation*}

\textbf{Proof of Theorem \ref{Th1.2}} \,\, In this part, we always assume that  $\lambda\in (0,\lambda^{*})$, where $\lambda^{*}$ is stated in Lemma \ref{lm5.5}. To prove Theorem \ref{Th1.2}, we first  prove that a critical level exists in  $(c_{\infty}, \bar{c})$ under assumptions $(A_1)$ and $(A_2)$. If in additional $(A_3)$ holds and $\lambda\in (0, \lambda^{**})$, then another critical level exists in $(\bar{c}, \min\{\frac{\mathcal{S}_{HL}^{2}}{4\mu_1}, \frac{\mathcal{S}_{HL}^{2}}{4\mu_2}, 2c_{\infty} \})$.

{\bf Step 1.} Let us first assume that $(A_1)$ and $(A_2)$ hold. By using   \eqref{e-4.9}, \eqref{e-4.1}, \eqref{e-5.3}, \eqref{e-5.13} and \eqref{e-4.14}, we have
 \begin{equation}\label{e-4.24}
   c_{\infty}< c^{**}\leq  I \circ \widetilde{\vartheta}_{\tilde{\delta},\tilde{y}} \leq
   \tilde{\mathcal{K}}< \bar{c}< c^{*}.
 \end{equation}
 
We  claim that the functional $I|_\mathcal{N}$ has a critical level in  $(c_{\infty}, \bar{c})$.  We argue by contradiction and suppose that it is incorrect. From Corollary \ref{co3.4}, we know that $I$ constrained on $\mathcal{N}$ satisfies the Palais-Smale condition in the energy interval $(c_\infty, \bar{c})$. Then in virtue of the deformation lemma ( \cite[Lemma 2.3]{WM}), we can prove that there exist a  constant $\sigma_1>0$ such that $c^{**}-\sigma_1>c_{\infty}$, $\tilde{\mathcal{K}}+\sigma_1<\bar{c}$ and a continuous function
\begin{equation*}
  \eta:[0,1]\times I^{\tilde{\mathcal{K}}+\sigma_1}\to I^{\tilde{\mathcal{K}}+\sigma_1}
\end{equation*}
such that
\begin{align}\label{e-5.16}
  \eta  \big(0,(u,v)\big)&=(u,v) \ \ \ \ \ \ \ \ \forall (u,v)\in I^{\tilde{\mathcal{K}}+\sigma_1}; \nonumber \\
  \eta   \big(s,(u,v)\big) &=(u,v) \ \ \ \ \ \ \ \ \forall (u,v) \in I^{c^{**}-\sigma_1}, \,\, \forall s\in[0,1]; \nonumber \\
  I\circ \eta  \big(s,(u,v)\big)&\leq I(u,v) \ \ \ \ \ \ \  \forall (u,v)\in I^{\tilde{\mathcal{K}}+\sigma_1}, \,\, \forall s\in [0,1]; \\
  \eta  (1,I^{\tilde{\mathcal{K}}+\sigma_1})&\subset I^{c^{**}-\sigma_1}.
\end{align}
 From \eqref{e-5.13} and (5.22), we  observe that
 \begin{equation}\label{e-5.18}
   (\delta,y)\in \partial \mathscr{H} \Longrightarrow I \circ \widetilde{\vartheta}_{\delta,y} \leq \tilde{\mathcal{K}}\Longrightarrow I \circ \eta  (1,\widetilde{\vartheta}_{\delta,y})\leq c^{**}-\sigma_1.
 \end{equation}
For $s\in [0,1]$ and $(\delta,y)\in \mathscr{H}$, we define
 \begin{equation*}
   \Gamma(\delta,y,s)=
   \begin{cases}
    \mathcal{T}(\delta,y,2s), \ \  &s\in [0,\frac{1}{2}]; \\
    \big( \gamma\circ \eta (2s-1, \widetilde{\vartheta}_{\delta,y}), \,\, \xi\circ \eta(2s-1, \widetilde{\vartheta}_{\delta,y})  \big), \ \ & s\in [\frac{1}{2},1 ],
   \end{cases}
 \end{equation*}
  where $ \mathcal{T}$ is defined as \eqref{e-5.5}. As shown in Lemma \ref{lm5.3}, we have already proved
  \begin{equation}\label{e-5.19}
    \forall s\in [0,\frac{1}{2}], \ \ \forall (\delta,y)\in \partial \mathscr{H}, \ \ \Gamma(\delta,y,s)\neq (\frac{1}{2},0).
  \end{equation}
 Moreover, by \eqref{e-4.14}, \eqref{e-5.13} and \eqref{e-5.16},  we deduce that
  \begin{equation*}
    I \circ \eta  (2s-1,\widetilde{\vartheta}_{\delta,y})\leq I \circ \widetilde{\vartheta}_{\delta,y} \leq \tilde{\mathcal{K}} <\bar{c}< c^*,
    \ \  \forall s\in[\frac{1}{2},1], \,\, (\delta,y)\in \partial\mathscr{H},
  \end{equation*}
which implies that
\begin{equation}\label{e-5.20}
\forall s\in [\frac{1}{2},1], \ \ \forall (\delta,y)\in \partial \mathscr{H}, \ \ \Gamma(\delta,y,s)\neq (\frac{1}{2},0).
\end{equation}
Furthermore, by \eqref{e-5.19}, \eqref{e-5.20} and the continuity of $\Gamma$, we prove that there exists $(\check{\delta}, \check{y}) \in \partial \mathscr{H}$ such that
\begin{equation}\label{e-5.21}
  \xi\circ \eta  (1,\widetilde{\vartheta}_{\check{\delta}, \check{y}})=0, \ \ \ \ \gamma \circ \eta   (1, \widetilde{\vartheta}_{\check{\delta},\check{y}})\geq \frac{1}{2}.
\end{equation}
%Indeed,  let us observe that $\partial\mathcal{H}$ is homotopic to a sphere in $\R^{5}$ and $(\frac{1}{2},0)\in \mathring{\mathcal{H}}$, so the line $(\frac{1}{2}, +\infty)\times \{0\}$ crosses $\partial \mathcal{H}$. Since $\partial\mathcal{H}$ is homotopic to $\Gamma(\partial\mathcal{H},1)$, \eqref{e-5.20} could be false if and only if $\Gamma(\delta,y,1)=(\frac{1}{2},0)$ for some $s\in [0,1]$ and $(\delta,y)\in\partial\mathcal{H}$, and this relation is impossible due to \eqref{e-5.19} and \eqref{e-5.20}.
In view of the definition of $c^{**}$ and \eqref{e-5.21}, we then get
$$
I\circ \eta (1,\widetilde{\vartheta}_{\check{\delta},\check{y}})\geq c^{**}.
$$
 which contradicts with \eqref{e-5.18}. Therefore,  the functional $I|_{\mathcal{N}}$ has at least a  critical point $(u_{\ell}, v_{\ell})\in \mathcal{N}$ such that $I(u_{\ell}, v_{\ell})\in (c_{\infty},\bar{c})$. And  $(u_{\ell}, v_{\ell})$  is also a critical  point of $I$, since $\mathcal{N}$ is a natural constraint. Moreover, $u_\ell\geq 0$ and $v_\ell\geq 0$. Note that $I(u_{\ell}, v_{\ell})<\bar{c}<\min \Big\{\frac{\mathcal{S}_{HL}^2}{4\mu_1}, \frac{\mathcal{S}_{HL}^2}{4\mu_2}\Big\}$, then by Corollary 2.3, we  prove that $u_\ell \not\equiv 0$ and $v_\ell \not\equiv0$.
Since $V_j(x)\in L_{loc}^q(\R^N)$ with $q>\frac{N}{2}$, then $u_\ell, v_\ell$ satisfy
$$
-\Delta u_\ell=a(x)u_\ell, \qquad -\Delta v_\ell=b(x)v_\ell
$$
where
$$
a(x)=-(V_1(x)+\lambda_1)+\mu_1(|x|^{-4}*u_\ell^{2})+\beta (|x|^{-4}*v_\ell^{2})\in L^{\frac{N}{2}}_{loc}(\R^N),
$$
$$
b(x)=-(V_2(x)+\lambda_2)+\mu_2(|x|^{-4}*v_\ell^{2})+\beta (|x|^{-4}*u_\ell^{2})\in L^{\frac{N}{2}}_{loc}(\R^N).
$$
 The Br\'ezis-Kato theorem (see \cite{BK}) implies that $u_\ell,v_\ell\in L_{loc}^{q_1}(q_1)$ for all $1\leq q_1<\infty$. Thus, $u_\ell,v_\ell\in \mathcal{C}^{0,\alpha}$ by classical  regularity results.
Hence, by the Harnack inequality (see \cite{PSbook}),  we conclude $u_{\ell}>0$ and  $v_{\ell}>0$.
Thus, we  succeed in proving that  $(u_{\ell}, v_{\ell})$ is a positive solution of system \eqref{S-4} and also of system \eqref{S}, and concluding the first part proof of Theorem \ref{Th1.2}.

{\bf Step 2.} In the sequel, we assume that  $(A_3)$ holds true also. From  \eqref{e-4.14}, \eqref{e-4.1}, \eqref{e-5.4} and \eqref{e-5.14}, we get
\begin{equation*}
  \bar{c}<c^* \leq I \circ (\widetilde{\vartheta}_{\bar{\delta},\bar{y}})\leq \tilde{s}<\min \Big\{\frac{\mathcal{S}_{HL}^2}{4\mu_1}, \frac{\mathcal{S}_{HL}^2}{4\mu_2}, 2c_{\infty}  \Big\}.
\end{equation*}
 We are ready to show that  $I|_{\mathcal{N}}$ has a critical level in the interval $(\bar{c}, \min \big\{\frac{\mathcal{S}_{HL}^2}{4\mu_1}, \frac{\mathcal{S}_{HL}^2}{4\mu_2}, 2c_{\infty}  \big\})$.   Arguing  as above,  we suppose by contradiction that there does not exist any critical levels in $(\bar{c}, \min \big\{\frac{\mathcal{S}_{HL}^2}{4\mu_1}, \frac{\mathcal{S}_{HL}^2}{4\mu_2}, 2c_{\infty}  \big\})$.  By Corollary \ref{co3.4} again, we know that  $I|_{\mathcal{N}}$ satisfies the Palais-Smale condition in  $(\bar{c}, \min \big\{\frac{\mathcal{S}_{HL}^2}{4\mu_1}, \frac{\mathcal{S}_{HL}^2}{4\mu_2}, 2c_{\infty}  \big\})$. By the deformation lemma again, then we prove that there exist a positive number $\sigma_2$ such that
$$c^*-\sigma_2>\bar{c}, \quad \tilde{s}+\sigma_2< \min \big\{\frac{\mathcal{S}_{HL}^2}{4\mu_1}, \frac{\mathcal{S}_{HL}^2}{4\mu_2}, 2c_{\infty} \big\}$$
 and a continuous function
\begin{equation*}
\hat{\eta}:[0,1]\times I^{\tilde{s}+\sigma_2}\to I^{\tilde{s}+\sigma_2}
\end{equation*}
such that
\begin{align*}
  \hat{\eta}  \big(0,(u,v)\big)&=(u,v) \ \ \ \ \ \ \ \ \forall (u,v)\in I^{\tilde{s}+\sigma_2}; \nonumber \\
  \hat{\eta}   \big(s,(u,v)\big) &=(u,v) \ \ \ \ \ \ \ \ \forall (u,v) \in I^{c^{*}-\sigma_2}, \,\, \forall s\in[0,1]; \nonumber \\
  I\circ \hat{\eta } \big(s,(u,v)\big)&\leq I(u,v) \ \ \ \ \ \ \  \forall (u,v)\in I^{\tilde{s}+\sigma_2}, \,\, \forall s\in [0,1]; \\
  \hat{\eta}  (1,I^{\tilde{s}+\sigma_2})&\subset I^{c^{*}-\sigma_2}.
\end{align*}
With the help of \eqref{e-5.14} and the properties listed above ,  we can easily find that
\begin{equation}\label{Guo-New-2}
 \forall(\delta,y)\in \mathscr{H}\Rightarrow I\circ  \widetilde{\vartheta}_{\delta,y} \leq \tilde{s} \Rightarrow I\circ \hat{\eta}(1, \widetilde{\vartheta}_{\delta,y}) \leq c^*-\sigma_2.
\end{equation}
Let $s\in [0,1]$ and $(\delta,y)\in \mathscr{H}$, we define the map
\begin{equation*}
  \hat{ \Gamma}(\delta,y,s)=
   \begin{cases}
    \mathcal{T}(\delta,y,2s), \ \  &s\in [0,\frac{1}{2}]; \\
    \big( \gamma\circ \hat{\eta} (2s-1, \widetilde{\vartheta}_{\delta,y}), \,\, \xi\circ \hat{\eta}(2s-1, \widetilde{\vartheta}_{\delta,y})  \big), \ \ & s\in [\frac{1}{2},1 ],
   \end{cases}
 \end{equation*}
%\begin{equation}\label{add-5.12}
% \Xi(\sigma,z):=(\gamma\circ \eta \circ \widetilde{\vartheta}_{\delta,z}, \, \, \xi \circ \eta \circ \widetilde{\vartheta}_{\delta,z}) \neq (\frac{1}{2}, 0 ).
%\end{equation}
where $ \mathcal{T}$ is defined in \eqref{e-5.5}.  As proved in Lemma \ref{lm5.3}, we have
\begin{equation*}
    \forall s\in [0,\frac{1}{2}], \ \ \forall (\delta,y)\in \partial \mathscr{H}, \ \ \hat{\Gamma}(\delta,y,s)\neq (\frac{1}{2},0).
  \end{equation*}
On the other hand, since
\begin{equation*}
 \forall (\delta,y )\in \partial \mathscr{H}\Rightarrow  I\circ \widetilde{\vartheta}_{\delta,y} \leq \tilde{\mathcal{K}}<\bar{c}<c^*-\sigma_2,
\end{equation*}
then
$$
 \hat{\Gamma}(\delta,y,s)= \hat{\Gamma}(\delta,y,\frac{1}{2})=\mathcal{T}(\delta,y,1), \ \  \forall s\in [\frac{1}{2},1], \ \forall (\delta,y)\in \partial \mathscr{H}.
$$
So
\begin{equation*}
 \hat{\Gamma}(\delta,y,s)\neq (\frac{1}{2},0), \ \ \forall s\in [\frac{1}{2},1], \ \forall (\delta,y)\in \partial \mathscr{H}.
\end{equation*}
By  the homotopy invariance of topological degree together with \eqref{e-5.6}, we prove that there exist a pair of $(\hat{\sigma}, \hat{y})\in \mathscr{H}$ such that
\begin{equation*}
  \xi\circ \hat\eta  (1,\widetilde{\vartheta}_{\hat{\delta}, \hat{y}})=0, \ \ \ \ \gamma \circ \hat\eta   (1, \widetilde{\vartheta}_{\hat{\delta},\hat{y}})= \frac{1}{2}
\end{equation*}
and
$$
I\circ \hat{\eta} (1,\widetilde{\vartheta}_{\hat{\delta},\hat{y}})\geq c^{*},
$$
which  contradicts with \eqref{Guo-New-2}.  Then the constrained functional $I|_{\mathcal{N}}$ has at least a critical point $(u_h,v_h)$ with  $I(u_h,v_h) \in (\bar{c}, \min\{\frac{\mathcal{S}_{HL}^2}{4\mu_1},\frac{\mathcal{S}_{HL}^2}{4\mu_2}, 2c_{\infty}\})$. Moreover, $(u_h,v_h)$  is a critical point of functional $I$ and $(u_h,v_{h})$ is definite in sign.  Arguing by the same methods in Step 1, we can also prove  that $(u_h,v_h)$ is a positive solution of system \eqref{S-4}, and also of system \eqref{S}.  Finally, we complete the whole proof of Theorem \ref{Th1.2}.

%In the end, by repeating the  same arguments in the first part proof of Theorem \ref{Th1.2}$(ii)$, we can also prove Theorem \ref{Th1.2}$(i)$ under the hypothesis $\lambda<\lambda^*$.

\medskip
\textbf{Acknowledgment}

\medskip
The authors would like to thank  Prof. Riccardo Molle  for fruitful discussion on this topic.
The research  was supported by the National Natural Science Foundation of China (No. 12201474, No. 11901222 and No. 12071169) and and the China Postdoctoral Science Foundation (No. 2021M690039).

\medskip

\medskip


\begin{thebibliography}{99}
%\addcontentsline{toc}{chapter}{Bibliography}

%% Use the widest label as parameter.

%% Reference items may be numbered or have labels of your choice.
%% The author's surname PRECEDES the initial of the first name.
%% The surnames are set in small caps.
%% Add \& before the last author's surname.
%% Book titles and journal names are italicized.
%% In book titles, first letters are capitalized.
%% Only journal volume numbers are boldfaced.
%% The issue number is only given when the issues are paginated separately.

%%%%%%%%%%% To ease editing, use normal size:

%\normalsize \baselineskip=17pt






\bibitem{AFM}
C.O. Alves,  G.M.  Figueiredo and  R. Molle:
 Multiple positive bound state solutions for a critical Choquard equation,
 {\it Discrete Contin. Dyn. Syst.},
 {\bf 41},  4887--4919  (2021).

%\bibitem{AGSY}
%C.O.  Alves,  F. Gao, M.  Squassina and M. Yang,
% Singularly perturbed critical Choquard equations,
% {\it J. Differential Equations},
% {\bf 263},  3943--3988  (2017).


\bibitem{ANY}
C.O. Alves, A. N\'obrega and M. Yang,
 Multi-bump solutions for Choquard equation with deepening potential well,
 {\it Calc. Var. Partial Differential Equations,}
 {\bf 55:3}, (2016).



\bibitem{AC1}
A. Ambrosetti and E. Colorado:
Bound and ground states of coupled nonlinear Schr\"{o}dinger equations,
{\it C. R. Math. Acad. Sci. Paris},
{\bf 342}, 453--458 (2006).


\bibitem{AC2}
A. Ambrosetti and E. Colorado:
Standing waves of some coupled nonlinear Schr\"{o}dinger equations,
{\it J. Lond. Math. Soc.},
{\bf 75}, 67--82 (2007).



\bibitem{AA}
N. Akhmediev and A. Ankiewicz:
Partially coherent solitons on a finite background,
{\it Phys. Rev. Lett.},
{\bf 82},  2661--2664 (1999).









\bibitem{BC}
V. Benci and  G. Cerami:
Existence of positive solutions of the equation $-\Delta u + a(x)u =u^{\frac{N+2}{N-2}}$ in $\R^N$,
{\it J. Funct. Anal.},
{\bf 80}, 90--117 (1990).





\bibitem{BD}
T. Bartsch, N. Dancer and Z.-Q Wang:
A Liouville theorem, a-priori bounds, and bifurcating branches of positive solutions for a nonlinear elliptic system,
{\it Calc. Var. Partial Differential Equations},
{\bf 37}, 345--361 (2010).


\bibitem{BWW}
T. Bartsch, Z.-Q Wang and J. Wei:
Bound states for a coupled Schr\"{o}dinger system,
{\it J. Fixed Point Theory Appl.},
{\bf 2}, 353--367 (2007).


\bibitem{BK}
H. Br\'ezis and T. Kato:
Remarks on the Schr\"odinger operator with singular complex potentials,
{\it J. Math. Pures Appl.},
{\bf 58}, 137--151 (1979).



\bibitem{CP1}
M. Clapp and A. Pistonia:
Existence and phase separation of entire solutions to a pure critical competitive elliptic system,
{\it Calc. Var. Partial Differential Equations},
{\bf 57:23},   (2018).








\bibitem{CM}
G. Cerami and R. Molle,
Multiple positive bound states for critical Schr\"odinger-Poisson systems,
{\it ESAIM Control Optim. Calc. Var.},
{\bf 25:73}, (2019).



\bibitem{CP2}
G. Cerami and D. Passaseo,
Nonminimizing positive solutions for equations with critical exponents in the half-space,
{\it  SIAM J. Math. Anal.},
{\bf 28},  867--885 (1997).






%\bibitem{CLZ}
%Z. Chen, C.-S. Lin, and W. Zou:
%Sign-changing solutions and phase separation for an elliptic system with critical exponent,
%{\it Comm. Partial Differential Equations},
%{\bf 39}, 1827--1859 (2014).









\bibitem{CZ1}
Z. Chen and W. Zou:
An optimal constant for the existence of least energy solutions of a coupled Schr\"{o}dinger system,
{\it Calc. Var. Partial Differential Equations},
{\bf 48},  695--711 (2013).



\bibitem{CZ2}
Z. Chen and W. Zou:
Positive least energy solutions and phase separation for coupled Schr\"{o}dinger equations with critical exponent,
{\it Arch. Ration. Mech. Anal.},
{\bf 205}, 515--551 (2012).



\bibitem{CZ3}
Z. Chen and W. Zou:
Positive least energy solutions and phase separation for coupled Schr\"{o}dinger equations with critical exponent: higher dimensional case,
 {\it Calc. Var. Partial Differential Equations},
 {\bf 52}, 423--467  (2015).





%\bibitem{CV}
%M. Conti, S. Terracini and G. Verzini:
%Nehari's problem and competing species systems,
%{\it Ann. Inst. H. Poincar\'{e} Anal. Non Lin\'{e}aire},
%{\bf 19}, 871--888 (2002).











\bibitem{DL}
D.G. de Figueiredo and O. Lopes:
Solitary waves for some nonlinear Schr\"{o}dinger systems,
{\it Ann. Inst. H. Poincar\'{e} Anal. Non Lin\'{e}aire},
{\bf 25}, 149--161 (2008).


\bibitem{DWW}
E.N. Dancer, J. Wei and T. Weth:
A priori bounds versus multiple existence of positive solutions for a nonlinear Schr\"{o}dinger system,
{\it Ann. Inst. H. Poincar\'{e} Anal. Non Lin\'{e}aire},
{\bf 27}, 953--969 (2010).


%\bibitem{DGY}
%Y. Ding, F.  Gao and M. Yang,
%Semiclassical states for Choquard type equations with critical growth: critical frequency case,
%{\it Nonlinearity},
%{\bf 33},  6695--6728  (2020).



\bibitem{DY}
L. Du and M. Yang:
Uniqueness and nondegeneracy of solutions for a critical nonlocal equation,
{\it Discrete Contin. Dyn. Syst.},
{\bf 39}, 5847--5866 (2019).



\bibitem{EGBB}
B.D. Esry, C.H. Greene, J.P. Jr. Burke and J.L. Bohn:
Hartree-Fock theory for double condensates,
{\it Phys. Rev. Lett.},
{\bf 78}, 3594--3597 (1997).




\bibitem{GLMY}
F. Gao,  H. Liu,  V.  Moroz and  M. Yang:
High energy positive solutions for a coupled Hartree system with Hardy-Littlewood-Sobolev critical exponents.
{\it J. Differential Equations},
{\bf 287}, 329--375 (2021).



\bibitem{GSYZ}

F. Gao, E.D. da Silva, M. Yang and J. Zhou:
Existence of solutions for critical Choquard equations via the concentration compactness method.
{\it Proc. Roy. Soc. Edinb. Sect. A},
{\bf 150}, 921--954 (2020)




\bibitem{GY}
F. Gao and M. Yang:
A strongly indefinite Choquard equation with critical exponent due to the Hardy-Littlewood-Sobolev inequality.
{\it Commun. Contemp. Math.},
{\bf 20,}  1750037   (2018).






\bibitem {GHPS}
L. Guo, T. Hu, S. Peng and W.  Shuai:
Existence and uniqueness of solutions for Choquard equation involving Hardy-Littlewood-Sobolev critical exponent,
{\it Calc. Var. Partial Differential Equations},
{\bf58:128},  (2019).



\bibitem{GL}
L. Guo and Q. Li:
Multiple high energy solutions for fractional Schr\"{o}dinger equation with critical growth,
{\it Calc. Var. Partial Differential Equations},
{\bf 61:15},   (2022).


\bibitem{GLLM}
L. Guo, Q. Li, X. Luo, R. Molle:
Standing waves for two-component elliptic system with critical growth in $\R^{4}$: the attractive case. Preprint
 {\it arXiv:2211.03425}.



\bibitem{L}
E.H. Lieb,
Existence and uniqueness of the minimizing solution of Choquard's nonlinear equation,
{\it Studies in Appl. Math.}
{\bf 57}, 93--105 (1976/77).


\bibitem{LL1}
E.H. Lieb and M. Loss,
Analysis.
Gradute Studies in Mathematics, AMS, Providence, Rhodeisland. (2001).



\bibitem{L1}
P.-L.  Lions,
The Choquard equation and related questions,
{\it Nonlinear Anal.},
{\bf 4},  1063--1072  (1980).





\bibitem{LL}
H. Liu and Z. Liu,
A coupled Schr\"{o}dinger system with critical exponent,
{\it Calc. Var. Partial Differential Equations},
{\bf 59:145}, (2020).



%\bibitem{LLW}
%
%{\color{red}{J. Liu, X. Liu and  Z.-Q. Wang:
%Sign-changing solutions for coupled nonlinear Schr\"{o}dinger equations with critical growth,
%{\it J. Differential Equations},
%{\bf 261(12)}, 7194--7236 (2016).}}






\bibitem{LW1}
T.C. Lin and J. Wei:
Ground state of $N$ coupled nonlinear Schr\"{o}dinger equations in $\R^N$, $N\leq 3$,
{\it Comm. Math. Phys.},
{\bf 255}, 629--653 (2005).


%\bibitem{LW2}
%T.C. Lin and J. Wei:
%Spikes in two coupled nonlinear Schr\"{o}dinger equations,
%{\it Ann. Inst. H. Poincar\'{e} Anal. Non Lin\'{e}aire},
%{\bf 22}, 403--439 (2005).





\bibitem{LW3}
Z. Liu and Z.-Q. Wang:
Multiple bound states of nonlinear Schr\"{o}dinger systems,
{\it Comm. Math. Phys.},
{\bf 282}, 721--731 (2008).



\bibitem{LW4}
Z. Liu and Z.-Q. Wang:
Ground states and bound states of a nonlinear Schr\"{o}dinger system,
{\it Adv. Nonlinear Stud.},
{\bf 10}, 175--193 (2010).


\bibitem{MZ}
L. Ma and L. Zhao,
Classification of positive solitary solutions of the nonlinear Choquard equation,
{\it Arch. Ration. Mech. Anal.},
{\bf 195},  455--467  (2010).


\bibitem{M}
G.P. Menzala,
On regular solutions of a nonlinear equation of Choquard's type,
{\it Proc. Roy. Soc. Edinburgh Sect. A},
{\bf 86}, 291--301 (1980).




\bibitem{MPi}
R. Molle and  A. Pistoia:
Concentration phenomena in weakly coupled elliptic systems with critical growth,
{\it Bull. Braz. Math. Soc.},
{\bf 35}, 395--418 (2004).





\bibitem{MPT}
I.M. Moroz, R. Penrose and P. Tod,
Spherically-symmetric solutions of the Schr\"{o}dinger-Newton equations,
{\it Classical Quantum Gravity},
{\bf 15}, 2733--2742 (1998).


%\bibitem{MPS}
%E. Montefusco, B. Pellacci and M. Squassina:
%Semiclassical states for weakly coupled nonlinear Schr\"{o}dinger systems,
%{\it J. Eur. Math. Soc.},
%{\bf 10}, 41--71 (2008).


\bibitem{MS1}
V. Moroz and J. Van Schaftingen:
Ground states of nonlinear Choquard equations: existence, qualitative properties and decay asymptotics,
{\it J. Funct. Anal.},
{\bf 265},  153--184 (2013).



\bibitem{MS2}
V. Moroz and J. Van Schaftingen:
 Existence of groundstates for a class of nonlinear Choquard equations,
{\it Trans. Amer. Math. Soc.},
{\bf 367}, 6557--6579 (2015).



\bibitem{NT}
B. Noris, H. Tavares, S. Terracini and G. Verzini:
Uniform H\"{o}lder bounds for nonlinear Schr\"{o}dinger systems with strong competition,
{\it Comm. Pure Appl. Math.},
{\bf 63}, 267--302  (2010).



\bibitem{PW}
S. Peng and Z.-Q. Wang:
Segregated and synchronized vector solutions for nonlinear Schr\"{o}dinger systems,
{\it  Arch. Ration. Mech. Anal.},
{\bf 208},  305--339 (2013).





\bibitem{PPW}
S. Peng, Y. Peng and Z.-Q. Wang,
On elliptic systems with Sobolev critical growth,
{\it Calc. Var. Partial Differential Equations},
{\bf 55:142}, (2016).





\bibitem{P}
S. Pekar,
Untersuchung \"uber die Elektronentheorie der Kristalle,
Akademie Verlag, Berlin, (1954).







\bibitem{PS}
A. Pistoia and N. Soave:
On Coron's problem for weakly coupled elliptic systems,
{\it Proc. Lond. Math. Soc.},
{\bf 116}, 33--67 (2018).










%\bibitem{PT}
%A. Pistoia and H. Tavares:
%Spiked solutions for Schr\"{o}dinger systems with Sobolev critical exponent: the cases of competitive and weakly cooperative interactions,
%{\it J. Fixed Point Theory Appl.},
%{\bf 19}, 407--446  (2017).


\bibitem{PSbook}
P. Pucci and J. Serrin,
{\em The maximum principle},
Progress in Nonlinear Differential Equations and their Applications, 73. Birkh\"auser Verlag, Basel, (2007).



\bibitem{SW}
Y. Sato  and Z.-Q. Wang:
On the multiple existence of semi-positive solutions for a nonlinear Schr\"odinger system,
{\it Ann. Inst. H. Poincar\'e Anal. Non Lin\'eaire},
{\bf 30}, 1--22 (2013).


\bibitem{S1}
B. Sirakov:
Least energy solitary waves for a system of nonlinear Schr\"{o}dinger equations in $\R^N$,
{\it Comm. Math. Phys.},
{\bf 271}, 199--221 (2007).


\bibitem{S}
M. Struwe,
A global compactness result for elliptic boundary value problems involving limiting nonlinearities,
{\it Math. Z.},
{\bf 187}, 511--517 (1984).


%
%
%
%
%\bibitem{TV}
%S. Terracini and G. Verzini:
%Multipulse Phase in k-mixtures of Bose-Einstein condenstates,
%{\it Arch. Rational Mech. Anal.},
%{\bf 194}, 717--741 (2009).




\bibitem{T}
E. Timmermans:
Phase seperation of Bose-Einstein condensates,
{\it Phys. Rev. Lett.},
{\bf 81}, 5718--5721 (1998).

\bibitem{WS}
J. Wang and J. Shi:
Standing waves for a coupled nonlinear Hartree equations with nonlocal interaction,
{\it Calc. Var. Partial Differential Equations},
{\bf 56:168},  (2017).

\bibitem{WY}
J. Wang and W. Yang:
Normalized solutions and asymptotical behavior of minimizer for the coupled Hartree equations,
{\it J. Differential Equations},
{\bf 265}, 501--544  (2018).



%\bibitem{WW1}
%J. Wei and T. Weth:
%Nonradial symmetric bound states for s system of two coupled Schr\"{o}dinger equations,
%{\it  Rend. Lincei Mat. Appl.},
%{\bf 18}, 279--293 (2007).


\bibitem{WW2}
J. Wei and T. Weth:
Radial solutions and phase separation in a system of two coupled Schr\"{o}dinger equations,
{\it Arch. Rational Mech. Anal.},
{\bf 190}, 83--106 (2008).



%\bibitem{WW3}
%J. Wei and W. Yao:
%Uniqueness of positive solutions to some coupled nonlinear Schr\"{o}dinger equations,
%{\it Commun. Pure Appl. Anal.},
%{\bf 11}, 1003--1011 (2012).




\bibitem{WM}
M. Willem:
Minimax Theorems.
Birkh\"{a}user, Basel. (1996).



\bibitem{WZ}
Y. Wu and W. Zou:
On a critical Schr\"{o}dinger system in $\R^4$ with steep potential wells,
{\it Nonlinear Anal.},
{\bf 191},  111643  (2020).


\bibitem{YWD}
M. Yang, Y. Wei and Y. Ding:
Existence of semiclassical states for a coupled Schr\"odinger system with potentials and nonlocal nonlinearities,
{\it Z. Angew. Math. Phys.},
{\bf 65}, 41--68 (2014).

\end{thebibliography}
\end{document}